%% file: early_stopped_aggregation.tex
\documentclass[reqno,11pt]{imsart}
\input{_preamble.tex}

\newcommand{\do@not@check@textbox@alter}{}
\usepackage[margin=1in]{geometry}
\begin{document}

\begin{frontmatter}
\title{Early-stopped aggregation: Adaptive inference with computational efficiency}
\runtitle{Early-stopped aggregation}

\begin{aug}
\author[A]{\fnms{Ilsang} \snm{Ohn}}\ead[label=e1]{ilsang.ohn@inha.ac.kr}
\author[B]{\fnms{Shitao} \snm{Fan}} 
\author[A]{\fnms{Jungbin} \snm{Jun}}
\and
\author[B]{\fnms{Lizhen} \snm{Lin}}
\address[A]{Department of Statistics, Inha University \\ 
 }
\address[B]{Department of Mathematics, University of Maryland }%
\end{aug}
\runauthor{Ohn, Fan, Jun, and Lin}

\begin{abstract}
When considering a model selection or, more generally, an aggregation approach for adaptive  statistical inference, it is often necessary to compute estimators  over a wide range of model complexities including unnecessarily large models even when the true data-generating process is relatively simple, due to the lack of prior knowledge. This requirement can lead to substantial computational inefficiency. In this work, we propose a novel framework for efficient model aggregation called the \emph{early-stopped aggregation (ESA)}: instead of computing and aggregating estimators for all candidate models, we compute only a small number of simpler ones using an early-stopping criterion and aggregate only these for final inference. Our framework is versatile and applies to both Bayesian model selection, in particular,  within the variational Bayes framework, and frequentist estimation, including a general penalized estimation setting. 

We investigate adaptive optimal property of the ESA approach across three learning paradigms.  We first show that ESA achieves optimal adaptive contraction rates in the variational Bayes setting under mild conditions. We extend this result to variational empirical Bayes, where prior hyperparameters are chosen in a data-dependent manner. In addition, we apply the ESA approach to frequentist aggregation including both penalization-based and sample-splitting implementations, and establish corresponding theory. As we demonstrate, there is a clear unification between early-stopped Bayes and frequentist penalized aggregation, with a common ``energy” functional comprising a data-fitting term and a complexity-control term that drives both procedures. We further present several applications and numerical studies that highlight the efficiency and strong performance of the proposed approach.
\end{abstract}

\begin{keyword}[class=MSC2020]
\kwd[Primary ]{62C10}
\kwd[; secondary ]{62G20}
\end{keyword}

\begin{keyword}
\kwd{Adaptive inference}
\kwd{Early stopping}
\kwd{Empirical Bayes}
\kwd{Exponentially weighted aggregate}
\kwd{Model aggregation}
\kwd{Penalized estimation}
\kwd{Posterior contraction rates}
\kwd{Variational Bayes}
\end{keyword}
\end{frontmatter}

\section{Introduction}

\subsection{Motivation and background}

Modern statistical learning problems are often dealt with an ordered collection of candidate models, ranging from simple, low-dimensional representations to highly complex and expressive ones. An ideal model should be well aligned with the underlying data-generating process, since increasing model complexity and expressiveness typically reduces approximation error but at the cost of inflated estimation error. Selecting an appropriate model within such a model ``ladder'' is therefore a central challenge, as the complexity of the true data-generating distribution is rarely known in practice. Consequently, model selection procedures must  rely on fully data-dependent approaches in order to achieve \textit{adaptive inference} to the unknown true data-generating distribution.

A vast literature has developed around model selection, addressing this challenge from a variety of perspectives and under different modeling assumptions. Various model selection techniques including information criteria and sample-splitting approaches like cross-validation, has been successfully applied in practice, and their adaptive optimality has been extensively studied. Given the richness of this literature, we do not attempt a comprehensive review here; see, for example, textbooks such as \citet{koltchinskii2011oracle,oneto2020model,zhang2023mathematical}.

Aggregation methods, which combine estimators across multiple candidate models, can be viewed as a natural generalization of model selection. Indeed, model selection corresponds to an extreme form of aggregation in which all weight is assigned to a single model. Theoretical properties, in particular oracle inequalities for statistical risks, of aggregation procedures have been well understood for a broad class of statistical problems; see, for example, \citet{juditsky2000functional,juditsky2008learning,yang2000combining,yang2004aggregating, lecue2007optimal, rigollet2007linear, goldenshluger2009universal}. A common feature of these results is that the candidate estimators are assumed to be deterministic, or random but constructed from a sample independent of the one used for aggregation. Several subsequent works have relaxed this restriction. For instance, \citet{dalalyan2012sharp} studied the exponentially weighted aggregate (EWA) of affine estimators for nonparametric regression, while \citet{alquier2011pac} and \citet{rigollet2011exponential} investigated the EWA of least-squares estimators with different sparsity patterns for nonparametric regression and high-dimensional linear regression, respectively. Despite these advances, classical aggregation procedures typically require that all candidate estimators along the model collection be computed prior to aggregation. This exhaustive computation can be  expensive in large-scale settings.

On the Bayesian side, a hierarchical Bayes approach provides a natural aggregation mechanism by placing a prior distribution over a collection of candidate models. The resulting hierarchical posterior can be interpreted as an aggregate of model-specific posterior distributions, with weights that are exponentially proportional to the corresponding marginal likelihoods. Under suitable prior design, hierarchical Bayes procedures have been shown to achieve adaptive optimal contraction rates in a wide range of applications \citep{belitser2003adaptive,lember2007universal,ghosal2008nonparametric,yang2017bayesian,gao2020general,han2021oracle}. This theory can be extended to generalized Bayes approaches, which replaces the log-likelihood by a general loss function to conduct model-free and robust inference. Recent studies have established  adaptive optimal  contraction rates for such generalized Bayesian procedures under appropriate conditions \citep{atchade2017contraction,atchade2018approach,syring2023gibbs}.

Despite these theoretical advantages, exact posterior computation via Markov chain Monte Carlo (MCMC) becomes infeasible in high-dimensional and/or massive data sets. Variational Bayes is a widely used computational alternative, recasting posterior computation as an optimization problem over a tractable variational family.  Recent studies have shown that variational Bayes can attain desirable statistical guarantees including optimal contraction rates \citep{zhang2020convergence,pati2018statistical,yang2020alpha, cherief2018consistency,ray2021variational, ray2020spike,ning2021spike}. However, constructing a variational approximation to a hierarchical posterior over multiple models or aggregating multiple variational posteriors in a theoretically grounded way, remains challenging.

Recently, \citet{ohn2024adaptive} addressed this issue by proposing a novel \textit{adaptive variational Bayes} framework, in which variational posteriors are computed for all candidate models and aggregated with weights according to their values of the variational free energy. This can be seen as the aggregation approach  built on the variational Bayes estimation. It was shown that this aggregated variational posterior can conduct adaptive optimal inference under fairly general conditions. While adaptive variational Bayes is statistically optimal and computationally more tractable than hierarchical Bayes, it still requires running variational Bayes for every model in the collection, including those that are overly large and contribute little to adaptivity. 

These observations suggest a common limitation of both frequentist and Bayesian aggregation methods: in order to attain statistical adaptivity, they require evaluating all candidate models in the ladder, including those far more complex than necessary. Motivated by this computational bottleneck, we propose a novel model aggregation framework with an early-stopping mechanism, called \textit{early-stopped aggregation (ESA)}. Rather than aggregating ``all'' estimators computed over  candidate models, our ESA approach exploits the ordered structure of the model ladder and may stop at an unnecessarily complex model by observing the local behavior of an aggregation criterion.

An alternative and widely used approach to controlling model complexity is penalization, which augments the empirical risk with an explicit penalty term.
Contrary to aggregation methods, penalization-based approaches are based on a single-stage optimization problem of this penalized empirical risk. This paradigm forms the basis of a broad class of methods, including regularized nonparametric regression and classification methods from classical \citep{wahba1990spline,mammen1997locally, horowitz2007rate,steinwart2007fast} to modern \citep{guntuboyina2020adaptive, sadhanala2024multivariate,ki2024mars} developments, and sparse high-dimensional regression, for which we refer to a textbook \citet{buhlmann2011statistics} for an overview of the relevant methods and papers \citet{bertsimas2020sparse,wang2020bridge} for recent progress.

However, in many statistical applications such as deep learning and clustering methods, controlling model complexity through a fixed penalty term is challenging, particularly when model selection involves architectural choices such as  depth and width of a neural network or the number of clusters. In such settings, the resulting optimization problem is often difficult to solve jointly over both parameters and the complexity of the model. The proposed ESA framework offers a practical and efficient alternative in these situations.  An additional practical advantage is that it can be implemented as a ``wrapper'' of existing learning algorithms, as the ESA procedure can be integrated with a wide range of existing methods without reformulating the underlying optimization process adapted to penalization. 

\subsection{Our contributions}

This paper develops a unified theoretical framework for ESA along
ordered sequence of candidate models and establishes its statistical and computational properties across Bayesian and frequentist paradigms. Our contributions are summarized  below.

\paragraph{A general early-stopped aggregation framework} We formalize early stopping as a ``local'' model selection method, which detours exhaustive ``global'' search over all candidate models. Within the proposed ESA framework, we track local changes of the criterion and terminates once additional increases in model complexity no longer yield improvement. After stopping the model search, we aggregate the estimators computed along the search path, such as variational posteriors or empirical risk minimizers, with exponential weights. We show that the resulting procedure nearly optimizes the bias-variance trade-off without any knowledge of the true data-generating process, i.e., achieves \textit{adaptive} optimality, in a manner that is independent of a specific inferential paradigm. 

\paragraph{Theoretical results for three learning methods}  
Although the ESA provides a general approach to  adaptive inference, we  primarily illustrate the ESA framework and its theoretical results with variational Bayes, in order to present the ESA principle in a concrete and transparent form. Later, we show that the same theoretical analysis can be extended to empirical Bayes and frequentist aggregation.

\begin{enumerate}
    \item \textit{Variational Bayes.} We develop the ESA theory within the generalized variational Bayes framework. We provide theoretical guarantees showing that the proposed ESA procedure retains the adaptive optimality  of hierarchical Bayes and adaptive variational Bayes \citep{ohn2024adaptive}. First, we establish a \textit{not-too-early-stopped} property, ensuring that the stopping index is larger than a ``near-optimal'' model index with high probability. Then we show that the ESA variational posterior achieves near-optimal contraction rates of the excess risk, which is adaptive to the true data-generating distribution. Third, we show that the early-stopped procedure terminates \textit{not-too-late}, theoretically supporting its computational efficiency.

    \item \textit{Variational empirical Bayes.} We extend the ESA approach for variational Bayes under ``fixed'' prior distributions to empirical Bayes settings,  where the hyperparameters are optimized in a data-dependent manner. This extension requires new deviation arguments that control the additional variability induced by hyperparameter optimization, which we achieve through a Rényi-type entropy condition on the family of candidate priors. This theoretical technique, to the best of our knowledge, is new in the related literature. Building on this argument, we show that the not-too-early-stopped and oracle contraction properties are preserved under empirical Bayes selection.
    
    \item \textit{Frequentist estimation.} Beyond the Bayesian setting, we show that a closely related early stopping phenomenon arises in frequentist exponential aggregation.  We identify a technical similarity in Bayesian and frequentist analysis, which reveals that the Kullback–Leibler (KL) divergence to the prior in variational Bayes plays a role analogous to a complexity penalty in frequentist aggregation. Building on this insight, we develop an early stopping method for frequentist exponential aggregation and establish near-oracle guarantees. Moreover, to address the practical difficulty of determining the complexity penalties in non-asymptotic settings, we further introduce a sample-splitting implementation that replaces penalty tuning with validation-based early stopping.
\end{enumerate}

\paragraph{Computational efficiency}
In addition to the theoretical results, the ESA method offers substantial practical benefits. By eliminating computation for overly complex models, the procedure becomes scalable to large model collections and high-dimensional applications. We illustrate the method in several examples, including large-scale image classification, clustering, high-dimensional linear regression and prediction for tabular data. Across these tasks, we demonstrate empirically that the proposed approach achieves strong performance with significantly reduced computational cost.

\subsection{Organization}
The rest of the paper is organized as follows. In \cref{sec:preliminaries}, we introduce the notation and statistical setup. We then provide a concise overview of relevant frequentist and Bayesian methods. In \cref{sec:method}, we describe the proposed ESA approach within variational Bayes framework and investigate the theoretical properties of ESA, including not-too-early and not-too-late stopping properties and oracle contraction rates. In \cref{sec:eb}, we extend the Bayesian ESA theory in \cref{sec:method} to empirical Bayes variational posteriors, where prior hyperparameters are selected in a data-dependent manner. In \cref{sec:erm}, we present a frequentist ESA method based on penalized empirical risk minimization, and derive the corresponding oracle inequality. Numerical experiments are presented in \cref{sec:experiment}. We conclude our paper in \cref{sec:conclusion} with a discussion of future research directions.

\section{Preliminaries}
\label{sec:preliminaries}

\subsection{Notation}
Let  $\R_{\ge0}:=\{z\in\R:z\ge0\}$. For a natural number $m\in\bN$, we let $[m]:=\{1,2,\dots, m\}$. For $d\in\bN\cup\{\infty\}$, $p\ge1$, and a $d$-dimensional vector $x:=(x_1,\dots, x_d)^\top\in\R^d$, we let $\|x\|_p:=\del[0]{\sum_{j=1}^d|x_j|^p}^{1/p}$ and $\|x\|_\infty:=\max_{j\in[d]}|x_j|$ 
Let  $I_d$ be the $d$-dimensional identity matrix.
Let $\Delta_d:=\cbr[0]{\omega\in[0,1]^d:\|\omega\|_1=1}$ denote the $d-1$-dimensional simplex. For two real numbers $a,b\in\R$, we write $a\vee b:=\max\{a,b\}$ and $a\wedge b:=\min\{a,b\}$.  For an arbitrary set $B$, we denote by $B^\complement$ its complement and by $|B|$ its cardinality.  For two positive sequences $(a_n)_{n\in \mathbb{N}}$ and $(b_n)_{n\in \mathbb{N}}$, we write $a_n\lesssim b_n$ or  $b_n\gtrsim a_n$, if there exists a positive constant $C>0$ such that $a_n\le Cb_n$ for any $n\in \mathbb{N}$. Moreover, we write $a_n\asymp b_n$ if both $a_n\lesssim b_n$  and  $a_n\gtrsim b_n$ hold. Let  $\ind(\cdot)$ denote the indicator function. 
For a measurable space $\cX$, we denote by $\cP(\cX)$ the set of all probability measures supported on $\cX$. For a probability measure $\P\in\cP(\cX)$ and a function $g$ on $\cX$, we write $\P g:=\int_{\cX} g\d\P$, i,e., $\P g$ denotes the expectation of $g$ with respect to the measure $\P$. For two probability distributions $\P_1,\P_2\in\cP(\cX)$, we denote by $\kl(\P_1, \P_2)$ the Kullback-Leibler (KL) divergence from $\P_2$ to $\P_1$, which is defined by $\kl(\P_1,\P_2):=\int \log(\frac{\d\P_1}{\d\P_2})\d\P_1$ if $\P_1\ll\P_2$ and $\kl(\P_1,\P_2):=\infty$ otherwise. For $\rho\in(0,1)\cup(1,\infty)$, the $\rho$-R\'enyi divergence from $\P_2$ to $\P_1$ is defined as $\cD_\rho(\P_1,\P_2):=\frac{1}{\rho-1}\log\del[0]{\int \del[0]{\frac{\d\P_1}{\d\P_2}}^{\rho-1}\d\P_1}$ if $\P_1\ll\P_2$ and $\cD_\rho(\P_1,\P_2):=\infty$ otherwise. For a set $\Theta$ equipped with a metric $\fh$, we say that the set  $\{\thetab_1,\dots, \thetab_N\}$ is a $\epsilon$-cover of  the set $\Theta$, if for any $\theta\in \Theta $, there exists $j\in[N]$ such that $\fh(\thetab_j,\theta)\le \epsilon$. Let $\cN(\epsilon, \Theta,\fh)$ denote the minimal cardinality of a $\epsilon$-cover of $\Theta$. 

\subsection{Setup}

Suppose that we observe a sample $X^{(n)}\in \cX_n$ of size $n$ from the true data-generating distribution $\sP_{\star}^{(n)}\in \cP(\cX_n)$, where $\cX_n$ denotes a measurable sample space.  We model the sample by a parameterized distribution $\Ptheta\in \cP(\cX_n)$ with parameter $\theta$.  The parameter $\theta$ can be also infinite-dimensional, which embraces semi- and non-parametric models.

For each distribution $\P^{(n)}\in \cP(\cX_n)$ and a given sample $X^{(n)}$, we introduce a \textit{loss} function $\ell_n:\cP(\cX_n)\times \cX_n\mapsto \R_{\ge0} $ to measure the ``discrepancy'' between the distribution $\P^{(n)}$ and the sample $X^{(n)}$. For example, for a regression setup, where the sample $X^{(n)}=\{(X_i, Y_i)\}_{i\in[n]}$ consists of $n$-pairs of input $X_i$ and output $Y_i$, we may use the square loss function $\ell_n(\P^{(n)}, X^{(n)})=\sum_{i=1}^n(Y_i- \P^{(n)}(Y_i|X_i))^2$, where $\P^{(n)}(Y_i|X_i)$ denotes the conditional expectation of $Y_i$ given $X_i$ under the law $\P^{(n)}$. For notational simplicity, we hide the dependence of the sample as $\ell_n(\P^{(n)}):=\ell(\P^{(n)}, X^{(n)})$. Moreover, we simply write $\ell_n(\theta)=\ell_n(\Ptheta)$  for a parameterized distribution $\Ptheta$, and $\ell_n^\star:=\ell_n(\Pstar)$ for the true distribution $\Pstar$.  We define the \textit{excess risk} of a distribution $\P^{(n)}\in \cP(\cX_n)$ as
      \begin{align*}
       \cL_n(\P^{(n)}):=   \cL_n(\P^{(n)}, \Pstar):= \Pstar\sbr{\ell_n(\P^{(n)})-\ell_n(\Pstar)}.
    \end{align*}
For simplicity, we write $\cL_n(\theta):=\cL_n(\theta;\Pstar):= \cL_n(\Ptheta, \Pstar)$. Our aim is to show that the proposed  methods,  such as variational Bayes methodology,  give a large mass on the parameters with small excess risk. Moreover,  we define the “excess’’ loss by $\lossdiff(\theta):=\ell_n(\theta)-\ell_n^\star$ for a parameter $\theta$, a quantity that will recur frequently in our theoretical analysis.

For conducting adaptive inference, the bias-variance trade-off should be optimized. This is usually done by selecting a suitable one among several candidate models. In this paper, we also follow such a route as we consider the collection of models, which can be ordered by their complexity. To be more precise, we define the nestedness of models.

\begin{definition}
For two models $\Theta_1$ and $\Theta_2$, the model $\Theta_1$ is \textit{nested} within the model $\Theta_{2}$ and write $\Theta_1\hookrightarrow \Theta_2$ if there exists an embedding $\sT:\Theta_1\mapsto \Theta_2$ satisfying $\sP^{(n)}_\theta=\sP^{(n)}_{\sT(\theta)}$ for any $\theta\in \Theta_{1}$.
\end{definition}

In the rest of the paper, we consider the following ``ladder'' of nested models
 \begin{align*}
        \Theta_{n,1}\hookrightarrow \Theta_{n,2}\hookrightarrow \cdots \hookrightarrow \Theta_{n,M_n}.
    \end{align*}
Here, the models $\Theta_{n,1},\dots,\Theta_{n,M_n}$ are allowed to depend on the sample size $n$. For the ease of the description, we let $\Theta_n:=\bigcup_{\k=1}^{M_n}\Theta_{n,\k}$. Under this nested nature of the models, the approximation error decreases as the model index $\k$ increases, while the estimation error increases, giving rise to the bias-variance trade-off along the model ladder.

\subsection{Related statistical methodologies}

\subsubsection{Exponentially weighted aggregate}
\label{subsec:ewa}

For each model $\k\in[M_n]$, suppose that we obtain a candidate estimator $\htheta_{n,\k}$, typically defined as  the empirical risk minimizer over the parameter space $\Theta_{n,\k}$. The exponentially weighted aggregate (EWA) combines these candidate estimators as
    \begin{align*}
        \htheta_{n}^{\le M_n}:=\sum_{\k=1}^{M_n}\uomega_{n,\k}^{\le M_n}\htheta_{n,\k},
    \end{align*}
where each weight $\uomega_{n,\k}$ assigned to the model $\k$ is proportional to $\exp(-\hcL_n)$, the exponential of a suitable estimate $\hcL_n$ of the excess risk $\cL_n(\htheta_{n,\k})$ of  $\htheta_{n,\k}$. Since the empirical risk $\ell_n(\htheta_{n,\k})$ underestimates the population counterpart, we need to introduce an additional \textit{penalty term} $H_{n,\k}$ designed to account for the discrepancy between the empirical and population risks. The choice of the penalty $H_{n,\k}$ depends on the statistical context and the complexity of
the model class $\Theta_{n,\k}$. Under appropriate calibration of the penalty, the penalized empirical risk
$\ell_n(\htheta_{n,\k}) + H_{n,\k}$ provides a valid surrogate for the population excess risk
$\cL_n(\htheta_{n,\k})$, ideally $\Pstar[\ell_n(\htheta_{n,\k})+H_{n,\k}]\asymp\Pstar[\cL_n(\htheta_{n,\k})] $. In this case, EWA satisfies an oracle inequality of the form
    \begin{align*}
        \cL_n( \htheta_{n}^{\le M_n})
        \lesssim \min_{\k\in[M_n]}\cbr{\min_{\theta\in\Theta_{n,\k}}\cL_n(\theta)+H_{n,\k}}
    \end{align*}
up to an additional non-leading remainder term with high probability \citep[e.g.,][]{dalalyan2012sharp,alquier2011pac,rigollet2011exponential}. This result shows that the EWA achieves a performance comparable to that of the best candidate model, optimally balancing approximation and estimation errors.

\subsubsection{Generalized variational Bayes}

For each model $\k\in[M_n]$, we impose the prior distribution $\Pi_{n,\k}$ on $\Theta_{n,\k}$. Given a specified loss function $\ell_n$, the generalized variational Bayes method aims to minimize the variational free energy (VFE) given by
    \begin{align}
        \vfe(Q):=\lambda\int \ell_n(\theta)\d Q(\theta)+\kl(Q,\Pi_{n,\k})
    \end{align}
over the variational distribution $Q$ belonging to the variational family $\cQ_{n,\k}\subset\cP(\Theta_{n,\k})$, which is specified by a user. Here, $\lambda>0$ is a constant called a learning rate, which adjusts the influence of the observed sample to the resulting variational posterior. The generalized variational posterior on model $\Theta_{n,\k}$ is given as the minimizer of the VFE as
    \begin{align}
        \hQ_{n,\k}=\argmin_{Q\in \cQ_{n,\k}} \vfe(Q).
    \end{align}
If we use the negative log-likelihood as a loss function, it recovers a standard (when $\lambda=1)$ and fractional (when $\lambda\in(0,1)$) variational posterior distributions.
We denote by the minimized VFE for a model $\k$ as
    \begin{align}
    \label{eq:vfe}
     \mvfe(\k):=\vfe(\hQ_{n,\k})=\lambda\int \ell_n(\theta)\d \hQ_{n,\k}(\theta)+\kl(\hQ_{n,\k},\Pi_{n,k})
    \end{align}
which will serve as a model selection criterion in our approach. Contraction properties of the generalized variational posterior have been thoroughly investigated in \cite{alquier2016properties}, while those of the variational fractional posterior, as a specific instance, have been analyzed in \citet{alquier2020concentration,yang2020alpha}. 

\subsubsection{Adaptive variational Bayes }
\label{subsec:avb}
A general and computable variational Bayes approach for adaptive inference, based on the aggregation principle,  was proposed by  \cite{ohn2024adaptive}. Within the adaptive variational Bayes framework, first the (generalized) variational posterior $\hQ_{n,\k}$ is computed for each model $\k\in[M_n]$. Then these $M_n$ many posteriors over all candidate models are aggregated. The aggregation weight assigned to the model $\k$ is given as
    \begin{align*}
        \homega_{n,\k}^{\le M_n}:=\frac{ \exp(- \mvfe(\k))}{
            \sum_{\k'\in[M_n]}\exp(- \mvfe(\k'))      }.
    \end{align*}
This leads to the (generalized) adaptive variational posterior given by
    \begin{align*}
        \hQ_{n}^{\le M_n} :=\sum_{\k=1}^{M_n}\homega_{n,\k}^{\le M_n}\hQ_{n,\k}.
    \end{align*}
\cite{ohn2024adaptive} showed that this aggregate of generalized variational posteriors can attain optimal contraction rates adaptively under fairly general conditions.

\section{Early-stopped aggregation for variational Bayes}
\label{sec:method}

While the adaptive variational Bayes approach we have described in \cref{subsec:avb} provides a general and theoretically optimal solution to adaptive inference, it relies on computing variational posteriors across the full range of candidate models. In practice, this may involve substantial computational effort devoted to models that are far more complex than necessary. This motivates a refinement of the aggregation procedure that reduces computational burden without sacrificing statistical guarantees.

\subsection{Early-stopped aggregation of variational posteriors}
To this end, we propose an early stopping mechanism that determines, in a data-driven fashion, how far one needs to proceed along the model ladder. Starting from the simplest model $\k=1$, the procedure monitors the optimized variational free energy as the model index $\k$ increases and terminates the computation once no further improvement is observed. The resulting method preserves the adaptive aggregation principle while substantially reducing computation. 

We now describe this early-stopped aggregation (ESA) approach for variational Bayes in detail. Specifically, we stop the variational Bayes procedure until we reach the \textit{stopping model index} $\hsm$, which is defined as
    \begin{align}
        \hsm:=\inf\cbr{\k\in[M_n]\setminus\{1\}:
        \mvfe(\k-1)< \mvfe(\k)}\wedge M_n.
    \end{align}
In other words, we stop the procedure when the optimized variational free energy increases. The stopping rule admits a natural statistical interpretation. Along the model ladder, increasing model complexity typically reduces approximation error while inflating estimation error. The proposed rule identifies the
point at which further increases in complexity no longer yield sufficient improvement in the data-fitting term to compensate for the increased complexity penalty. In this sense, the stopping index provides a
data-driven approximation to the optimal bias-variance trade-off. 
We propose to use the adaptive variational posterior with early stopping
    \begin{align}
        \esavp
        :=\sum_{\k=1}^{\hsm}\omegaesa\hQ_{n,\k}
          \text{ with }
          \omegaesa:=\frac{ \exp(-\mvfe(\k))}{  \sum_{\k'=1}^{\hsm}  \exp(-\mvfe(\k')) }.
    \end{align}
This ESA procedure is summarized in \cref{alg:esa_Vb}.

\begin{algorithm}[H]
\caption{Early-stopped aggregation for variational Bayes}
Compute $\hQ_{n,1}$ and $\mvfe(1)$ and set $\sm \leftarrow 1$\;
\Repeat{
    $\mvfe(\sm) > \mvfe(\sm-1)$ or $\sm=M_n$
}{
 $\sm \leftarrow \sm + 1$\;
    Compute $\hQ_{n,\sm}$ and $\mvfe(\sm)$\;
}
Set $\hsm \leftarrow \sm$ \tcp*[r]{Early stopping model index}
Compute $\omegaesa :=\exp\left(-\mvfe(\k) \right)/\sum_{\k'=1}^{\hsm}\exp\left(-\mvfe(\k') \right)$ for $\k\in[\hsm]$\;

\Return{The ESA variational posterior $\hQ_n^{\esa} := \sum_{\k=1}^{\hsm} \omegaesa \hQ_{n,\k}.$
}
\label{alg:esa_Vb}
\end{algorithm}

We discuss two algorithmic variants of ESA in the following remarks, respectively.

\begin{remark}[Backward ESA]
The ESA procedure proceeds along the model ladder starting from simpler models
and increasing the model complexity. This choice is primarily motivated by computational considerations, as fitting variational approximations on smaller models is typically cheaper, and early stopping allows the procedure to terminate before reaching highly complex models. From a theoretical perspective, the ESA idea is not tied to a particular direction of traversal along the model ladder.  In principle, we may also start from a sufficiently large model and move toward simpler ones, stopping when the criterion no longer improves.  Such a backward search strategy could also be analyzed theoretically, although we focus on the forward version for computational convenience. 
\end{remark}

\begin{remark}[Earlier termination]
\label{remark:earlier_termination}
As an alternative, we may introduce a ``promoting'' parameter $\delta>0$, in which case the
stopping rule is modified to $ \mvfe(\k-1)< (1+\delta)\mvfe(\k)$. This modification leads to an earlier termination of the early stopping procedure and hence can substantially improve computational efficiency, especially in large-scale problems. While we do not provide a theoretical guarantee for this heuristic modification, we adopt this approach in some large-scale numerical studies to reduce computational costs.
\end{remark}

\subsection{Theoretical properties of ESA for varational Bayes}
\label{sec:theory}

\subsubsection{Assumptions}

Throughout the paper, the constants $\xi_1,\xi_2,\dots,$ are all absolute positive constants that will be fixed throughout our theoretical analysis. The first assumption concerns the probabilistic behavior of the excess loss.

\begin{assumption}
\label{assume:loss}
There exist positive absolute constants $\xi_1$, $\xi_2$ and $\xi_3$ such that
        \begin{align}
        \label{eq:assume_learnability}
         \Pstar\sbr{\exp\del{-\lambda\{\ell_n(\theta)-\ell_n^\star \}}}&\le \exp\del{- \xi_1 \cL_n(\theta)}\\
     \Pstar\sbr{\exp\del{\lambda \xi_2  \{\ell_n(\theta)-\ell_n^\star \}}}
      &\le \exp\del{ \xi_3 \cL_n(\theta)}
    \end{align}
for any $\theta\in \Theta_{n}$ and any sufficiently large $n\in\bN$.
\end{assumption}

\cref{assume:loss} is a slightly simplified version of the Bernstein condition which is standard in theoretical studies for generalized (variational) posteriors \citep[e.g.,][]{alquier2016properties,syring2023gibbs,ohn2024adaptive}. This condition holds when the loss function is a sub-exponential random variable whose variance is bounded by the excess risk up to a constant. This setting covers many statistical applications; see the references we have cited. Also,  fractional likelihood satisfies \cref{assume:loss} (for example, see Proposition M.1 of \citet{ohn2024adaptive}), while standard likelihood violates \eqref{eq:assume_learnability}. Thus, our analysis does not cover the standard variational Bayes approach, which requires more involved theoretical techniques.

We define the \textit{excess risk of a model $\k$} as
    \begin{align}
     \label{eq:terror_model}
      \cE_n(\k):=- \log\del{\int\exp\del{-\xi_1\cL_n(\theta)}\d\Pi_{n,\k}(\theta)},
    \end{align}
which is the log-exp-sum of the excess risk over the prior $\Pi_{n,\k}$ on the model $\k$. This quantity can be viewed as a population version of the variational free energy of an ``oracle'' posterior distribution
    \begin{align*}
        \d\Pi_{n,\k}^{-\xi_1\cL_n}(\theta)
        :=\frac{\exp(-\xi_1\cL_n(\theta))\d\Pi_{n,\k}(\theta)}{\int\exp\del{-\xi_1\cL_n(\theta')}\d\Pi_{n,\k}(\theta')}
        = \frac{\exp(-\xi_1\cL_n(\theta))\d\Pi_{n,\k}(\theta)}{\exp(-\cE_n(\k))}
    \end{align*}
as shown below
    \begin{align*}
    \cE_n(\k)&= -\log(\exp(-\cE_n(\k)))\int\d \Pi_{n,\k}^{-\xi_1\cL_n}(\theta)\\
            &=\int\cbr{\xi_1\cL_n(\theta)+\log\del{\frac{\exp(-\xi_1\cL_n(\theta))}{\exp(-\cE_n(\k))}} }\d \Pi_{n,\k}^{-\xi_1\cL_n}(\theta)\\
        &=\xi_1\int \cL_n(\theta)\d \Pi_{n,\k}^{-\xi_1\cL_n}(\theta)+\kl(\Pi_{n,\k}^{-\xi_1\cL_n},\Pi_{n,\k}).
    \end{align*}
Moreover, since the infimum in Donsker and Varadhan’s variational formula (given in \cref{lemma:variational_formula}) is attained by $\Pi_{n,\k}^{-\xi_1\cL_n}$, we have
   \begin{align}
      \cE_n(\k)&= \inf_{Q\in\cP(\Theta_{n,\k})} \cbr{\xi_1\int \cL_n(\theta)\d Q(\theta)+\kl(Q, \Pi_{n,\k})}.
    \end{align}
This relationship implies that \eqref{eq:terror_model} is equivalent to the optimal trade-off over the model $\k$ between the approximation and estimation errors, represented by the integrated excess risk $\int \cL_n(\theta)\d Q(\theta)$ and the KL divergence to the prior  $\kl(Q, \Pi_{n,\k})$, respectively.

\begin{assumption}
\label{assume:variational_family}
There exists an absolute constant $\xi_4>0$ such that
        \begin{align}
        \label{eq:assume_vb_approx}
        \log\del{\int \exp(\xi_3\cL_n(\theta) )\d Q_{n,\k}^*(\theta)} &\le \xi_4 \cE_n(\k)\\
        \kl(Q_{n,\k}^*, \Pi_{n,\k}) &\le \xi_4 \cE_n(\k)
    \end{align}
for some distribution $Q_{n,\k}^*\in \cQ_{n,\k}$ for any $\k\in[M_n]$ and any sufficiently large $n\in \bN$. 
\end{assumption}

\cref{assume:variational_family} requires the existence of an ``oracle'' variational posterior $Q_{n,\k}^*$ within the variational family $ \cQ_{n,\k}$, which provides a almost optimal trade-off between approximation and estimation errors. While such requirements on the variational family are common in the related literature such as \citep{zhang2020convergence,alquier2020concentration,yang2020alpha,ohn2024adaptive},  here we impose a slightly stronger version. Specifically, the condition \eqref{eq:assume_vb_approx} involves the exponential moment of the excess risk over $ Q_{n,\k}^*$,  which is, in view of Jensen's inequality, a stronger requirement than the usual expectation-based condition $\int \cL_n(\theta)\d Q_{n,\k}^*(\theta)\lesssim \cE_n(\k)$ found in the earlier works  mentioned above.  
By assuming this stronger light-tailed condition on the  oracle variational distribution (e.g., Gaussian or uniform distributions), we are able to attain exponentially decaying probability bounds in theorems in what follows. 

The following proposition illustrates $\cE_n(\k)$ is an upper bound of the average excess risk over the variational posterior $\hQ_{n,\k}$.
\begin{proposition}
\label{prop:risk_bound_vb}
Under \cref{assume:loss,assume:variational_family}, there exists an absolute constant $\tC_1>0$ such that
        \begin{align*}
             \Pstar\sbr{ \int \cL_n(\theta)\d \hQ_{n,\k}(\theta)}
             \le \tC_1\cE_n(\k)
        \end{align*}
  for any $\k\in[M_n]$ and any sufficiently large $n\in \bN$.
\end{proposition}

\subsubsection{Not-too-early-stopped property}

In this subsection, we show that the early stopping procedure contains  a ``near-optimal'' model with a probability converging to 1. For a given constant $\tau>0$, we let  $\kopt(\tau)\in[M_n]$ be a model index defined as
    \begin{align}
      \label{eq:kopt_model}
        \kopt(\tau)
        :=\inf\cbr{\k\in[M_n]:\cE_{n}(\k)\le (1+\tau)\cE_{n}(\k+1)}\wedge M_n.
    \end{align}
That is, for any $\k\in[ \kopt(\tau)-1 ]$, we have $\cE_{n}(\k)> (1+\tau)\cE_{n}(\k+1)$. This is not equal to the optimal model minimizing the excess risk always, but it is in some situations. The next proposition states that the model $\kopt(\tau)$ is equal to the optimal model when the excess risk curve is U-shaped.

\begin{proposition}
\label{prop:ushaped_risk}
Suppose that we consider $M$ models for each $n\in\bN$ and assume that  $\cE_n(\k)\asymp n^{\alpha(\k)}$ (up to a multiplicative poly-logarithmic factor) for all $\k\in[M]$, where $\{\alpha(\k):\k\in[M]\}$ is a collection of real numbers such that $\alpha(1)>\cdots>\alpha(\k^\star-1)>\alpha(\k^\star)<\alpha(\k^\star+1)<\cdots<\alpha(M)$. Then for any $\tau>0$, we have $ \kopt(\tau)=\k^\star$ eventually.
\end{proposition}

\begin{proof}
By assumption, for $\k<\k^\star$, $\cE_n(\k-1)/\cE_n(\k)\asymp n^{\alpha(\k-1)-\alpha(\k)}\to  \infty$
as $n\to \infty$.  This implies that $\cE_n(\k-1)> (1+\tau)\cE_n(\k)$, and thus, $\kopt(\tau)\ge\k^\star$ eventually. Similarly, $\cE_n(\k^\star)\le (1+\tau)\cE_n(\k^\star+1)$, and thus,  $\kopt(\tau)\le \k^\star$ eventually.
\end{proof}

\begin{example}
Consider a Gaussian sequence model $\Ptheta:=\otimes_{i=1}^\infty \N(\theta_i,n^{-1})$ with an infinite-dimensional parameter $\theta:=(\theta_i)_{i\in\bN}$. Suppose that the sample $X^{(n)}$ is generated from the true distribution $\P_{\theta^\star}^{(n)}$ indexed by the true parameter $\theta^\star:=(\theta^\star_i)_{i\in\bN}$ with $\theta^\star_i:=i^{-\beta_*-1/2}$ for $\beta_*>0$. Note that  $\theta^\star$ belongs to a Sobolev ball with smoothness $\beta\in(0,\beta_*)$. We consider a square loss function $\ell_n(\theta)=n\|X^{(n)}-\theta\|_2^2$. Then the corresponding excess risk is given by $\cL_n(\theta)=n\|\theta-\theta^\star\|_2^2$. Let $\{q(\k):\k\in[M]\}$ be a collection of real numbers such that $0\le q(1)<q(2)<\cdots <q(M)\le 1$. For each $\k$, we consider a model given as
    \begin{align*}
        \Theta_{n,\k}:=\cbr{\theta=(\theta_i)_{i\in\bN}:\theta_{i}=0\text{ for any }i>n^{q(\k)}}
    \end{align*}
and impose a prior distribution $\Pi_{n,\k}$ on $\Theta_{n,\k}$, under which $\theta_i\sim \N(0,1)$ independently for $i\le n^{q(\k)} $ and $\theta_i=0$ for $i>n^{q(\k)}$. Then it is easy to see that
    \begin{align*}
        \cE_n(\k)\asymp n\sum_{i>n^{q(\k)}}(\theta_i^\star)^2 + n^{q(\k)}\log n
        \asymp n^{1-2\beta_*q(\k)}+n^{q(\k)}\log n.
    \end{align*}
Therefore, the condition of \cref{prop:ushaped_risk} holds with $\alpha(\k)=(1-2\beta_*q(\k))\vee q(\k)$, and $\k^\star$ is the model index for which $q(\k^\star)$ is closest to $1/(2\beta_*+1)$.
\end{example}

The following theorem shows that our early-stopping strategy does not halt the procedure prematurely.

\begin{theorem}[Not-too-early-stopping]
\label{thm:selection_vb}
Under \cref{assume:loss,assume:variational_family}, there exist absolute constants $\tau_*>0$ and $\tC_2>0$ such that
    \begin{align}
       \Pstar\del[1]{\hsm\ge \kopt(\tau_*) +1}\ge 1-\eta_n
    \end{align}
for any sufficiently large $n\in \bN$, where we let $\eta_n:=\kopt(\tau_*)\exp\del[1]{-\tC_2 \cE_n(\kopt(\tau_*))}.$
\end{theorem}

\begin{remark}[Application to hyperparameter tuning]
The ESA procedure can also be applied to hyperparameter tuning. In this setting, we consider a single ``large'' parameter space $\Theta_n$ together with a collection of prior distributions $\Pi_{n,1},\dots,\Pi_{n,M_n}$ on $\Theta_n$, corresponding to different choices of a hyperparameter. For each prior $\Pi_{n,\k}$, the excess risk is defined as the same as \eqref{eq:terror_model}. Consequently, \cref{thm:selection_vb} ensures a not-too-early stopping property for the resulting hyperparameter tuning procedure based on early stopping. 
\end{remark}

\subsubsection{Oracle contraction rates}

In this section, we investigate the contraction rate of the ESA variational posterior in terms of the excess risk. Toward this aim, we first show that the ``non-variational'' counterpart of the ESA variational posterior contracts at the oracle rate and then show that the variational approximation error of the ESA variational posterior to this non-variational posterior is small. Given the selected upper bound $\hsm$, we define the prior distribution as
    \begin{align*}
        \Pi_{n,\le\hsm}:=\frac{1}{\hsm}\sum_{\k=1}^{\hsm}\Pi_{n,\k}
    \end{align*}
and the corresponding generalized posterior as
    \begin{align*}
        \d\esapst(\theta):=  
        \frac{\exp(-\lambda\ell_n(\theta))\d\Pi_{n,\le\hsm}(\theta)}{\int \exp(-\lambda\ell_n(\theta'))\d\Pi_{n,\le\hsm}(\theta')}.
    \end{align*}
That is, $\esapst$ is as the posterior distribution over the union of the models from the first to the $\hsm$-th with the prior distribution $\Pi_{n,\le\hsm}$.

The next theorem, as an intermediate result, shows that the ESA posterior can attain the oracle rate
    \begin{align}
        \cE_n^*:=\min\cbr{\cE_n(\k):\k\in[\kopt(\tau_*)+1]}
        =\cE_n(\kopt(\tau_*))\wedge \cE_n(\kopt(\tau_*)+1).
    \end{align}
    
\begin{theorem}[Contraction rate of posterior with early stopping]
\label{thm:contract_posterior}
Assume that $\eta_n\to0$ as $n\to\infty$. Then under the same assumption of \cref{thm:selection_vb},
    \begin{align}
       \Pstar\sbr{\esapst\del{ \cL_n(\theta)\ge A_n( \cE_n^*+\log M_n)}}\to0
    \end{align}
as $n\to\infty$ for any diverging sequence $(A_n)_{n\in\bN}$ with $A_n\to\infty$.
\end{theorem}

\begin{remark}[Variational assisted Bayesian inference]
We note that even when the goal is to compute the non-variational ESA posterior, the stopping model index can be determined using variational approximations. This suggests a practical strategy for Bayesian inference, in which variational Bayes is used solely to identify a suitable upper bound on model complexity, after which the original posterior distribution can be computed on a substantially reduced model space.
\end{remark}

Based on the above contraction result of the exact posterior  $\esapst$ with the ESA approach, we can attain the following contraction result of its variational approximation $ \esavp$ by showing that its variational approximation error is not much large.

\begin{theorem}[Contraction rate of ESA variational posterior]
\label{thm:contract_esva}
Under the same assumption of \cref{thm:contract_posterior},
    \begin{align}
       \Pstar\sbr{  \esavp\del{ \cL_n(\theta)\ge A_n( \cE_n^*+\log M_n)}}\to0
    \end{align}
as $n\to\infty$ for any diverging sequence $(A_n)_{n\in\bN}$ with $A_n\to\infty$.
\end{theorem}

\subsubsection{Not-too-late-stopped property}

Let $\kopt:=\kopt(\tau_*)$ be the model index considered in \cref{thm:selection_vb}. In this section, we show that under a certain additional condition, our ESA procedure excludes overly complex models, i.e., stopping occurs not too late. This supports the computational efficiency of the ESA procedure. For $\nu>0$, we define a model index $\kover(\nu)$ as an index such that
    \begin{align}
        \kover(\nu):=\sup\{\k\in[M_n]\setminus[\kopt-1]:\cE_n(\k) \le (1+ \nu)\cE_n^*\}\vee \kopt
    \end{align}
    
\begin{theorem}
\label{thm:overestimation}
Under the same assumptions of \cref{thm:selection_vb}, there exist absolute constants $\nu_*>0$ and $\tC_3>0$ such that
    \begin{align}
       \Pstar\del[1]{\hsm\le \kover(\nu_*)+1 }\ge 1-2\exp\del{-\tC_3 \cE_n^*}
    \end{align}
for any sufficiently large $n\in \bN$.
\end{theorem}

If the excess risk curve is U-shaped, the stopping model index is equal to the  model index that minimizes the excess risk.

\begin{corollary}
Assume the same setup of \cref{prop:ushaped_risk}. Then under the same assumptions of \cref{thm:selection_vb}, we have
    \begin{align*}
        \Pstar\del[1]{\hsm=\k^\star+1 }\ge 1-\eta_n-2\exp\del{-\tC_3 \cE_n^*}
    \end{align*}
for any sufficiently large $n\in \bN$
\end{corollary}

\begin{proof}
Combining the results of \cref{prop:ushaped_risk} and \cref{thm:selection_vb}, we have $  \Pstar\del[1]{\hsm\ge \k^\star+1 }\ge1-\eta_n$. On the other hand, since $\cE_n(\k^\star+1)/\cE_n(\k^\star)\to\infty$, we have $\k^\star+1>\kover(\nu)$, i.e.,  $\k^\star+1\ge\kover(\nu)+1$ for any $\nu>0$ eventually. Combining these two results, we get the desired result.
\end{proof}

\begin{remark}
Relative to \citet{ohn2024adaptive}, the key novelty of the present paper lies in replacing full global aggregation over the entire model ladder with a local early-stopping mechanism. In  \citet{ohn2024adaptive}, variational posteriors are computed for all candidate models and then aggregated using weights determined by their variational free energies. In contrast, the present paper shows that one can traverse the ladder sequentially, stop once the energy criterion ceases to improve, and aggregate only along the visited path. Establishing that such a local stopping rule still preserves adaptive inference requires new theory, including not-too-early and not-too-late stopping guarantees shown above and new concentration arguments linking local behavior of the energy functional to a global oracle comparison. In particular,  in order to show that a local stopping rule can recover the global oracle rate, one needs new technical tools to control the concentration of the energy functional \emph{uniformly} along the model ladder. This constitutes a genuinely new theoretical ingredient that has no counterpart in the earlier papers.

In addition, ESA is not merely a device tailored to a single procedure such as generalized variational Bayes. Instead, it embodies a more general principle that also arises in other settings such as empirical Bayes and frequentist aggregation, as we show in the following sections.
\end{remark}

\section{Early-stopped aggregation for variational empirical Bayes}
\label{sec:eb}

In this section, we extend the ESA approach to the use of data-dependent priors. In many practical implementations of variational Bayes, hyperparameters governing the prior distribution are tuned by minimizing a certain objective function such as the variational free energy and the marginal likelihood. This \textit{empirical Bayes} strategy provides a computationally attractive alternative to hierarchical Bayes procedures or cross-validation that can become costly when dealing with more than one hyperparameter. Due to this probable computational efficiency, the empirical Bayes methods have been employed in various applications, such as Dirichlet mixture models \citep{mcauliffe2006nonparametric}, sparse linear regression \citep{castillo2018empirical,wang2020simple}, group variable selection \citep{ge2025variational}, and matrix factorization \citep{wang2021empirical}. Establishing ESA theory for the empirical variational Bayes setting requires substantial new technical tools. In particular, we develop new deviation arguments to control the additional randomness induced by hyperparameter optimization, and achieve this through a R\'enyi-type entropy condition on the class of candidate priors.

\subsection{Early-stopped aggregation of variational empirical posteriors}

We begin by outlining the statistical framework that incorporates empirical Bayes  methods into our variational Bayes estimation procedure. Suppose that we consider a family of prior distributions such that 
    \begin{align}
    \Pi[\Psi]_{n,\k} := \{ \Pi_{n,\k}(\cdot|\psi):\psi\in \Psi_{n,\k}\}
    \end{align}
indexed by a hyperparameter $\psi\in \Psi_{n,\k}$ for each model $\k$ and then we optimize the hyperparameter $\psi$ based on the observed sample. We consider a slightly modified variational free energy as
  \begin{align}
        \vfe(Q,\psi):=\frac{\lambda}{\barho}\int \ell_n(\theta)\d Q(\theta)+\kl(Q,\Pi_{n,\k}(\cdot|\psi)),
    \end{align}
where we adopt the learning rate parameter $\lambda/\barho$ instead of $\lambda$ as in the original formulation, solely to reusing \cref{assume:loss} without redefining the relevant constants. We specify the value of $\barho$ in \cref{assume:hyperpar} given later. The empirical Bayes procedure works in our framework as
     \begin{align}
       (\chQ_{n,\k},\cpsi_{n,\k})
       =\argmin_{(Q,\psi)\in\cQ_{n,\k}\times \Psi_{n,\k}}
       \cbr{\vfe(Q,\psi) +\Upsilon_{n,\k}(\psi)},
    \end{align}
where the hyperparameter $\psi$ is optimized as well. Here, $\Upsilon_{n,\k}:\Psi_{n,\k}\mapsto \R_{\ge0}$ quantifies a prior belief about the hyperparameter, analogous to imposing a prior distribution on $\psi.$ For model comparison, we use the maximized VFE over both the variational family and the hyperparameter space, which is given as
 \begin{align}
     \ebmvfe(\k)&:=\vfe(\chQ_{n,\k},\cpsi_{n,\k})\\
     &=\frac{\lambda}{\barho}\int \ell_n(\theta)\d \chQ_{n,\k}(\theta)+\kl(\chQ_{n,\k},\Pi_{n,\k}(\cdot|\cpsi_{n,\k}))
    \end{align}
for each model $\k\in[M_n]$. Accordingly, our stopping model index is given by
  \begin{align}
        \csm:=\inf\cbr{\k\in[M_n]\setminus\{1\}:
        \ebmvfe(\k-1)< \ebmvfe(\k)}\wedge M_n.
    \end{align}
The ESA approach with the empirical Bayes principle yields the following aggregate of the variational posteriors, which we call the \textit{ESA-EB variational posterior},
    \begin{align}
        \esaebvp
        :=\sum_{\k=1}^{\csm}\omegaebesa \chQ_{n,\k}
        \text{ with }
          \omegaebesa :=\frac{\exp(-\ebmvfe(\k))}{\sum_{\k'=1}^{\csm}\exp(-\ebmvfe(\k'))}.
    \end{align}

\subsection{Not-too-early-stopped property and oracle contraction rates}

In the empirical Bayes setting, it is natural to define the excess risk of a model $\k$ as the best possible one over the family of candidate priors, i.e., we define it as
     \begin{align}
        \ccE_n(\k):=\inf_{\psi\in \Psi_{n,\k}}\sbr{- \log\del{\int\exp\del{-\xi_1\cL_n(\theta)}\d\Pi_{n,\k}(\theta|\psi)}}.
    \end{align}
The aim of this section is to demonstrate that the ESA approach achieves a near-oracle contraction rate in terms of the  excess risk defined above.
    
A technical challenge arises because employing data-dependent hyperparameter optimization breaks the fixed-prior assumption underlying the theoretical analysis in \cref{sec:theory}, thereby preventing a straightforward use of the previously derived results. To circumvent this issue, we introduce the notion of the R\'enyi-type covering number for the set $\Pi[\Psi]_{n,\k}$. For simplicity, the $(1+\rho)$-R\'enyi divergence from $\Pi_{n,\k}(\cdot|\psi_0)$ to  $\Pi_{n,\k}(\cdot|\psi_1)$ is denoted by $\renyihp(\psi_1,\psi_0):=\cD_{1+\rho}(\Pi_{n,\k}(\cdot|\psi_1), \Pi_{n,\k}(\cdot|\psi_0))$ for $\psi_0,\psi_1\in\Psi_{n,\k}$. We say that a finite set $\Psib$ is a $(1+\rho)$-R\'enyi-type hyperparameter cover of $\Pi[\Psi]_{n,\k}$ with precision $R>0$, if it satisfies that, for any $\psi\in\Psi_{n,\k}$,  there exists $\psib\in\Psib$ such that $\renyihp(\psi,\psib)\le R$. We denote by the minimal cardinality of such a R\'enyi-type cover by $\cN\del{R_{n,\k},\Pi[\Psi]_{n,\k}, \cD_{1+\rho} }$.

\begin{assumption}
\label{assume:hyperpar}
There exists an absolute constant $\rho>0$ such that, for any sufficiently large $n\in \bN$ and any $\k\in[M_n]$, the inequality
    \begin{align}
       \log \cN\del{R_{n,\k},\Pi[\Psi]_{n,\k}, \cD_{1+\rho} }\le U_{n,\k}
    \end{align}
holds for some constants $R_{n,\k}>0$ and $U_{n,\k}>0$ with $U_{n,\k}+\rho R_{n,\k}\le (2\barho)^{-1}\ccE_n(\k)$, where $\barho=(1+\rho)/\rho>1$ is the H\"older conjugate of $1+\rho$.
\end{assumption}

\begin{example}
Consider a $k$-dimensional parametric model and the family of Gaussian prior distributions with mean zero and diagonal covariance matrix $\Pi[\Psi]=\{\N(0,\diag(\psi)):\psi\in[\psi_{\min},\psi_{\max}]^k\}$ for this model. Then we have that
    \begin{align*}
        \cD_2^{\Pi}(\psi_1,\psi_0)
        \le \frac{1}{\psi_{\min}}\|\psi_1-\psi_0\|_2^2
        \le k\frac{\psi_{\max}-\psi_{\min}}{\psi_{\min}}\|\psi_1-\psi_0\|_\infty,
    \end{align*}
where for the derivation of the first inequality, for example, see the proof of Theorem 4.2 in \cite{ning2021spike}. This implies that
    \begin{align*}
        \log \cN\del{R,\Pi[\Psi], \cD_{2} }
         &\le \log \cN\del[2]{\frac{\psi_{\min}}{k(\psi_{\max}-\psi_{\min})}R,  [\psi_{\min},\psi_{\max}]^k, \|\cdot\|_\infty}\\
         &\le k\log\del[2]{1+\frac{k(\psi_{\min}-\psi_{\max})^2}{\psi_{\min}R}}.
    \end{align*}
Hence, if $R=(\log n)^{1/2}$, the above scales as $k\log(1+k/(\log n)^{1/2})$, which tends to zero as $n\to\infty$. Since the estimation error of a $k$-dimensional parameter space is typically proportional to $k\log n$, \cref{assume:hyperpar} would be satisfied.
\end{example}

For theoretical analysis of empirical Bayes procedures, it is crucial to control
additional variability due to hyperparameter optimization. Previous studies developed several technical tools. \cite{rousseau2017asymptotic} and \citet{donnet2018posterior} introduce a notion of ``locally maximized'' likelihood function, by which the fluctuation due to the change of the hyperparameter can be controlled. In contrast, \cite{zhang2020empirical} develops a ``prior-decomposition'' condition, under which the prior distribution is an aggregate of several ``fixed'' distributions where only the aggregation weights depend on the hyperparameter. This simple structure allows an easier theoretical analysis. While we follow a similar route as  \cite{zhang2020empirical} in the sense that theoretical assumptions are placed on the prior, not the likelihood function, our condition is more general as it does not impose any structural assumption on the prior. We note that the theoretical analysis in \cite{rousseau2017asymptotic,donnet2018posterior,zhang2020empirical} is developed for the standard likelihood, which relies on more sophisticated technical machinery, and it remains unclear whether our analysis can be extended to this standard likelihood setting.

The next assumption is an adaptation of \cref{assume:variational_family} tailored to the empirical Bayes framework.

\begin{assumption}
\label{assume:eb_family}
There exists an absolute constant $\xi_5>0$ such that
       \begin{align}
        \log\del{\int \exp(\xi_3\cL_n(\theta) )\d Q_{n,\k}^*(\theta)} &\le \xi_5\ccE_n(\k)\\
        \kl(Q_{n,\k}^*, \Pi_{n,\k}(\cdot|\psi^*_{n,\k}) &\le \xi_5 \ccE_n(\k)\\
       \Upsilon_{n,\k}(\psi^*_{n,\k}) &\le \xi_5\ccE_n(\k)
    \end{align}
for some distribution $Q_{n,\k}^*\in \cQ_{n,\k}$ and hyperparameter $\psi_{n,\k}^*\in\Psi_{n,\k}$ for any $\k\in[M_n]$ and any sufficiently large $n\in \bN$.
\end{assumption}

The next theorem illustrates that the ESA-EB procedure does not terminate before reaching the near-optimal model index defined as
   \begin{align}
        \ckopt(\tau)
        :=\inf\cbr{\k\in[M_n]:\ccE_{n}(\k)\le (1+\tau)\ccE_{n}(\k+1)}\wedge M_n
    \end{align}
with high probability, when $\tau>0$ is sufficiently large.

\begin{theorem}[Not-too-early-stopping of ESA-EB]
\label{thm:selection_eb}
Under \cref{assume:loss,assume:hyperpar,assume:eb_family}, there exist absolute constants $\ctau_*>0$ and $\tC_4>0$ such that
    \begin{align}
       \Pstar\del[1]{\csm\ge \ckopt(\ctau_*) +1}\ge 1-\ceta_n
    \end{align}
for any sufficiently large $n\in \bN$, where we let $\ceta_n:=\ckopt(\ctau_*)\exp\del[1]{-\tC_4 \ccE_n(\ckopt(\ctau_*))}.$
\end{theorem}

Based on the above theorem, we can derive the oracle contraction rate
  \begin{align}
        \ccE_n^*:=\min\cbr{\ccE_n(\k):\k\in[\ckopt(\ctau_*)+1]}
        =\ccE_n(\ckopt(\ctau_*))\wedge \ccE_n(\ckopt(\ctau_*)+1)
    \end{align}
of the ESA-EB variational posterior.

\begin{theorem}[Contraction rate of ESA-EB variational posterior]
\label{thm:contract_eb}
Assume that $\ceta_n\to0$ as $n\to\infty$. Then under the same assumption of \cref{thm:selection_eb},
    \begin{align}
       \Pstar\sbr{  \esavp\del{ \cL_n(\theta)\ge A_n( \ccE_n^*+\log M_n)}}\to0
    \end{align}
as $n\to\infty$ for any diverging sequence $(A_n)_{n\in\bN}$ with $A_n\to\infty$.
\end{theorem}

\section{Frequentist aggregation with early stopping}
\label{sec:erm}

In this section, we describe the ESA approach in the frequentist aggregation framework. We first consider ESA with an explicit penalty function, and then discuss two alternative methods based on sample-splitting and variational Bayes, respectively.

\subsection{Early-stopped aggregation of empirical risk minimizers}

For each $\k\in[M_n]$, we consider the empirical risk minimizer (ERM) over the model $\Theta_{n,\k}$
    \begin{align}
        \htheta_{n,\k}
        =\argmin_{\theta\in \Theta_{n,\k}}\ell_n(\theta)
    \end{align}
For comparing the performances of candidate models, we consider a penalized empirical risk 
    \begin{align}
        \mPER(\k):=\lambda\ell_n(  \htheta_{n,\k})+ H_{n,\k}
    \end{align}
with penalty term $H_{n,\k}>0$ imposed on the model $\Theta_{n,\k}$ to prevent overfitting. We aggregate the ERMs up to a data-driven model index $\usm$ given by
\begin{align}
        \usm=\inf\cbr{\k\in[M_n]\setminus\{1\}:
        \mPER(\k-1)<  \mPER(\k)}\wedge M_n.
    \end{align}
In other words, we stop the procedure until the penalized empirical risk increases. Then, building upon the model selection result with early stopping, we propose to use the ESA-ERM estimator defined as 
    \begin{align}
        \esaerm:=\sum_{\k=1}^{\usm}\omegafesa\htheta_{n,\k}
        \text{ with }
         \omegafesa
        :=\frac{\exp(-\mPER(\k))}{\sum_{\k'=1}^{\usm}\exp(-\mPER(\k'))}.
    \end{align}
The ESA-ERM estimator coincides with an exponentially weighted aggregate (EWA) of point estimators discussed in \cref{subsec:ewa}, with the crucial distinction that the aggregation is performed only up to a data-driven early stopping index.

\subsection{Not-too-early-stopped property and oracle inequality}

The aim of this section is to show that the proposed ESA-ERM estimator attains the almost minimal convergence rate of the excess risk. We make several assumptions for this purpose. 

\begin{assumption}
\label{assume:penalty}
There exists a metric  $\fh:\Theta_n\times \Theta_n\mapsto \R_{\ge0}$ and absolute constants $\kappa\ge1$ and $\xi_6>0$ such that
    \begin{align}
    \label{eq:assume_lip}
      \frac{1}{n} |\ell_n(\theta_1)-\ell_n(\theta_2)|\le \xi_6\fh^\kappa(\theta_1,\theta_2)
    \end{align}
almost surely for any $\theta_1,\theta_2\in \Theta_n$ and any sufficiently large $n\in \bN$. There exists an absolute constant $\xi_7\in(0,\xi_1/(2\xi_6))$  such that the penalty $H_{n,\sk}$ satisfies
        \begin{align}
         \label{eq:assume_entropy}
    \sup_{H>H_{n,\k}}\log \cN((\xi_7 n^{-1}H)^{1/\kappa}, \{\theta\in\Theta_{n,\k}: \cL_n(\theta)\le2 H\},\fh) &\le\frac{1}{4} H_{n,\k}
    \end{align}
 for any $\k\in[M_n]$ and any sufficiently large $n\in \bN$. 
\end{assumption}

The Lipschitzness assumption \eqref{eq:assume_lip} is satisfied by many practically used loss functions, including the cross-entropy, hinge, quantile and Huber losses; for more detailed illustrations, we refer to \citet{alquier2019estimation}. While the square loss function, most frequently used for regression, is not Lipschitz for unbounded outcomes, we can still apply our technical machinery to the square loss by using a truncation argument used in, for instance, \citet{bauer2019deep,ohn2022nonconvex,jiao2023deep,syring2023gibbs}.

The condition \eqref{eq:assume_entropy} is a version of a local complexity bound that has been a standard assumption in both frequentist and Bayesian theory, tracing back to the seminal contributions of Le Cam \citep{lecam1973convergence,lecam1975local}. 

\begin{example}
Consider a $k$-dimensional parametric model $\Theta(k, B):=\{\theta\in\R^k:\|\theta\|_p\le B\}$, with the $L_p$ boundedness. Then it is well known that
    \begin{align*}
      \log \cN(\zeta,\Theta(k, B),\|\cdot\|_p)\le k\log (3B/\zeta)
    \end{align*}
for any $\zeta>0$ \citep[e.g., Proposition C.2 of][]{ghosal2017fundamentals}. Therefore, if \eqref{eq:assume_lip} is satisfied with the $L_p$ distance metric, \eqref{eq:assume_entropy} is satisfied if we set the penalty term as $4k \log(3Bn/k)$. This suggests that we can use a penalty proportional to the parameter dimension of a candidate model for ESA.
\end{example}

We define the (frequentist) excess risk of a model $\k$ as
    \begin{align}
    \label{eq:excess_risk_erm}
        \ucE_n(\k):=\xi_1\min_{\theta\in\Theta_{n,\k}}\cL_n(\theta)+H_{n,\k},
    \end{align}
which, as will be demonstrated in the next proposition, becomes as an upper bound of the excess risk of the ERM estimator $\htheta_{n,\k}$. The proof of the proposition is quite standard, but we provide it for the sake of completeness.

\begin{proposition}
\label{prop:risk_bound_erm}
Under \cref{assume:loss,assume:penalty}, there exists an absolute constant $\tC_5>0$ such that
        \begin{align*}
            \Pstar\del{\cL_n(\htheta_{n,\k})\le \tC_5 \{\ucE_n(\k) +\log n\}}\ge 1-\frac{1}{n}
        \end{align*}
for any $\k\in[M_n]$ and any sufficiently large $n\in \bN$.
\end{proposition}

In accordance with the excess risk given in \eqref{eq:excess_risk_erm}, we define the near-optimal model as
  \begin{align}
        \ukopt(\tau)
        :=\inf\cbr{\k\in[M_n]:\ucE_{n}(\k)\le (1+\tau)\ucE_{n}(\k+1)}\wedge M_n
    \end{align}
for $\tau>0$. In the following theorem, we show that the ESA procedure does not terminate before reaching the near-optimal model index with high probability.

\begin{theorem}[Not-too-early-stopping of ESA-ERM]
\label{thm:selection_erm}
Under \cref{assume:loss,assume:penalty}, there exist absolute constants $\utau_*>0$ and $\tC_6>0$ such that
        \begin{align}
       \Pstar\del[1]{\usm\ge \ukopt(\utau_*)+1 }\ge 1-\ueta_n
    \end{align}
for any sufficiently large $n\in \bN$, where we let $\ueta_n:=\ukopt(\utau_*)\exp\del[1]{-\tC_6\ucE_n(\ukopt(\utau_*))}.$
\end{theorem}

\begin{remark}[Technical similarity in Bayesian and frequentist analysis]
The proof of \cref{thm:selection_erm} follows essentially the same strategy as that of \cref{thm:selection_vb},  despite the former being purely frequentist and the latter arising from a Bayesian variational formulation. This structural similarity highlights that the KL divergence term in the variational free energy in \eqref{eq:vfe} plays a role analogous to a complexity penalty in frequentist model selection.
\end{remark}

The next theorem illustrates that the ESA-ERM estimator enjoys a near-oracle rate given by
  \begin{align}
        \ucE_n^*:=\min\cbr{\ucE_n(\k):\k\in[\ukopt(\utau_*)+1]}
        =\ucE_n(\ukopt(\utau_*))\wedge \ucE_n(\ukopt(\utau_*)+1).
    \end{align}

\begin{theorem}[Oracle inequality for ESA-ERM estimator]
\label{thm:oracle_erm}
Assume that the function $\theta\mapsto \cL_n(\theta)$ is convex. 
Under the same assumption of \cref{thm:selection_erm}, there exists an absolute constant  $\tC_7>0$ such that
    \begin{align}
        \Pstar\del{ \cL_n(\esaerm)\le \tC_7\cbr[1]{\ucE_n^*+\log(M_nn)}}
        \ge 1-\ueta_n-n^{-1}
    \end{align}
for any sufficiently large $n\in \bN$.
\end{theorem}

\subsection{Practical alternatives}
\label{subsec:practical}

In practice, implementing the ESA procedure via penalized empirical risk minimization can be far from straightforward, as the appropriate penalty or calibration constant is rarely self-evident. While our theoretical results identify the correct asymptotic order of the penalty, they do not prescribe its exact non-asymptotic value, since the asymptotic guarantees remain unchanged under any constant multiplicative factor. To circumvent this issue, the ESA procedure can be implemented using established model selection criteria that serve as proxies for the excess risk. Common examples include validation error obtained via sample splitting, the Bayesian VFE, and information criteria such as AIC, with theoretical justification for the first two approaches provided in \cref{appendix:freq_etc}. These quantities effectively estimate the excess risk along the model ladder, enabling the construction of aggregation weights in a manner analogous to the theoretical penalized risk. We numerically investigate these practical alternatives for nonparametric regression in \cref{subsec:nonpar_reg}.

\section{Numerical studies}
\label{sec:experiment}

\subsection{Large-scale image classification with variational neural networks}

In this numerical experiment, we assess the performance of the proposed ESA approach on several image classification benchmarks. We consider standard grayscale benchmarks MNIST and Fashion-MNIST, each containing 10 classes as well as the color image datasets CIFAR-10 and CIFAR-100, which comprise 10 and 100 classes, respectively. Moreover, we also consider  Tiny-ImageNet, a substantially more demanding dataset with 200 distinct classes and higher-resolution images.

\paragraph{ESA and competitors} For each dataset, we consider a ladder of neural network models; detailed model specifications are provided in \cref{appendix:nn_models}.  When applying the proposed ESA procedure, we introduce the promoting parameter $\delta$ (see \cref{remark:earlier_termination}) to encourage earlier termination and further improve computational efficiency. ESA is compared with two baselines: full aggregation (FA), which trains all models in the ladder and then aggregate them with exponential weights, and best single model selection (MS), which selects a single model based on the final VFE.

\paragraph{Training details}  We employ the cross-entropy loss function for classification.  During training, the expected loss under the variational distribution is approximated using Monte Carlo sampling, while the KL divergence term is computed analytically for each Bayesian convolutional or linear layer. To stabilize optimization, the KL term is linearly warmed up during the early training epochs.

We place independent standard Gaussian priors on all network parameters. 
For the variational family, we adopt a Gaussian mean-field approximation of the form  $Q=\prod_j \mathcal{N}(\mu_j,\sigma_j^2),$ where the standard deviation is parameterized as $\sigma_j=\mathrm{softplus}(\rho_j)+10^{-5}.$
The variational parameters $(\mu_j,\rho_j)$ are optimized jointly  during training.  We note that all convolutional and linear layers are Bayesian, and the total KL term is computed by summing layer-wise KL divergences over all Bayesian layers. We tune the inverse temperature parameter according to the approach given in \cref{appendix:inv_temp}.

\paragraph{Results} \cref{tab:image-classification} reports the image classification results on five benchmark datasets. Overall, the proposed ESA achieves accuracy comparable to FA while substantially reducing computation time. For example, on MNIST and Fashion-MNIST, ESA reduces the computation time by more than three times compared to FA, while maintaining similar accuracy. On CIFAR-10, ESA achieves $85.55\%$ accuracy compared to $86.27\%$ for FA, while providing a $4.8\times$ speed-up. Similar trends are observed on CIFAR-100 and Tiny-ImageNet, where ESA consistently achieves competitive accuracy with significantly lower computational cost. Compared to the best single model selection, ESA further improves predictive accuracy while maintaining efficient training.

\begin{table}[t]
\centering
\caption{Comparison of ESA with full aggregation (FA) and best single model selection (MS) on image classification benchmarks. Accuracy and computation time are reported as mean $\pm$ standard deviation over three runs with random seeds 2025, 2026, and 2027. $\Delta$Acc is computed using mean accuracies (ESA minus FA). The computational time for MS is identical to that of FA, as the same set of models are trained.}
\label{tab:image-classification}

\begin{tabular}{llcccc}
\toprule
Dataset & Method & Accuracy & Time & $\Delta$Acc (ESA--FA) & Speed-up (FA/ESA) \\
\midrule

\multirow{3}{*}{MNIST}
& ESA  & $95.95\pm0.08$ & $4.51\pm0.14$ sec   & \multirow{3}{*}{$-1.46$} & \multirow{3}{*}{$3.14\times$} \\
& FA   & $97.41\pm0.08$ & $14.14\pm0.24$ sec  &                         & \\
& MS & $92.91\pm0.34$ &  $14.14\pm0.24$ sec               &                         & \\
\midrule

\multirow{3}{*}{Fashion-MNIST}
& ESA  & $87.54\pm0.57$ & $6.09\pm0.03$ sec   & \multirow{3}{*}{$-1.83$} & \multirow{3}{*}{$3.09\times$} \\
& FA   & $89.37\pm0.48$ & $18.84\pm0.07$ sec  &                         & \\
& MS & $85.55\pm0.34$ & $18.84\pm0.07$ sec    &                         & \\
\midrule

\multirow{3}{*}{CIFAR-10}
& ESA  & $85.55\pm0.74$ & $1.44\pm0.02$ min  & \multirow{3}{*}{$-0.72$} & \multirow{3}{*}{$4.79\times$} \\
& FA   & $86.27\pm1.09$ & $6.89\pm0.05$ min  &                         & \\
& MS & $84.21\pm0.28$ &  $6.89\pm0.05$ min                &                         & \\
\midrule

\multirow{3}{*}{CIFAR-100}
& ESA  & $69.92\pm1.06$ & $24.78\pm12.67$ min & \multirow{3}{*}{$-1.29$} & \multirow{3}{*}{$1.64\times$} \\
& FA   & $71.21\pm0.25$ & $40.60\pm1.45$ min  &                          & \\
& MS & $65.82\pm0.25$ &  $40.60\pm1.45$ min             &                          & \\
\midrule

\multirow{3}{*}{Tiny-ImageNet}
& ESA  & $50.50\pm2.87$ & $60.97\pm27.40$ min & \multirow{3}{*}{$-0.58$} & \multirow{3}{*}{$1.44\times$} \\
& FA   & $51.08\pm2.17$ & $88.05\pm8.12$ min  &                          & \\
& MS & $46.40\pm2.69$ &  $88.05\pm8.12$ min     &                          & \\
\bottomrule
\end{tabular}

\end{table}

\subsection{Clustering with the unknown number of clusters}

In this experiment, we apply the proposed ESA approach to variational Bayesian Gaussian mixture model (VGMM) for clustering. 

\paragraph{VGMM-ESA}
We set $M = 10$ as the maximum number of mixture components. For each $\k\in[M]$, the  samples $X_1,\ldots,X_n\in\R^d$ are modeled using a finite Bayesian Gaussian mixture model with latent allocations $Z_i\in[\k]$:
\begin{align*}
    Z_i &\iidsim \texttt{Categorical}(\pi), \\
    X_i \mid (Z_i = j) &\indsim \N(\mu_j, \Lambda_j^{-1}).
\end{align*}
We place a conjugate Normal–Wishart prior independently on the means and covariances of the Gaussian components, and a Dirichlet prior for the mean of the categorical distribution, i.e., 
    \begin{align*}
        \Lambda_k &\iidsim \texttt{Wishart}(W_0,\nu_0),\\
\mu_k \mid \Lambda_k &\indsim \N \big(m_0,(\beta_0\Lambda_k)^{-1}\big),\\
    \pi &\sim \texttt{Dirichlet}((\alpha_0,\dots, \alpha_0)).
    \end{align*}
We set the hyperparameters as  $m_0 = n^{-1}\sum_{i=1}^nX_i$, $W_0 = \{n^{-1}\sum_{i=1}^n(X_i-m_0)(X_i-m_0)^\top\}^{-1}$, $\nu_0 = d$, $\beta_0 = 1$ and $\alpha_0 = 1$. For every model $\k$, coordinate-ascent variational Bayes procedure  is employed under a mean-field variational family assumption. A collection of simpler candidate variational posteriors are then aggregated using the ESA approach. 

\paragraph{Baselines}
We compare the proposed VGMM-ESA method with several VGMM-based approaches, including full aggregation of VGMMs (VGMM-FA), VGMM selected based on the VFE (VGMM-MS), VGMM with the true number of components (VGMM-oracle),  and VGMM with an infinite number of components modeled using a Dirichlet process prior (inf-VGMM-DP). We also benchmark classical clustering methods, including spectral clustering, K-means, and bisecting K-means. For each method, the number of components is determined in a data-driven way, searched over $\{2,\ldots,10\}$. Specifically, we employ the eigengap criterion for spectral clustering and the elbow method for K-means and bisecting K-means. For reference, we additionally report oracle results obtained using the true number of components. All methods, including the VGMM-based procedures, are implemented using \texttt{scikit-learn}.

\paragraph{Simulation setup} Synthetic samples of  size $n=500$ are generated from the two data-generating mechanisms described below:
\begin{enumerate}
    \item (Setting A: Heterogeneous Gaussians) We generate $X_i \iidsim \sum_{k=1}^3 \pi_k\,\N(\mu_k,\Sigma_k)$ with $\pi=(0.35,0.5,0.15)$ and
\begin{align*}
\mu_1 &= (-4,0)^\top,\quad \mu_2=(0,0)^\top,\quad \mu_3=(4,0)^\top,\\
\Sigma_1 &= \begin{pmatrix} 2 & 0 \\ 0 & 1 \end{pmatrix},\quad
\Sigma_2 = R_{\pi/3} \begin{pmatrix} 2 & 0 \\ 0 & 0.2 \end{pmatrix} R_{\pi/3}^\top, \quad
\Sigma_3 = 0.15\,I_2,
\end{align*}
where $R_\phi \in \mathbb{R}^{2\times 2}$ denotes the rotation matrix corresponding to the angle $\phi$.

    \item (Setting B: Interleaving semicircles) We  draw $Z_i \iidsim \Ber(0.5)$ and $\phi_i \iidsim \Unif(0,\pi)$, and then generate
\begin{align*}
X_i |Z_i,\phi_i\indsim \N \big(\mu_{Z_i}(\phi_i),\, 0.15 \,I_2\big)
\end{align*}
where $\mu_0$ represents the upper unit semicircle and $\mu_1$ does its shifted reflection, namely, $\mu_0(\phi)=(\cos\phi,\sin\phi)^\top$ and $\mu_1(\phi)=(0.8-\cos\phi,0.5-\sin\phi)^\top$. In this case, the Gaussian mixture model misspecifies the true model.
\end{enumerate}

\paragraph{Results}
We evaluate clustering performance using adjusted rand index (ARI), adjusted mutual information (AMI), and normalized mutual information (NMI) as well as the computation time. \cref{tab:clustering-results} presents the results over 50 simulation replicates. See also \cref{fig:clustering_visual} for a visual illustration of the clustering results. Across both correctly specified and misspecified settings, the VGMM-based model selection and aggregation methods outperform the competing baselines, while VGMM-ESA substantially reduces computation cost compared with the global methods VGMM-FA and VGMM-MS.

\begin{table}[t]
\centering
\caption{Comparison of clustering methods under Settings A and B. Clustering performance (ARI, AMI, NMI) and computation time are reported as mean $\pm$ standard deviation over 50 replicates. Bold numbers indicate the best performance in each column (ties are possible). ``$\uparrow$'' and ``$\downarrow$'' indicate that higher and lower values are preferred, respectively.}
\label{tab:clustering-results}
\small
\setlength{\tabcolsep}{6pt}
\begin{tabular}{l l c c c c}
\toprule
Dataset & Method & ARI $(\uparrow)$ & AMI $(\uparrow)$ & NMI $(\uparrow)$ & Time (sec, $\downarrow$) \\
\midrule

\multirow{11}{*}{Setting A}
& VGMM-ESA & \textbf{0.945 $\pm$ 0.020} & \textbf{0.927 $\pm$ 0.021} & \textbf{0.927 $\pm$ 0.021} & 0.969 $\pm$ 0.287 \\
& VGMM-FA & \textbf{0.945 $\pm$ 0.020} & \textbf{0.927 $\pm$ 0.021} & \textbf{0.927 $\pm$ 0.021} & 8.111 $\pm$ 1.547 \\
& VGMM-MS & \textbf{0.945 $\pm$ 0.020} & \textbf{0.927 $\pm$ 0.021} & \textbf{0.927 $\pm$ 0.021} & 7.612 $\pm$ 1.372 \\
& VGMM-Oracle & \textbf{0.945 $\pm$ 0.020} & \textbf{0.927 $\pm$ 0.021} & \textbf{0.927 $\pm$ 0.021} & 0.061 $\pm$ 0.018 \\
& inf-VGMM-DP & 0.945 $\pm$ 0.020 & 0.926 $\pm$ 0.021 & 0.927 $\pm$ 0.021 & 1.153 $\pm$ 0.311 \\
& Spectral-Eigengap & 0.875 $\pm$ 0.095 & 0.875 $\pm$ 0.061 & 0.875 $\pm$ 0.060 & 0.033 $\pm$ 0.007 \\
& Spectral-Oracle & 0.899 $\pm$ 0.056 & 0.886 $\pm$ 0.050 & 0.887 $\pm$ 0.050 & 0.019 $\pm$ 0.005 \\
& K-means-Elbow & 0.601 $\pm$ 0.030 & 0.718 $\pm$ 0.023 & 0.719 $\pm$ 0.023 & 0.203 $\pm$ 0.057 \\
& K-means-Oracle & 0.854 $\pm$ 0.029 & 0.836 $\pm$ 0.028 & 0.837 $\pm$ 0.028 & 0.011 $\pm$ 0.004 \\
& Bisecting-K-means-Elbow & 0.546 $\pm$ 0.069 & 0.681 $\pm$ 0.037 & 0.682 $\pm$ 0.037 & 0.023 $\pm$ 0.008 \\
& Bisecting-K-means-Oracle & 0.875 $\pm$ 0.034 & 0.844 $\pm$ 0.032 & 0.845 $\pm$ 0.032 & \textbf{0.001 $\pm$ 0.001} \\

\midrule

\multirow{11}{*}{Setting B}
& VGMM-ESA & 0.404 $\pm$ 0.026 & \textbf{0.502 $\pm$ 0.033} & \textbf{0.504 $\pm$ 0.033} & 0.993 $\pm$ 0.502 \\
& VGMM-FA & 0.404 $\pm$ 0.026 & \textbf{0.502 $\pm$ 0.033} & \textbf{0.504 $\pm$ 0.033} & 5.385 $\pm$ 2.794 \\
& VGMM-MS & 0.404 $\pm$ 0.026 & \textbf{0.502 $\pm$ 0.033} & \textbf{0.504 $\pm$ 0.033} & 5.009 $\pm$ 2.465 \\
& VGMM-Oracle & \textbf{0.427 $\pm$ 0.055} & 0.339 $\pm$ 0.046 & 0.340 $\pm$ 0.046 & 0.127 $\pm$ 0.185 \\
& inf-VGMM-DP & 0.346 $\pm$ 0.028 & 0.483 $\pm$ 0.028 & 0.485 $\pm$ 0.028 & 0.873 $\pm$ 0.442 \\
& Spectral-Eigengap & 0.258 $\pm$ 0.058 & 0.199 $\pm$ 0.046 & 0.200 $\pm$ 0.046 & 0.024 $\pm$ 0.007 \\
& Spectral-Oracle & 0.258 $\pm$ 0.058 & 0.199 $\pm$ 0.046 & 0.200 $\pm$ 0.046 & 0.014 $\pm$ 0.005 \\
& K-means-Elbow & 0.285 $\pm$ 0.034 & 0.360 $\pm$ 0.048 & 0.362 $\pm$ 0.048 & 0.128 $\pm$ 0.052 \\
& K-means-Oracle & 0.193 $\pm$ 0.050 & 0.145 $\pm$ 0.039 & 0.146 $\pm$ 0.039 & 0.006 $\pm$ 0.003 \\
& Bisecting-K-means-Elbow & 0.292 $\pm$ 0.036 & 0.441 $\pm$ 0.064 & 0.443 $\pm$ 0.063 & 0.013 $\pm$ 0.007 \\
& Bisecting-K-means-Oracle & 0.194 $\pm$ 0.049 & 0.146 $\pm$ 0.038 & 0.147 $\pm$ 0.038 & \textbf{0.001 $\pm$ 0.000} \\

\bottomrule
\end{tabular}
\end{table}

\begin{figure}[t]
    \centering
    \includegraphics[scale=0.35]{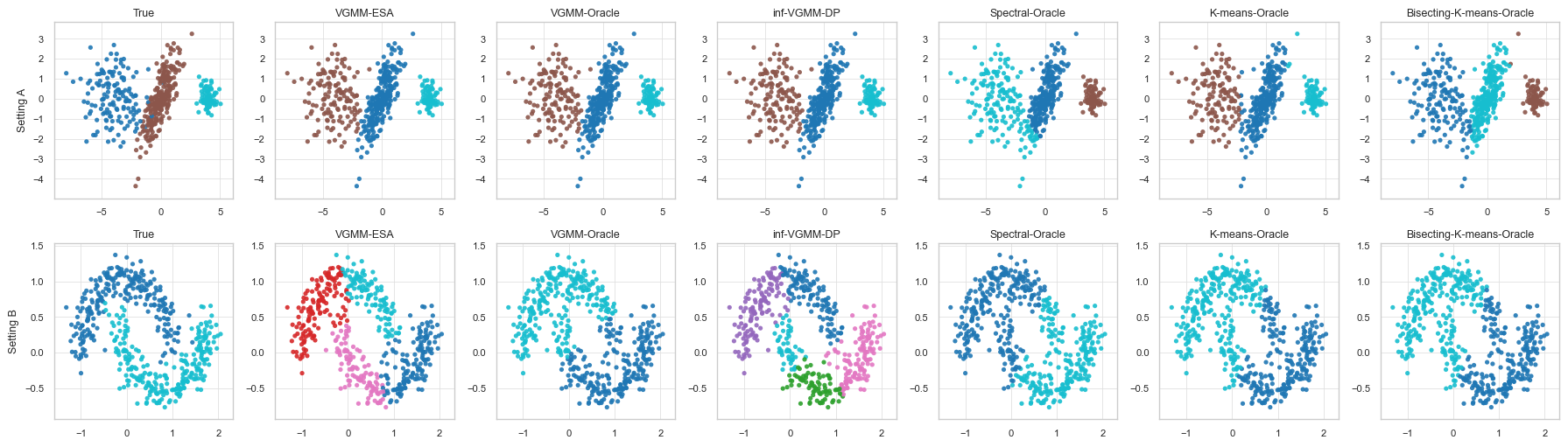}
    \caption{Clustering visualizations for Setting A (upper panels) and Setting B (lower panels)}
    \label{fig:clustering_visual}
\end{figure}

\subsection{Empirical Bayes inference for sparse linear regression}

In this experiment, we consider a  sparse linear regression model with unknown noise variance.  We evaluate an empirical Bayes ESA approach built on the sum of single effects (SuSiE) regression framework \citep{wang2020simple}, which optimizes the prior variance parameters during training.

\paragraph{SuSiE-ESA} We consider the linear model $ Y|(\theta,\sigma^2) \sim \N(X\theta,\sigma^2 I_n)$, where $Y\in\R^n$ is the centered response and $X\in\R^{n\times p}$ is the centered design matrix. Given the number of single effects $\sk$, the SuSiE approach models the regression coefficient as
\begin{align*}
\theta = \sum_{h=1}^{\sk} \theta_{(h)}, \quad
\theta_{(h)} = \gamma_h b_h ,
\end{align*}
with the prior distributions $\gamma_h \iidsim \texttt{Categorical}(1/p,\dots,1/p)$ and  $b_h \iidsim \N(0,\psi_j).$ Under this specification, each component $\theta_{(h)}$ has at most one nonzero coordinate.

Iterative Bayesian stepwise selection (IBSS) is used for fitting the SuSiE method, which is coordinate-ascent variational Bayes assuming a mean-field variational family, where each coordinate update reduces to a single-effect regression (SER) fit to the current expected residuals. We obtained SuSiE fits via the \texttt{susieR} R-package. The  noise variance $\sigma^2$ and the prior variances $\psi_h$ are estimated internally in an  empirical Bayes manner. We set $M=10$ as the maximum number of single effects and consider the model ladder $\sk\in[M]$. Along this ladder, SuSiE-ESA computes variational fits sequentially and aggregates them up to the data-dependent stopping index. Throughout, we use the promoting parameter of $\delta = 10^{-4}$ encourage earlier termination (see \cref{remark:earlier_termination}).

\paragraph{Baselines}
We compare the proposed SuSiE-ESA method with several SuSiE-based approaches, including full aggregation along the model ladder (SuSiE-FA), selection based on the VFE (SuSiE-MS), and an oracle SuSiE fit using the true number of nonzero coefficients (SuSiE-oracle). We also consider several Bayesian sparse regression methods. These include variational Bayes with a point-mass spike and Laplace slab prior (implemented in the \texttt{sparsevb} R package; \citep{ray2021variational}) and with a point-mass spike and Gaussian slab prior (\texttt{varbvs}; \citep{carbonetto2012scalable}), as well as a maximum a posteriori estimator with  a Laplace spike and Laplace slab prior (\texttt{SSLASSO}; \citep{rovckova2018spike}). For these methods, we first estimate the noise variance using the procedure implemented in the \texttt{selectiveInference} R package \citep{reid2016study} and then apply each method to the rescaled data $(\tilde X,\tilde Y)=(X/\hat\varsigma, Y/\hat\varsigma)$.  Finally,  widely used frequentist approaches, LASSO and SCAD, are also considered.

\paragraph{Simulation setup}
For each replicate, we generate input vectors as $X_i \iidsim \N(0,I_p)$. We then generate outcomes from the linear model $Y_i=X_i^\top\theta^\star+Z_i$ by drawing a noise variable $Z_i$, where the ground truth $\theta^\star$ is set to 
    \begin{align*}
    \theta^\star = (1,2,3,0,\ldots,0)^\top,
    \end{align*}
which has $s=3$ nonzero entries. The noise variables $Z_1,\dots,Z_n$ are sampled i.i.d. from either a Gaussian distribution with mean 0 and standard deviation 1.5 or a Cauchy distribution with location 0 and scale 0.5.
Throughout the experiments we fix $n=100$ and $p=1,000$. 

\paragraph{Results}
We evaluate estimation accuracy using the $L_2$ error $\|\hat\theta-\theta^{\star}\|_2$, where $\htheta$ represents either a point estimator or a variational posterior mean. For Bayesian approaches, we compute the posterior inclusion probability (PIP) for each variable, denoted by $\mathrm{PIP}_j$ for the $j$-th variable, and define the selected variable set as $\hS=\{j\in[p]:\mathrm{PIP}_j>0.5\}$. We assess variable selection performance using the true positive rate (TPR) and the false discovery rate (FDR) defined as
\begin{align*}
\mathrm{TPR}=\frac{|\hS\cap S^{\star}|}{|S^{\star}|}, \quad
\mathrm{FDR}=\frac{|\hS\setminus S^{\star}|}{\max(1,|\hS|)},
\end{align*}
where $S^\star:=\{j\in[p]:\theta^\star_j\neq 0\}$.

\cref{tab:susie-results} summarizes $L_2$ error, FDR, TPR and computation time over 30 replicates. On both settings, SuSiE-based model selection and aggregation methods outperform the competing baselines, while SuSiE-ESA substantially reduces computation cost compared with the global methods SuSiE-FA and SuSiE-MS.

\begin{table}[t]
\centering
\caption{Comparison of sparse regression methods. Evaluation metrics are reported as mean $\pm$ standard deviation over 30 replicates. Bold numbers indicate the best performance in each column (ties are possible). ``$\uparrow$'' and ``$\downarrow$'' indicate that higher and lower values are preferred, respectively.}
\label{tab:susie-results}
\small
\setlength{\tabcolsep}{6pt}
\begin{tabular}{l l c c c c}
\toprule
Noise & Method & $L_2$ error $(\downarrow)$ & TPR $(\uparrow)$ & FDR $(\downarrow)$ & Time (sec, $\downarrow$) \\
\midrule

\multirow{9}{*}{Gaussian}
& SuSiE-ESA    & $0.293 \pm 0.139$ & \textbf{1.000 $\pm$ 0.000} & \textbf{0.000 $\pm$ 0.000} & $0.406 \pm 0.457$ \\
& SuSiE-FA     & $0.287 \pm 0.129$ & \textbf{1.000 $\pm$ 0.000} & \textbf{0.000 $\pm$ 0.000} & $1.611 \pm 0.274$ \\
& SuSiE-MS     & \textbf{0.284 $\pm$ 0.126} & \textbf{1.000 $\pm$ 0.000} & \textbf{0.000 $\pm$ 0.000} & $1.611 \pm 0.274$ \\
& SuSiE-Oracle & \textbf{0.284 $\pm$ 0.126} & \textbf{1.000 $\pm$ 0.000} & \textbf{0.000 $\pm$ 0.000} & \textbf{0.062 $\pm$ 0.014} \\
& varbvs       & $0.300 \pm 0.118$ & \textbf{1.000 $\pm$ 0.000} & $0.008 \pm 0.046$ & $0.240 \pm 0.084$ \\
& sparsevb     & $0.368 \pm 0.270$ & $0.978 \pm 0.085$ & $0.049 \pm 0.151$ & $0.192 \pm 0.020$ \\
& SSLASSO      & $0.463 \pm 0.366$ & $0.911 \pm 0.150$ & \textbf{0.000 $\pm$ 0.000} & $0.203 \pm 0.028$ \\
& LASSO        & $0.833 \pm 0.216$ & \textbf{1.000 $\pm$ 0.000} & $0.785 \pm 0.175$ & $0.063 \pm 0.008$ \\
& SCAD         & $0.485 \pm 0.178$ & \textbf{1.000 $\pm$ 0.000} & $0.660 \pm 0.262$ & $0.099 \pm 0.015$ \\

\midrule

\multirow{9}{*}{Cauchy}
& SuSiE-ESA    & $2.447 \pm 1.369$ & $0.356 \pm 0.391$ & $0.011 \pm 0.061$ & $0.394 \pm 0.483$ \\
& SuSiE-FA     & $2.439 \pm 1.380$ & $0.356 \pm 0.391$ & $0.011 \pm 0.061$ & $1.599 \pm 0.780$ \\
& SuSiE-MS     & $2.444 \pm 1.391$ & $0.356 \pm 0.391$ & $0.011 \pm 0.061$ & $1.599 \pm 0.780$ \\
& SuSiE-Oracle & \textbf{2.435 $\pm$ 1.379} & $0.356 \pm 0.391$ & $0.011 \pm 0.061$ & $0.073 \pm 0.032$ \\
& varbvs       & $2.482 \pm 1.359$ & $0.344 \pm 0.396$ & $0.022 \pm 0.085$ & $0.221 \pm 0.113$ \\
& sparsevb     & $2.571 \pm 1.307$ & $0.333 \pm 0.382$ & $0.072 \pm 0.239$ & $0.210 \pm 0.049$ \\
& SSLASSO      & $2.849 \pm 1.266$ & $0.233 \pm 0.341$ & \textbf{0.000 $\pm$ 0.000} & $0.129 \pm 0.043$ \\
& LASSO        & $2.681 \pm 1.104$ & \textbf{0.500 $\pm$ 0.469} & $0.500 \pm 0.425$ & \textbf{0.070 $\pm$ 0.019} \\
& SCAD         & $2.555 \pm 1.242$ & \textbf{0.500 $\pm$ 0.461} & $0.545 \pm 0.408$ & $0.269 \pm 0.092$ \\

\bottomrule
\end{tabular}
\end{table}

\subsection{Hyperparameter tuning for frequentist nonparametric regression models}
\label{subsec:nonpar_reg}

We compare four hyperparameter tuning strategies in nonparametric regression problems: early-stopped aggregation (ESA), full aggregation (FA),  best model selection (MS), and $5$-fold cross-validation (CV). For each method, we report the test root mean squared error (RMSE) and the wall-clock time for hyperparameter tuning.  We use four benchmark regression data sets: Boston Housing, Concrete, Energy Efficiency, and Wine Quality.
For each data set, we randomly split the data into training and test sets,
and repeat the experiment for 30 independent runs with different random seeds.

\paragraph{Model classes and ladders}

We consider three widely used regression models: random forest (RF), gradient boosted trees implemented via XGBoost (XGB), and $k$-nearest neighbor regression (kNN). For each model class, we construct a ladder of models with increasing complexity and denote the resulting predictors by $\{\hf_{\sk}: \sk=1,\dots,M\}$. 
For each trained predictor, we compute a model selection criterion.  In ESA, the models are fitted sequentially along the ladder and the procedure is stopped once the criterion ceases to decrease.  The final predictor is obtained by aggregating the fitted models up to the stopping index using exponential weights based on the negative criterion values. Full aggregation also uses such exponential weights but computed over all models in the ladder. The model selection approach uses the predictor that minimizes the model selection criterion for prediction.

Let $\{(X_i, Y_i)\}_{i\in[n]}$ denote the training sample, which consists of $n$-pairs of input $X_i$ and output $Y_i$.  For random forests, we specify its model complexity by the maximum depth of the trees.  We consider the ladder $\texttt{max\_depth} \in \{2,4,8,12,16,32\}$.  We use the out-of-bag (OOB) sum of squared errors for the model selection criterion, which is defined as
\begin{align*}
\text{OOBE}(\sk)
= \sum_{i=1}^n \big(Y_i - \widehat f^{\mathrm{OOB}}_{\sk}(X_i)\bigr)^2,
\end{align*}
where $\widehat f^{\mathrm{OOB}}_{\sk}(X_i)$ denotes the OOB prediction for observation $i$.  For XGBoost regression, we consider the model ladder with different values $\texttt{max\_depth} \in \{2,4,6,8,12\}$ of  the maximum tree depth. For each depth level, we use the validation sum of squares as the model selection criterion, which is discussed in \cref{subsec:sample_split}. For $k$-nearest neighbor regression,  we consider the ladder $\texttt{num\_nbr}\in \{1,3,5,10,20,40,80,160\}$ for the number of neighbors, We use the corrected Akaike information criterion (AICc; \citep{sugiura1978further}) as the model selection criterion, which is given as
\begin{align*}
    \text{AICc}(\sk) 
= n\log\del{\text{SSE}(\sk)/n} + 2\mathrm{df}(\sk) + \frac{2\mathrm{df}(\sk)(\mathrm{df}(\sk)+1)}{n-\mathrm{df}(\sk)-1}.
\end{align*}
where $ \text{SSE}(\sk):=\sum_{i=1}^n (Y_i - \hf_{\sk}(X_i))^2$ denotes the training sum of squared errors and $\mathrm{df}(\sk)$  the effective degrees of freedom of model $\sk$, which we approximate by the sample size divided by the number of neighbors.

\paragraph{Results}

\cref{tab:reg_rmse,tab:reg_time} summarize the predictive performance and hyperparameter tuning time of the competing methods across the benchmark data sets. Overall, ESA achieves predictive accuracy comparable to that of the other tuning strategies. In most dataset and model combinations, the test RMSE obtained by ESA is nearly identical to those of full aggregation and the best single model selected by a model selection criterion, and is also close to the performance obtained by cross-validation. This indicates that early stopping along the model ladder does not noticeably degrade predictive performance.

Meanwhile, ESA substantially reduces the computational cost of hyperparameter tuning as it often avoids fitting the entire ladder of models. As a result, ESA consistently requires less tuning time than cross-validation, which must repeatedly train models across multiple folds. In several cases, the reduction in tuning time is particularly pronounced, especially for the tree-based models such as random forests and XGBoost.

Taken together, these results suggest that ESA provides an effective compromise between predictive accuracy and computational efficiency. While maintaining a level of accuracy comparable to standard tuning strategies, ESA can substantially reduce the cost of hyperparameter tuning, making it particularly attractive in settings where model training is computationally expensive.

\begin{table}[t]
\centering
\caption{Comparison of hyperparameter tuning methods for regression benchmarks. Test RMSE are reported as mean $\pm$ standard deviation over 30 random training-test splits. Bold numbers indicate the best performance in each row.}
\label{tab:reg_rmse}
\begin{tabular}{llcccc}
\toprule
Dataset & Model & ESA & FA & MS & CV \\
\midrule
\multirow{3}{*}{Boston}
& RF  & 3.550 $\pm$ 0.556 & 3.549 $\pm$ 0.556 & \textbf{3.538 $\pm$ 0.555} & 3.539 $\pm$ 0.558 \\
& XGB & \textbf{3.365 $\pm$ 0.530} & 3.407 $\pm$ 0.524 & 3.396 $\pm$ 0.544 & 3.371 $\pm$ 0.518 \\
& kNN & 5.123 $\pm$ 0.640 & 5.123 $\pm$ 0.640 & 5.123 $\pm$ 0.640 & \textbf{4.942 $\pm$ 0.735} \\

\midrule

\multirow{3}{*}{Concrete}
& RF  & 5.227 $\pm$ 0.443 & 5.227 $\pm$ 0.443 & \textbf{5.220 $\pm$ 0.445} & 5.220 $\pm$ 0.444 \\
& XGB & \textbf{4.387 $\pm$ 0.446} & 4.447 $\pm$ 0.462 & 4.457 $\pm$ 0.471 & 4.391 $\pm$ 0.439 \\
& kNN & 8.938 $\pm$ 0.567 & 8.938 $\pm$ 0.567 & 8.938 $\pm$ 0.567 & \textbf{8.936 $\pm$ 0.628} \\

\midrule

\multirow{3}{*}{Energy}
& RF  & 0.980 $\pm$ 0.391 & 0.978 $\pm$ 0.392 & 0.973 $\pm$ 0.394 & \textbf{0.973 $\pm$ 0.395} \\
& XGB & \textbf{0.531 $\pm$ 0.045} & 0.532 $\pm$ 0.054 & 0.543 $\pm$ 0.055 & 0.546 $\pm$ 0.061 \\
& kNN & 1.956 $\pm$ 0.243 & 1.956 $\pm$ 0.243 & 1.956 $\pm$ 0.243 & \textbf{1.708 $\pm$ 0.235} \\

\midrule

\multirow{3}{*}{Wine}
& RF  & 0.588 $\pm$ 0.027 & 0.588 $\pm$ 0.027 & 0.573 $\pm$ 0.027 & \textbf{0.573 $\pm$ 0.027} \\
& XGB & 0.588 $\pm$ 0.026 & \textbf{0.585 $\pm$ 0.025} & 0.589 $\pm$ 0.025 & 0.588 $\pm$ 0.025 \\
& kNN & 0.652 $\pm$ 0.028 & 0.652 $\pm$ 0.028 & 0.652 $\pm$ 0.028 & \textbf{0.650 $\pm$ 0.028} \\

\bottomrule
\end{tabular}
\end{table}

\begin{table}[t]
\centering
\caption{Hyperparameter tuning time in seconds, reported as mean $\pm$ standard deviation over 20 random training-test splits. Bold numbers indicate the minimum running time in each row (ties are possible).  Since MS trains the same set of models as FA, their computational times are identical, and we omit the MS running time.}
\label{tab:reg_time}
\begin{tabular}{llccc}
\toprule
Dataset & Model & ESA & FA & CV \\
\midrule

\multirow{3}{*}{Boston}
& RF  & \textbf{2.879 $\pm$ 0.262} & 3.066 $\pm$ 0.061 & 13.765 $\pm$ 0.372 \\
& XGB & \textbf{0.397 $\pm$ 0.162} & 1.088 $\pm$ 0.320 & 5.172 $\pm$ 0.647 \\
& kNN & \textbf{0.008 $\pm$ 0.001} & 0.034 $\pm$ 0.002 & 0.068 $\pm$ 0.003 \\

\midrule

\multirow{3}{*}{Concrete}
& RF  & \textbf{3.134 $\pm$ 0.040} & \textbf{3.134 $\pm$ 0.040} & 13.809 $\pm$ 0.145 \\
& XGB & \textbf{0.821 $\pm$ 0.258} & 1.523 $\pm$ 0.335 & 7.988 $\pm$ 0.934 \\
& kNN & \textbf{0.029 $\pm$ 0.002} & 0.129 $\pm$ 0.004 & 0.116 $\pm$ 0.002 \\

\midrule

\multirow{3}{*}{Energy}
& RF  & \textbf{2.664 $\pm$ 0.378} & 3.116 $\pm$ 0.024 & 14.569 $\pm$ 0.197 \\
& XGB & \textbf{0.679 $\pm$ 0.274} & 1.078 $\pm$ 0.299 & 5.250 $\pm$ 0.549 \\
& kNN & \textbf{0.010 $\pm$ 0.002} & 0.047 $\pm$ 0.001 & 0.083 $\pm$ 0.001 \\

\midrule

\multirow{3}{*}{Wine}
& RF  & \textbf{3.211 $\pm$ 0.016} & \textbf{3.211 $\pm$ 0.016} & 13.640 $\pm$ 0.072 \\
& XGB & \textbf{0.571 $\pm$ 0.242} & 0.712 $\pm$ 0.185 & 3.690 $\pm$ 0.498 \\
& kNN & \textbf{0.237 $\pm$ 0.037} & 0.300 $\pm$ 0.002 & 0.310 $\pm$ 0.002 \\
\bottomrule
\end{tabular}
\end{table}

\section{Conclusion and future work}
\label{sec:conclusion}

In this paper, we introduced early-stopped aggregation (ESA), a versatile framework designed to overcome the computational bottlenecks of traditional model aggregation by stopping the evaluation process once additional increases in model complexity no longer yield improvement. By monitoring local changes in a data-driven criterion-such as the variational free energy or penalized empirical risk along an ordered ladder of candidate models, ESA avoids the exhaustive computation typically required for adaptive inference. Our theoretical investigation establishes that ESA achieves adaptive optimal contraction rates across variational Bayes, empirical Bayes, and frequentist paradigms, demonstrating that computational efficiency does not have to come at the cost of statistical optimality. Furthermore, extensive numerical studies across various tasks consistently show that ESA provides significant speed-ups while maintaining predictive performance comparable to full aggregation as well as other baseline methods.

There are several interesting directions for future research.  First, it would be desirable to establish sharper optimality results for the ESA framework.  The current theory provides near-oracle guarantees, but these bounds are not sharp because the near-optimal model index is defined up to a potentially large multiplicative constant.  Although \cref{prop:ushaped_risk} shows that ESA achieves exact adaptive optimality under an idealized U-shaped risk condition, a general result matching the optimal guarantees of full aggregation remains an open problem.

Another promising direction is the extension of ESA to multiple model indices. 
The present work focuses on a ``one-dimensional'' ladder of nested models, whereas many modern learning problems involve multidimensional model complexity.  For example, in random forest algorithms, we may need to simultaneously tune the number of trees and the complexity of individual weak learners, such as tree depth.  Developing early-stopping procedures capable of navigating such multidimensional grids or graphs while retaining theoretical guarantees would substantially broaden the applicability of ESA to more complex model selection problems.

\bibliographystyle{plainnat}
\bibliography{_references}

\newpage

\setcounter{page}{1}
\renewcommand{\thepage}{S-\arabic{page}}

\begin{appendices}
\crefalias{section}{appendix}
\crefalias{subsection}{appendix}
\crefalias{subsubsection}{appendix}

	\begin{center}
		{\large \textsc{Supplement to ``Early-stopped aggregation: Adaptive inference with computational efficiency''}} \\
		\medskip 
		{Ilsang Ohn, Shitao Fan, Jungbin Jun and Lizhen Lin}
		\medskip
	\end{center}
		 
\setcounter{tocdepth}{3}
\tableofcontents

\section{Other frequentist aggregation approaches}
\label{appendix:freq_etc}

\subsection{Sample-splitting approach}
\label{subsec:sample_split}

As we mentioned in \cref{subsec:practical}, implementing the ESA procedure for empirical risk minimization requires careful calibration of the penalty term in practice.  To circumvent this issue, we can adopt a sample-splitting strategy, which has been used in earlier work on frequentist aggregation \citepS{lecue2007optimal,rigollet2007linear, goldenshluger2009universal}. Specifically, we partition the full sample into a training sample and an ``aggregation'' sample. Candidate estimators are constructed using the training sample, while the aggregation weights are computed based on the aggregation sample. When the training and aggregation samples are independent, the candidate estimators can be treated as non-random objects when analyzing the aggregation step, by conditioning on the training sample. Accordingly, all theoretical guarantees are stated with respect to the distribution of the aggregation sample. Throughout this subsection, we use the same notation as before, now interpreted with respect to the aggregation sample. In particular, $X^{(n)}$ and $\Pstar$ denote the aggregation sample and its distribution, and $\ell_n$ denotes the loss function evaluated on the aggregation sample.

Let $\ttheta_{n,\k}$ denote a candidate estimator constructed using the training sample, not necessarily the ERM, associated with a model $\k\in[M_n]$. The sample-splitting version of ESA (SS-ESA) evaluates these candidate estimators using the aggregation sample, which thus plays the role of a validation set. Specifically, SS-ESA monitors the validation loss $\ell_n(\ttheta_{n,\k})$ computed on the aggregation sample and aggregates candidate estimators up to a data-driven stopping index
$\tsm$ defined as
    \begin{align}
        \tsm=\inf\cbr{\k\in[M_n]\setminus\{1\}:
        \ell_n(\ttheta_{n,\k-1})<  \ell_n(\ttheta_{n,\k})}\wedge M_n.
    \end{align}
In other words, we stop the procedure until the validation loss increases. Then, building upon the model selection result with early stopping, we define the SS-ESA estimator as
    \begin{align}
        \esass:=\sum_{\k=1}^{\tsm}\omegass\ttheta_{n,\k}
        \text{ with }
         \omegass
        :=\frac{\exp(-\alpha\ell_n(\ttheta_{n,\k}))}{\sum_{\k'=1}^{\usm}\exp(-\alpha\ell_n(\ttheta_{n,\k'}) )}.
    \end{align}
Here, the parameter $\alpha>0$ controls the smoothness of the exponential weights, with larger values placing more mass on candidate estimators with smaller validation loss.

In the next two theorems, we first establish that, with high probability, the SS-ESA procedure does not terminate before reaching the near-optimal model index given by
\begin{align}
    \tkopt(\tau)
    :=\inf\cbr{\k\in[M_n]:\cL_{n}(\ttheta_{n,\k})\le (1+\tau)\cL_{n}(\ttheta_{n,\k+1})}\wedge M_n,
\end{align}
provided that $\tau>0$ is sufficiently large, and then we derive an associated oracle inequality.

\begin{theorem}[Not-too-early-stopping of SS-ESA]
\label{thm:selection_ss}
Assume that the training and aggregation samples are independent. Then conditional on the training sample,  under \cref{assume:loss}, there exist absolute constants $\ttau_*>0$ and $\tC_8>0$ such that
        \begin{align}
       \Pstar\del[1]{\tsm\ge \tkopt(\ttau_*)+1 }\ge 1-\teta_n
    \end{align}
for any sufficiently large $n\in \bN$, where we let $\teta_n:=\tkopt(\ttau_*)\e^{-\tC_8 \cL_{n}(\ttheta_{n,\tkopt(\ttau_*)})}.$
\end{theorem}

\begin{theorem}[Oracle inequality for SS-ESA estimator]
\label{thm:oracle_ss}
Assume that the function $\theta\mapsto \cL_n(\theta)$ is convex. 
Under the same assumption of \cref{thm:selection_erm}, conditional on the training sample, there exists absolute constants $\tC_9>0$ such that
    \begin{align}
        \Pstar\del[3]{ \cL_n(\esass)\le \tC_9\cbr[2]{\min_{\k\in[\ukopt(\utau_*)+1]}\cL_n(\ttheta_{n,\k})+\log(M_nn)}}
        \ge 1-\teta_n- n^{-1}
    \end{align}
for any sufficiently large $n\in \bN$.
\end{theorem}

\subsection{Early-stopped aggregation of posterior means}

Another way to avoid explicitly calibrating the penalty term is to exploit a Bayesian principle. We consider a ``Bayesian'' aggregate of point estimators defined as
     \begin{align}
       \esavbmean:=\sum_{\k=1}^{\hsm}\omegaesa\batheta_{n,\k},
       \text{ with }\batheta_{n,\k}:=\int \theta \d\hQ_{n,\k},
    \end{align}
which is equal to the posterior mean of the ESA variational posterior $\esavp$.

\begin{theorem}[Oracle inequality for ESA variational posterior mean]
\label{thm:oracle_vbmean}
    Assume that the function $\theta\mapsto \cL_n(\theta)$ is convex. 
    Under the same assumption of \cref{thm:selection_vb}, there exists an absolute constant $\tC_{10}>0$ such that
    \begin{align}
        \Pstar\del{ \cL_n(\esavbmean)\le \tC_{10}\cbr[1]{\cE_n^*+\log(M_nn)}}
        \ge 1-\eta_n-n^{-1}
    \end{align}
for any sufficiently large $n\in \bN$.
\end{theorem}

\section{Proofs for \cref{sec:theory}}

Before giving the proofs, we give several lemmas used in the proofs.

\subsection{Lemmas}

\begin{lemma}[Donsker and Varadhan’s variational formula, e.g., Lemma 2.2 of \citetS{alquier2020concentration}]
\label{lemma:variational_formula}
Let $\Theta$ be a measurable space and let $\Pi\in\cP(\Theta)$. Then for any function $g$  on $\Theta$ with $\int \exp(g(\theta))\d\Pi(\theta)<\infty$, we have
    \begin{equation}
       \sup_{Q\in\cP(\Theta)} \cbr{\int g(\theta)\d Q(\theta)-\kl(Q,\Pi)}=\log\del{\int \exp(g(\theta))\d \Pi(\theta)}.
    \end{equation}
with the convention $\infty-\infty=-\infty$. Moreover, the supremum with respect to $Q$ is attained when $Q=\Pi^{g}$ with $\Pi^{g}(\d\theta)\propto \exp(g(\theta))\Pi(\d\theta)$.
In particular,  for any $Q\in\cP(\Theta)$, any measurable subset $B\subset\Theta$ and any positive constant $a>0$,
    \begin{equation}
    \label{eq:kl_ineq_prob}
        Q(B)\le \frac{1}{a}\cbr{\kl(Q, \Pi)+ \e^a\Pi(B)}.
    \end{equation}
\end{lemma}

\begin{lemma}
\label{lemma:exp_moment_bound}
Under \cref{assume:loss,assume:variational_family}, 
    \begin{align*}
         \Pstar\sbr{\exp\del[2]{\xi_2\lambda\int \lossdiff(\theta)\d Q^*_{n,\k}(\theta)}}
          \le \exp(\xi_4\cE_n(\k))
    \end{align*}
for any $\k\in[M_n]$ and any sufficiently large $n\in \bN$.
\end{lemma}

\begin{proof}
The desired result follows as
    \begin{align*}
       \Pstar\sbr{\exp\del[3]{\xi_2\lambda\int \lossdiff(\theta)\d Q^*_{n,\k}(\theta)}}
        &\le\Pstar\sbr{\int\exp\del[1]{\xi_2\lambda \lossdiff(\theta)}\d Q^*_{n,\k}(\theta)}\\
        &= \int\sP_{\star}^{(n)}\sbr{\exp\del[1]{\xi_2\lambda\lossdiff(\theta)}}\d Q^*_{n,\k}(\theta)\\
        &\le  \int\exp(\xi_3 \cL_n(\theta))\d Q^*_{n,\k}(\theta)\\
         &\le \exp(\xi_4\cE_n(\k)),
    \end{align*}    
where we use Jensen's inequality for the first line, Fubini's theorem for the second, \cref{assume:loss} for the third, and \cref{assume:variational_family} for the fourth.
\end{proof}

\begin{lemma}
\label{lemma:composite_elbo_max}
For any $\sm\in[M_n]$, define $ \Pi_{n,\le\sm}:=\sm^{-1}\sum_{\k=1}^{\sm}\Pi_{n,\k}$ and
    \begin{align*}
        \hQ_{n, \le \sm}:=\sum_{\k=1}^{\sm}\homega_{n,\k}^{\le \sm} \hQ_{n,\k}
    \end{align*}
with $\homega_{n,\k}^{\le \sm}:=\exp(-\mvfe(\k))/\sum_{\k'=1}^{\sm}\exp(-\mvfe(\k'))$. Then we have
    \begin{align*}
      \vfe(  \hQ_{n, \le \sm}; \Pi_{n,\le\sm})
      &\le \log(\sm) + \inf_{(\tomega_1,\dots, \tomega_{\sm})\in\Delta_m}\sum_{\k=1}^{\sm}\tomega_{\k}\log\del{\frac{\tomega_{\k}}{\exp(-\mvfe(\k))  }}
      \\
      &\le \log(\sm)+    \min_{\k\in[\sm]} \mvfe(\k).
      \end{align*}
\end{lemma}

\begin{proof}
By    Lemma 6.1 of \citetS{cherief2018consistency}, for any $\tomega=(\tomega_1,\dots, \tomega_{\sm})\in\Delta_{\sm}$, we have
    \begin{align*}
        \kl\del[3]{\sum_{\k=1}^{\sm}\tomega_{\k} \hQ_{n,\k}, \frac{1}{\sm}\sum_{\k=1}^{\sm}\Pi_{n,\k}}
        &\le \sum_{\k=1}^{\sm}\tomega_{\k}\log(\sm\tomega_{\k}) + \sum_{\k=1}^{\sm}\tomega_{\k}\kl\del[1]{\hQ_{n,\k}, \Pi_{n,\k}}\\
        &= \log(\sm)+ \sum_{\k=1}^{\sm}\tomega_{\k}\cbr{\log(\tomega_{\k}) + \kl\del[1]{\hQ_{n,\k}, \Pi_{n,\k}}}.
    \end{align*}
This implies that
    \begin{align*}
        \vfe\del[3]{\sum_{\k=1}^{\sm}\tomega_{\k} \hQ_{n,\k}; \Pi_{n,\le\sm}}
        &=\kl\del[3]{\sum_{\k=1}^{\sm}\tomega_{\k} \hQ_{n,\k}, \frac{1}{\sm}\sum_{\k=1}^{\sm}\Pi_{n,\k}}
        +\sum_{\k=1}^{\sm}\tomega_{\k}\int\{\lambda\ell_n(\theta)\}\d\hQ_{n,\k}(\theta)\\
        &   \le \log(\sm) +\sum_{\k=1}^{\sm}\tomega_{\k}\sbr{\log(\tomega_{\k}) +\kl\del[1]{\hQ_{n,\k}, \Pi_{n,\k}}+\int \{\lambda\ell_n(\theta)\}\d\hQ_{n,\k}(\theta)}\\
         &  =  \log(\sm) + \sum_{\k=1}^{\sm}\tomega_{\k} \{\log(\tomega_{\k}) +\mvfe(\k)\}\\
         &=   \log(\sm) +\sum_{\k=1}^{\sm}\tomega_{\k}\log\del{\frac{\tomega_{\k}}{\exp(-\mvfe(\k))  }}
        \end{align*}
is minimized when $\tomega_{\k}=\homega_{n,\k}^{\le \sm}\propto\exp(-\mvfe(\k)) $. This proves the first desired inequality. The second ineqaulity is attained when we set $\tomega_{\k}=1$ and $\tomega_{\k'}=0$ for any $\k'\neq \k$.
\end{proof}

\subsection{Proof of \cref{prop:risk_bound_vb}}
\begin{proof}
Let $g(\theta):=\xi_1\cL_n(\theta)-\lambda\lossdiff(\theta)$. Then we have
    \begin{align*}
        \Pstar[\exp(g(\theta))]=\exp(\xi_1\cL_n(\theta))\Pstar[\exp(-\lambda\lossdiff(\theta))]\le 1
    \end{align*}
for any $\theta$. Then by Donsker and Varadhan’s variational formula (\cref{lemma:variational_formula}),
    \begin{align*}
        \int \xi_1\cL_n(\theta)\d Q(\theta)
        \le \int \lambda\lossdiff(\theta)\d Q(\theta)+ \kl(Q, \Pi_{n,\k})
    \end{align*}
everywhere for any distribution $Q\in \cP(\Theta_{n,\k})$. Therefore, the above inequality holds also for the variational posterior $\hQ_{n,\k}$. Then, by the optimization optimality of $\hQ_{n,\k}$, we have
    \begin{align*}
       \Pstar\sbr{ \int \xi_1\cL_n(\theta)\d \hQ_{n,\k}(\theta)}
       & \le \Pstar\sbr{  \int \lambda\lossdiff(\theta)\d \hQ_{n,\k}(\theta)
        +\kl(\hQ_{n,\k}(\theta), \Pi_{n,\k})}\\
        &\le \Pstar\sbr{ \int \lambda\lossdiff(\theta)\d Q_{n,\k}^*(\theta)
        +\kl(Q_{n,\k}^*, \Pi_{n,\k})}\\
        &=\int \lambda \cL_n(\theta)\d Q_{n,\k}^*(\theta)
        +\kl(Q_{n,\k}^*, \Pi_{n,\k}).
    \end{align*}
By  \cref{assume:variational_family}, we have $\kl(Q_{n,\k}^*, \Pi_{n,\k})\le \xi_4\cE_n(\k)$ and
    \begin{align}
    \label{eq:approx_bound}
        \int \lambda\cL_n(\theta)\d Q_{n,\k}^*(\theta)
        \le \lambda\xi_3^{-1}\log\del{\int \exp(\xi_3\cL_n(\theta))\d Q_{n,\k}^*(\theta)}\le\lambda\xi_3^{-1}\xi_4\cE_n(\k),
    \end{align}
where the first inequality holds due to
Jensen's inequality. This complete the proof.
\end{proof}

\subsection{Proof of \cref{thm:selection_vb}}

\begin{proof}
Let $\kopt:=\kopt(\tau_*)$ for notational simplicity. For each $\k$, we define
    \begin{align}
    \label{eq:gamma_def}
        \Gamma_n(\k)
        := \mvfe(\k)-\lambda\ell_n^\star
        =\lambda\int \lossdiff(\theta)\d\hQ_{n,\k}(\theta) + \kl(\hQ_{n,\k},\Pi_{n,\k}) 
    \end{align}    
so that our estimator is equivalently written as
    \begin{align*}
        \hsm=\inf\cbr{\k\in[M_n]\setminus\{1\}:
        \Gamma_n(\k-1)< \Gamma_n(\k)}\wedge M_n.
    \end{align*}
We proceed as
    \begin{align*}
       \Pstar\del[1]{\hsm< \kopt+1}
       \le \sum_{\k=1}^{\kopt}\Pstar\del{\hsm= \k}
       \le \sum_{\k=1}^{\kopt}\Pstar\del{ \Gamma_n(\k-1)<\Gamma_n(\k)}.
    \end{align*}
For each summand, we decompose it as
    \begin{align}
   &\Pstar\del{ \Gamma_n(\k-1)<\Gamma_n(\k)}\nonumber\\
       &
       \le  \Pstar\del{ \Gamma_n(\k)\ge \tC'\cE_{n}(\k)}
        + \Pstar\del{ \Gamma_n(\k-1)\le \tC'\cE_{n}(\k)}, \label{eq:elbo_ineq_decomp}
    \end{align}   
where the absolute constant $\tC'>0$ will be chosen later. 
For the second term of the right-hand side of \labelcref{eq:elbo_ineq_decomp}, we note that, by Donsker and Varadhan’s variational formula (\cref{lemma:variational_formula}),
    \begin{align*}
        -\Gamma_n(\k-1)
        &=-\lambda\int \lossdiff(\theta)\d\hQ_{n,\k-1}(\theta) - \kl(\hQ_{n,\k-1},\Pi_{n,\k-1}) \\
        &\le \log\del{\int  \exp\del[1]{-\lambda\lossdiff(\theta)}\d\Pi_{n,\k-1}(\theta)}.
    \end{align*}
Therefore, by Markov's inequality and Fubini's theorem, we have
    \begin{align*}
         \Pstar\del{ \Gamma_n(\k-1)\le \tC'\cE_{n}(\k)}
          &\le \Pstar\del{\log\del{\int  \exp\del[1]{-\lambda\lossdiff(\theta)}\d\Pi_{n,\k-1}(\theta)}\ge -\tC'\cE_{n}(\k)}\\
          &\le \exp(\tC'\cE_{n}(\k))\sP_{\star}^{(n)}\sbr{\int  \exp\del[1]{-\lambda\lossdiff(\theta)}\d\Pi_{n,\k-1}(\theta)}\\
          &= \exp(\tC'\cE_{n}(\k))\int\Pstar \sbr{\exp\del[1]{-\lambda\lossdiff(\theta)}}\d\Pi_{n,\k-1}(\theta)
    \end{align*}
By \cref{assume:loss}, the expectation term of the last display is bounded as
    \begin{align*}
        \int\Pstar \sbr{\exp\del[1]{-\lambda\lossdiff(\theta)}}\d\Pi_{n,\k-1}(\theta)
        &\le  \int \exp(- \xi_1 \cL_n(\theta))\d\Pi_{n,\k-1}(\theta)
        =\exp(-\cE_n(\k-1))
    \end{align*}
Therefore, by using the fact that $\cE_n(\k-1)> (1+\tau_*)\cE_{n}(\k)$ for $\k\in[\kopt]$, we have
    \begin{align}   
        \Pstar\del{ \Gamma_n(\k-1)\le \tC'\cE_{n}(\k)}          
        &\le  \exp\del{\tC'\cE_{n}(\k)- (1+\tau_*)\cE_{n}(\k)}.
        \label{eq:term1}
    \end{align}
Next, we focus on the first term of the right-hand side of \labelcref{eq:elbo_ineq_decomp}. By the optimization optimality of the variational posterior, we have
    \begin{align*}
        \Gamma_n(\k)
        &\le \lambda\int \lossdiff(\theta)\d Q^*_{n,\k}(\theta) + \kl(Q^*_{n,\k},\Pi_{n,\k}) \\
        &\le \lambda\int \lossdiff(\theta)\d Q^*_{n,\k}(\theta) +\xi_4\cE_{n}(\k),
    \end{align*}
where the second inequality follow from \cref{assume:variational_family}. Therefore, by applying Markov's inequality and \cref{lemma:exp_moment_bound} in order, we have
    \begin{align}
         \Pstar&\del{ \Gamma_n(\k)\ge \tC'\cE_{n}(\k)}\nonumber\\
          &\le  \Pstar\del{\xi_2\lambda\int \lossdiff(\theta)\d Q^*_{n,\k}(\theta)  \ge \xi_2(\tC'-\xi_4)\cE_{n}(\k)}\nonumber\\
         &\le   \exp\del[2]{-\xi_2(\tC'-\xi_4)\cE_{n}(\k)}\Pstar\sbr{\exp\del{\xi_2\lambda\int \lossdiff(\theta)\d Q^*_{n,\k}(\theta)}}\nonumber\\
         &\le  \exp\del[2]{-\xi_2(\tC'-\xi_4(1+\xi_2^{-1}))\cE_n(\k)}.    \label{eq:term2}
    \end{align}
Hence, when we take $\tC'$ larger than $\xi_4(1+\xi_2^{-1})$, and take $\tau_*$ larger than $\tC'-1$, we get the following bound by combining \labelcref{eq:term1} and \labelcref{eq:term2}
 \begin{align*}
       \Pstar\del[1]{\hsm< \kopt+1}
       \le\sum_{\k=1}^{\kopt} \exp\del{-\tC''\cE_n(\k)}
       \le \kopt\exp\del{-\tC''\cE_n(\kopt)},
    \end{align*}
which completes the proof.
\end{proof}

\subsection{Proof of \cref{thm:contract_posterior}} 

\begin{proof}
Let $\koracle=\argmin_{\k\in[\kopt(\tau_*)+1]}\cE_n(\k)$, $\fM_n:=\{ X^{(n)}\in\cX_n:\hsm\ge \kopt(\tau_*) +1\}$ and 
    \begin{align*}
        \Xi_n:=\cbr{\theta\in\Theta_n:\cL_n(\theta)\ge A_n(\cE_n^*+\log M_n)}.
    \end{align*}
Then, in view of \cref{thm:selection_vb}, it suffices to bound the quantity $\Pstar\sbr{\esapst\del{\Xi_n}\ind(\fM_n)}$. Define the event
    \begin{align*}
       \text{ $\fA_n$}
       :=\cbr{X^{(n)}\in\cX_n:\int \exp\del[1]{-\lambda\lossdiff(\theta)}\d\Pi_{n,\koracle}(\theta)\ge \exp(-\sqrt{A_n} \cE_n^*)}.
    \end{align*}
Let $Q^*:=Q_{n,\koracle}^*$ which satisfies the condition in \cref{assume:variational_family} for the model $\koracle$. By \cref{lemma:variational_formula}, we have
    \begin{align*}
        \log\del{\int \exp\del[1]{-\lambda\lossdiff(\theta)}\d\Pi_{n,\koracle}(\theta)}
        &\ge \int-\lambda\lossdiff(\theta)\d Q^*(\theta)-\kl(Q^*,\Pi_{n,\koracle})\\
        &\ge\int-\lambda\lossdiff(\theta)\d Q^*(\theta)-\xi_4 \cE_n(\koracle)
    \end{align*}
everywhere, where the second inequality follows from \cref{assume:variational_family}. Hence, by applying Markov's inequality and \cref{lemma:exp_moment_bound} in order, we obtain
    \begin{align*}
       \Pstar\del{\fA_n^\complement}
        &=\Pstar\del{\int \exp\del[1]{-\lambda\lossdiff(\theta)}\d\Pi_{n,\koracle}(\theta)< \exp(-\sqrt{A_n} \cE_n^*)}\\
         &\le  \Pstar\del{\lambda\int \lossdiff(\theta)\d Q^*(\theta)  >(\sqrt{A_n}-\xi_4)\cE_{n}^*}\\
          &\le   \exp\del[2]{-\xi_2(\sqrt{A_n}-\xi_4)\cE_n^*}\Pstar\sbr{\exp\del{\xi_2\lambda\int \lossdiff(\theta)\d Q^*(\theta)}}\\
          &\le  \exp\del[2]{-\xi_2(\sqrt{A_n}-\xi_4(1+\xi_2^{-1}))\cE_n^*}, 
    \end{align*}
which goes to zero as $A_n\to\infty$. Therefore, it suffices to bound the term $\sP_{\star}^{(n)}\sbr{\esapst\del{\Xi_n}\ind(\fM_n\cap \fA_n)}$. We first note that on the event $\fM_n\cap \fA_n$,
    \begin{align*}
        \int \exp\del[1]{-\lambda\lossdiff(\theta)}\d\Pi_{n,\le\hsm}(\theta)
        &= (\hsm)^{-1}\sum_{\k\in[\hsm]}\int \exp\del[1]{-\lambda\lossdiff(\theta)}\d\Pi_{n,\k}(\theta)\\
        &\ge(\hsm)^{-1}\int \exp\del[1]{-\lambda\lossdiff(\theta)}\d\Pi_{n,\koracle}(\theta)\\
&\ge \frac{1}{M_n}\exp(-\sqrt{A_n} \cE_n^*)
    \end{align*}
    This leads to
    \begin{align*}
       \Pstar&\sbr{\esapst\del{\Xi_n}\ind(\fM_n\cap \fA_n)}\\
        &\le M_n\exp(\sqrt{A_n}  \cE_n^*)
       \Pstar\sbr[4]{\sum_{\k\in[\hsm]}\frac{1}{\hsm}\int_{\Xi_n}\exp\del[1]{-\lambda\lossdiff(\theta)}\d\Pi_{n,\k}(\theta)\ind(\fM_n)},
    \end{align*}
where the expectation term is further bounded as
    \begin{align*}
       \Pstar&\sbr[4]{\sum_{\k\in[\hsm]}\frac{1}{\hsm}\int_{\Xi_n}\exp\del[1]{-\lambda\lossdiff(\theta)}\d\Pi_{n,\k}(\theta)\ind(\fM_n)}\\
        &=\sum_{\sm=\kopt+1}^{M_n}\sP_{\star}^{(n)}\sbr[4]{\sum_{\k\in[\hsm]}\frac{1}{\hsm}\int_{\Xi_n}\exp\del[1]{-\lambda\lossdiff(\theta)}\d\Pi_{n,\k}(\theta)\ind(\hsm=\sm)}\\
    &=\sum_{\sm=\kopt+1}^{M_n}\sP_{\star}^{(n)}\sbr[4]{\sum_{\k\in[\sm]}\frac{1}{\sm}\int_{\Xi_n}\exp\del[1]{-\lambda\lossdiff(\theta)}\d\Pi_{n,\k}(\theta)\ind(\hsm=\sm)}\\
    &\le \sum_{\sm=\kopt+1}^{M_n}\sP_{\star}^{(n)}\sbr[4]{\sum_{\k\in[\sm]}\frac{1}{\sm}\int_{\Xi_n}\exp\del[1]{-\lambda\lossdiff(\theta)}\d\Pi_{n,\k}(\theta)}\\
    &\le \sum_{\sm=\kopt+1}^{M_n}\sum_{\k\in[\sm]}\frac{1}{\sm}\int_{\Xi_n}\sP_{\star}^{(n)}\sbr{\exp\del[1]{-\lambda\lossdiff(\theta)}}\d\Pi_{n,\k}(\theta)\\
    &= \sum_{\sm=\kopt+1}^{M_n}\sum_{\k\in[\sm]}\frac{1}{\sm}\int_{\Xi_n}\exp(-\xi_1\cL_n(\theta))\d\Pi_{n,\k}(\theta) \\
    &\le M_n \exp(-\xi_1 A_n(\cE_n^*+\log M_n)).
    \end{align*}
The proof is done since $A_n$ is eventually larger than $\sqrt{A_n}$.
\end{proof}

\subsection{Proof of \cref{thm:contract_esva}}

\begin{proof}
We define the events $\fM_n$ and $\fA_n$, the set $\Xi_n$ and the distribution $Q^*$ as we did in the proof of \cref{thm:contract_posterior}. By the inequality \eqref{eq:kl_ineq_prob} given in \cref{lemma:variational_formula},
    \begin{align*}
       \Pstar&\sbr{ \esavp\del{\Xi_n}\ind(\fM_n)}
        \le \Pstar(\fA_n^\complement)+\sP_{\star}^{(n)}\sbr{ \esavp\del{\Xi_n}\ind(\fM_n\cap \fA_n)}\\
        &\le  \Pstar(\fA_n^\complement)
        +\frac{1}{T_n}\sP_{\star}^{(n)}\sbr{ \kl(\esavp,\esapst)\ind(\fM_n)}
        +\frac{\e^{T_n}}{T_n}\sP_{\star}^{(n)}\sbr{\esapst\del{\Xi_n}\ind(\fM_n\cap \fA_n)},
    \end{align*}
where we set $T_n:=\sqrt{A_n}(\cE_n^*+\log M_n)$. In the proof of \cref{thm:contract_posterior}, we have shown that the first and third terms of the above display tend to zero. Thus, the proof is done if we show that the second term converges to zero. For each $\sm\in[M_n]$,  we define  $\Pi_{n,\le\sm}:=\frac{1}{\sm}\sum_{\k=1}^{\sm}\Pi_{n,\k}$ and
    \begin{align*}
       \bar{\ell}_{n}(\sm):=-\frac{1}{\lambda}\log\del{ \int \exp(-\lambda\ell_n(\theta))\d\Pi_{n,\le\sm}(\theta)}.
    \end{align*}
We start with the identity
    \begin{align}
         &\sP_{\star}^{(n)}\sbr{ \kl(\esavp,\esapst)\ind(\fM_n)}
         =        \Pstar\sbr{\int\log\del{\frac{ \exp(-\lambda\bar{\ell}_{n}(\hsm))\d \esavp(\theta)}{\exp(-\lambda\ell_n(\theta))\d\Pi_{n,\le\hsm}(\theta)}}\d \esavp(\theta)\ind(\fM_n)} \nonumber \\
            &=\Pstar\sbr{\cbr{\kl(\esavp,\Pi_{n,\le\hsm})
         +\lambda\int \ell_n(\theta)\d \esavp(\theta)}\ind(\fM_n)
         -\lambda\bar{\ell}_{n}(\hsm)\ind(\fM_n)} \nonumber\\
          &=\Pstar\sbr{\cbr{\kl(\esavp,\Pi_{n,\le\hsm})
         +\lambda\int \lossdiff(\theta)\d \esavp(\theta)}\ind(\fM_n)}
         +\Pstar\sbr{-\lambda(\bar{\ell}_{n}(\hsm)-\ell_n^\star)\ind(\fM_n)} \label{eq:variational_error},
    \end{align}
We will bound two terms in \labelcref{eq:variational_error} separately. For the first term, we note that, on the event $\fM_n$, \cref{lemma:composite_elbo_max} implies that
    \begin{align*}
     \kl(\esavp,\Pi_{n,\le\hsm})
         +\lambda\int \lossdiff(\theta)\d \esavp(\theta)
       & \le \min_{\k\in[\hsm]}\{\mvfe(\k)\}-\lambda\ell_n(\theta) + \log(\hsm)\\
       &\le\mvfe(\koracle)-\lambda\ell_n(\theta) +\log (M_n)\\
       &\le   \kl(Q^*,\Pi_{n,\koracle})
         +\lambda\int \lossdiff(\theta)\d Q^*(\theta).
    \end{align*}
Therefore, we have
    \begin{align*}
      \Pstar&\sbr{\cbr{\kl(\esavp,\Pi_{n,\le\hsm})
         +\lambda\int \lossdiff(\theta)\d \esavp(\theta)}\ind(\fM_n)}\\
      &\le\Pstar\sbr{ \kl(Q^*,\Pi_{n,\koracle})
         +\lambda\int \lossdiff(\theta)\d Q^*(\theta)}+\log (M_n)\\
         &= \int\sP_{\star}^{(n)}\sbr{\lambda\lossdiff(\theta)}\d Q^*(\theta)+\kl(Q^*,\Pi_{n,\koracle})+\log (M_n)\\
          &\le  \int \lambda\cL_n(\theta)\d Q^*(\theta)+\xi_4\cE_n(\koracle)+\log (M_n)\\
          &\le  (\lambda\xi_3^{-1}+1)\xi_4\cE_n(\koracle)+\log (M_n)
    \end{align*}
where we use the inequality derived in \eqref{eq:approx_bound} in the proof of \cref{prop:risk_bound_vb} for the second inequality. 
Next, for the second term in \labelcref{eq:variational_error}, we have
  \begin{align*}
  \Pstar\sbr{-\lambda(\bar{\ell}_{n}(\hsm)-\ell_n^\star)\ind(\fM_n)}
   &\le   \Pstar\sbr{\log(1+\exp(-\lambda(\bar{\ell}_{n}(\hsm)-\ell_n^\star)))\ind(\fM_n)}\\
      &\le   \Pstar\sbr{\log(1+\exp(-\lambda(\bar{\ell}_{n}(\hsm)-\ell_n^\star)))}\\
   &\le   \log\del{1+\Pstar\sbr{\exp(-\lambda(\bar{\ell}_{n}(\hsm)-\ell_n^\star)))}}\\
     &\le   \log\del{ 1+\sum_{\sm=1}^{M_n}\sP_{\star}^{(n)}\sbr{\exp(-\lambda(\bar{\ell}_{n}(\sm)-\ell_n^\star)))}}\\
   &=   \log\del{ 1+\sum_{\sm=1}^{M_n}\sP_{\star}^{(n)}\sbr{\int \exp(-\lambda(\ell_n(\theta)-\ell_n^\star))\d\Pi_{n,\le\sm}(\theta)}}\\
     &\le\log\del{ 1+\sum_{\sm=1}^{M_n}\int \exp(-\xi_1 \cL_n(\theta))\d\Pi_{n,\le\sm}(\theta)}\\
     &\le \log (1+ M_n),
    \end{align*}
where the first inequality follows from the inequality $z\le \log(1+\exp(z))$ for $z\in \R$, the second one does from the nonnegativity of $\log(1+z)$ for $z>0$ and the third one does from Jensen's inequality. This completes the proof.
\end{proof}

\subsection{Proof of \cref{thm:overestimation}}
\begin{proof}
Let $\Gamma_{n}$ be the quantity defined in \eqref{eq:gamma_def}. Let $\kover:=\kover(\nu^*)$. Then the event $\{\hsm >\kover+1\}$ implies that  $\Gamma_{n}(\k)$ is monotonically decreasing until $\kover+1$. Therefore, we have
    \begin{align*}
         \Pstar\del{\hsm >\kover+1}
          &\le \Pstar\del{\Gamma_n(\koracle)>\Gamma_n(\kover+1)}, 
    \end{align*}
where $\koracle=\argmin_{\k\in[\kopt(\tau_*)+1]}\cE_n(\k)$.
Due to \cref{thm:selection_vb}, it suffices to bound the first term of the last display. By employing  a similar argument used in the proof of \cref{thm:selection_vb}, we have
    \begin{align*}
        \Pstar\del{\Gamma_n(\koracle)>\Gamma_n(\kover+1)}
          &\le    \Pstar\del{\Gamma_n(\koracle)\ge \tC'_1\cE_n^*}
          + \Pstar\del{\Gamma_n(\kover+1)\le \tC'_1\cE_n^*}\\
          &\le    \exp(-\tC'_2\cE_n^*)  + \Pstar\del{\Gamma_n(\kover+1)\le \tC'_1\cE_n^*}
    \end{align*}    
for some constants $\tC'_1,\tC'_2>0$. To bound the second term of the last display, we use Donsker and Varadhan’s variational formula (\cref{lemma:variational_formula}) and Markov inequality to have
    \begin{align*}
        \Pstar\del{\Gamma(\kover+1)\le \tC'_1\cE_n^*}
          &\le \exp(\tC'_1\cE_n^*)\int\Pstar \sbr{\exp\del[1]{-\lambda\lossdiff(\theta)}}\d\Pi_{n,\kover+1}(\theta)\\
          &\le    \exp(\tC'_1\cE_n^*)\int \exp(-\xi_1\cL_n(\theta))\d\Pi_{n,\kover+1}(\theta)\\
          &=\exp(\tC'_1\cE_n^*-\cE_n(\kover+1)),
    \end{align*}
where the second line follows by \cref{assume:loss}.Thus, by the definition of $\kover$, we have the desired result for sufficiently large $\nu^*$.
\end{proof}

\section{Proofs for \cref{sec:eb}}

Before giving the proofs, we state one lemma which is repeated in the proofs.

\begin{lemma}
\label{lemma:eb_lossdiff_bound}
Under \cref{assume:loss,assume:hyperpar}, we have
    \begin{align*}
         \Pstar\sbr{\int  \exp\del[1]{-\lambda\barho^{-1}\lossdiff(\theta)}\d\Pi_{n,\k}(\theta|\cpsi_{n,\k})}\le  \exp(-(2\barho)^{-1}\ccE_n(\k))
    \end{align*}
for any $\k\in[M_n]$ and  any sufficiently large $n\in \bN$.
\end{lemma}

\begin{proof}
Let $\Psib$ be a minimal $(1+\rho)$-R\'enyi-type hyperparameter cover of $\Pi[\Psi]_{n,\k}$ with precision $R_{n,\k}$ such that $\log(|\Psib|)\le U_{n,\k}$, which exists by assumption. For each $\psib\in\Psib$, define the ball as
    \begin{align*}
        \cB_{n,\k}(\psib):=\cbr{\psi\in\Psi_{n,\k}:\renyihp(\psi,\psi^\bullet)\le R_{n,\k}}.
    \end{align*}
By H\"older's inequality, we have
    \begin{align*}
        &\int  \exp\del[1]{-(\lambda/\barho)\lossdiff(\theta)}\d\Pi_{n,\k}(\theta|\cpsi_{n,\k})\\
        &= \int  \exp\del[1]{-(\lambda/\barho)\lossdiff(\theta)}\frac{\d\Pi_{n,\k}(\theta|\cpsi_{n,\k})}{\d\Pi_{n,\k}(\theta|\psib)}\d\Pi_{n,\k}(\theta|\psib)\\
        &\le  \cbr{\int  \exp\del[1]{-\lambda\lossdiff(\theta)}\d\Pi_{n,\k}(\theta|\psi^\bullet)}^{1/\barho}
       \cbr{ \int\del{ \frac{\d\Pi_{n,\k}(\theta|\cpsi_{n,\k})}{\d\Pi_{n,\k}(\theta|\psib)}}^{1+\rho}\d\Pi_{n,\k}(\theta|\psib)}^{1/(1+\rho)},
    \end{align*}
where the term in the second bracket is bounded as, on the event $\cpsi_{n,\k}\in\cB_{n,\k}(\psib)$, 
    \begin{align*}
         \int\del{ \frac{\d\Pi_{n,\k}(\theta|\cpsi_{n,\k})}{\d\Pi_{n,\k}(\theta|\psib)}}^{1+\rho}\d\Pi_{n,\k}(\theta|\psib)
         =\exp\del[1]{\rho\renyihp(\cpsi_{n,\k},\psib)}
         &\le \exp(\rho R_{n,\k}).
    \end{align*}
Therefore, we have
    \begin{align*}
         &\Pstar\sbr{\ind(\cpsi_{n,\k}\in\cB_{n,\k}(\psib))\int  \exp\del[1]{-(\lambda/\barho)\lossdiff(\theta)}\d\Pi_{n,\k}(\theta|\cpsi_{n,\k})}\\
         &\le \exp(\rho R_{n,\k})\Pstar\sbr{  \cbr{\int  \exp\del[1]{-\lambda\lossdiff(\theta)}\d\Pi_{n,\k}(\theta|\psib)}^{1/\barho}}.
    \end{align*}
The expectation term is bounded as
        \begin{align*}
       \Pstar\sbr{  \cbr{\int  \exp\del[1]{-\lambda\lossdiff(\theta)}\d\Pi_{n,\k}(\theta|\psi^\bullet)}^{1/\barho}}
       &\le  \cbr{\Pstar\sbr{\int \exp\del[1]{-\lambda\lossdiff(\theta)}\d\Pi_{n,\k}(\theta|\psib)}}^{1/\barho}\\
        &\le  \cbr{\int \exp\del[1]{-\xi_1\cL_n(\theta)}\d\Pi_{n,\k}(\theta|\psib)}^{1/\barho}\\
        &\le  \cbr{\sup_{\psi\in\Psi_{n,\k}}\int \exp\del{-\xi_1\cL_n(\theta)}\d\Pi_{n,\k}(\theta|\psi)}^{1/\barho}\\
       &=\exp(-\ccE_n(\k)/\barho),
    \end{align*}
where the first inequality follows from Jensen's inequality, and the second one from \cref{assume:loss}. Hence, since $\log(|\Psib|)+\rho R_{n,\k}\le (2\barho)^{-1}\ccE_n(\k)$ by assumption, we have
    \begin{align*}
    &\Pstar\sbr{\int  \exp\del[1]{-(\lambda/\barho)\lossdiff(\theta)}\d\Pi_{n,\k}(\theta|\cpsi_{n,\k})}\\
        &\le\sum_{\psib\in\Psib}\Pstar\sbr{\ind(\cpsi_{n,\k}\in\cB_{n,\k}(\psib))\int  \exp\del[1]{-(\lambda/\barho)\lossdiff(\theta)}\d\Pi_{n,\k}(\theta|\cpsi_{n,\k})}\\
        &\le |\Psib|\exp(\rho R_{n,\k}-\ccE_n(\k)/\barho)\\
        &\le \exp(-(2\barho)^{-1}\ccE_n(\k)),
    \end{align*}
which completes the proof.
\end{proof}

\subsection{Proof of \cref{thm:selection_eb}}
\begin{proof}
 For notational convenience, we let $\ckopt:=\ckopt(\ctau_*)$. For each $\k$, we define
    \begin{align*}
        \cGamma_n(\k)
        &:= \ebmvfe(\k)-\frac{\lambda}{\barho}\ell_n^\star\\
        &=\frac{\lambda}{\barho}\int \lossdiff(\theta)\d\chQ_{n,\k}(\theta) + \kl(\chQ_{n,\k},\Pi_{n,\k}(\cdot|\cpsi_{n,\k})).
    \end{align*}    
Then our early stopping model index is equivalently written as
    \begin{align*}
        \csm=\inf\cbr{\k\in[M_n]\setminus\{1\}:
        \cGamma_n(\k-1)< \cGamma_n(\k)}\wedge M_n.
    \end{align*}
As we did in the proof of \cref{thm:selection_vb}, it suffices to show that both $I_{n1}:=\Pstar\del{\cGamma_n(\k)\ge \tC'\ccE_{n}(\k)}$ and $I_{n2}:=\sP_{\star}^{(n)}\del{ \cGamma_n(\k-1)\le \tC'\ccE_{n}(\k)}$ tend to zero for each $\k\in[\ckopt]$ for some constant $\tC'>0$. 

We first bound $I_{n2}$. 
By a similar argument based on Donsker and Varadhan's variational formula as in the proof of  \cref{thm:selection_vb}, we have
   \begin{align*}
         &\Pstar\del{ \cGamma_n(\k-1)\le \tC'\ccE_{n}(\k)}\\
         &\le \Pstar\del{-\frac{\lambda}{\barho}\int \lossdiff(\theta)\d\chQ_{n,\k-1}(\theta) - \kl(\chQ_{n,\k-1},\Pi_{n,\k-1}(\cdot|\cpsi_{n,\k-1}))\ge -\tC'\ccE_{n}(\k)}\\
            &\le \Pstar\del{\int  \exp\del[1]{-(\lambda/\barho)\lossdiff(\theta)}\d\Pi_{n,\k-1}(\theta|\cpsi_{n,\k-1})\ge \exp(-\tC'\ccE_{n}(\k))}\\
            &\le  \exp(\tC'\ccE_{n}(\k)))\Pstar\sbr{\int  \exp\del[1]{-(\lambda/\barho)\lossdiff(\theta)}\d\Pi_{n,\k-1}(\theta|\cpsi_{n,\k-1})}\\
             &\le  \exp(\tC'\ccE_{n}(\k))-(2\barho)^{-1}\ccE_n(\k-1))\\
             &\le\exp(-((2\barho)^{-1}(1+\ctau_*)-\tC')\ccE_{n}(\k)),
    \end{align*}
 where the forth inequality follows from \cref{lemma:eb_lossdiff_bound} and the last inequality  from $\ccE_n(\k-1)> (1+\tau_*)\ccE_{n}(\k)$ for $\k\in[\ckopt]$,
 
We now turn our attention to $I_{n1}$. By the optimization optimality of the pair $(\chQ_{n,\k},\cpsi_{n,\k})$ and the nonnegativity of the function $ \Upsilon_{n,\k}$, we have
    \begin{align*}
        \cGamma_n(\k)
         &\le \lambda\barho^{-1}\int \lossdiff(\theta)\d \chQ_{n,\k}(\theta) + \kl(\chQ_{n,\k},\Pi_{n,\k}(\cdot|\cpsi_{n,\k})) + \Upsilon_{n,\k}(\cpsi_{n,\k}) \\
        &\le \lambda\barho^{-1}\int \lossdiff(\theta)\d Q^*_{n,\k}(\theta) + \kl(Q^*_{n,\k},\Pi_{n,\k}(\cdot|\psi^*_{n,\k})) + \Upsilon_{n,\k}(\psi^*_{n,\k}) \\
        &\le \lambda\barho^{-1}\int \lossdiff(\theta)\d Q^*_{n,\k}(\theta) +2\xi_5\ccE_{n}(\k),
    \end{align*}
where the third inequality follows from \cref{assume:eb_family}. Therefore, by applying Markov's inequality and \cref{lemma:exp_moment_bound} in order, we have
    \begin{align*}
         \sP_{\star}^{(n)}\del{ \cGamma_n(\k)\ge \tC'\ccE_{n}(\k)}
          &\le  \Pstar\del{\xi_2\lambda\int \lossdiff(\theta)\d Q^*_{n,\k}(\theta)  \ge \barho \xi_2(\tC'-2\xi_5)\ccE_{n}(\k)}\\
         &\le   \exp\del[2]{-\xi_2\barho(\tC'-2\xi_5)\ccE_{n}(\k)}\Pstar\sbr{\exp\del{\xi_2\lambda\int \lossdiff(\theta)\d Q^*_{n,\k}(\theta)}}\nonumber\\
         &\le  \exp\del[2]{-\xi_2\barho(\tC'-\xi_5(1+(\xi_2\barho)^{-1}))\ccE_n(\k)}.  
    \end{align*}
If we take $\tC'$ and $\ctau_*$ sufficiently large, we have the desired result.
\end{proof}

\subsection{Proof of \cref{thm:contract_eb} }

\begin{proof}
Let $\koracle=\argmin_{\k\in[\ckopt(\ctau_*)+1]}\ccE_n(\k)$, $Q^*:=Q^*_{n,\koracle}$, $\psi^*:=\psi^*_{n,\koracle}$ and $\Pi^*:=\Pi_{n,\koracle}(\cdot|\psi^*)$ for notational simplicity. We define two events $\fM_n:=\{ X^{(n)}\in\cX_n:\csm\ge \ckopt(\ctau_*) +1\}$ and 
    \begin{align*}
       \text{ $\fA_n$}
       :=\cbr{X^{(n)}\in\cX_n:\int \exp\del[1]{-\lambda\barho^{-1}\lossdiff(\theta)}\d\Pi_{n,\koracle}(\theta|\cpsi_{n,\koracle})\ge \exp(-\sqrt{A_n} \ccE_n^*)}.
    \end{align*}
 We also define the subset of the parameter space   
    \begin{align*}
        \Xi_n:=\cbr{\theta\in\Theta_n:\cL_n(\theta)\ge A_n(\ccE_n^*+\log M_n)}.
    \end{align*}
For each $\sm\in[M_n]$, we define an aggregated empirical prior distribution
        \begin{align*}
        \Pi_{n,\le\sm}:=\frac{1}{\sm}\sum_{\k=1}^{\sm}\Pi_{n,\k}(\cdot|\cpsi_{n,\k})
    \end{align*}
up to a model $\sm$. We also define the ESA empirical Bayes posterior as
    \begin{align*}
        \esaebpst(\d\theta):=  
        \frac{\exp(-\lambda\barho^{-1}\ell_n(\theta))\d\Pi_{n,\le\csm}(\theta)}{\int \exp(-\lambda\barho^{-1}\ell_n(\theta'))\d\Pi_{n,\le\csm}(\theta')}.
    \end{align*}
By the inequality \eqref{eq:kl_ineq_prob} given in \cref{lemma:variational_formula},
    \begin{align}
      & \Pstar\sbr{ \esaebvp\del{\Xi_n}\ind(\fM_n)}
        \le \Pstar(\fA_n^\complement)+\sP_{\star}^{(n)}\sbr{ \esaebvp\del{\Xi_n}\ind(\fM_n\cap \fA_n)}\nonumber\\
        &\le  \Pstar(\fA_n^\complement)
        +\frac{1}{T_n}\sP_{\star}^{(n)}\sbr{ \kl(\esavp,\esaebpst)\ind(\fM_n)}
        +\frac{\e^{T_n}}{T_n}\sP_{\star}^{(n)}\sbr{\esaebpst\del{\Xi_n}\ind(\fM_n\cap \fA_n)},
  \label{eq:eb_decompose}
    \end{align}
where we set $T_n:=\sqrt{A_n}(\ccE_n^*+\log M_n)$. We bound  each of three terms in the above display separately.

We first bound the first term in \eqref{eq:eb_decompose}. Note that
    \begin{align*}
        &\log\del{\int \exp\del[1]{-\lambda\barho^{-1}\lossdiff(\theta)}\d\Pi_{n,\koracle}(\theta|\cpsi_{n,\koracle})}\\
        &\ge \int-\lambda\barho^{-1}\lossdiff(\theta)\d \chQ_{n,\koracle}(\theta)-\kl\del{ \chQ^*_{n,\koracle},\Pi_{n,\koracle}(\cdot|\cpsi_{n,\koracle})}\\
         &\ge \int-\lambda\barho^{-1}\lossdiff(\theta)\d \chQ_{n,\koracle}(\theta)
         -\kl\del{ \chQ_{n,\koracle},\Pi_{n,\koracle}(\cdot|\cpsi_{n,\koracle})}-\Upsilon_{n,\koracle}(\cpsi_{n,\koracle})\\
           &\ge \int-\lambda\barho^{-1}\lossdiff(\theta)\d Q^*(\theta)
         -\kl(Q^*,\Pi^*)-\Upsilon_{n,\koracle}(\cpsi^*)\\
        &\ge\int-\lambda\barho^{-1}\lossdiff(\theta)\d Q^*(\theta)-2\xi_5 \ccE_n^*,
    \end{align*}
where we use  \cref{lemma:variational_formula}, the nonnegativity of $\Upsilon_{n,\k}$, the optimization optimality of $(\chQ_{n,\koracle}, \cpsi_{n,\koracle})$, and \cref{assume:eb_family} for the first, third and fourth inequalities, respectively. Hence, by applying Markov's inequality and \cref{lemma:exp_moment_bound} in order, we obtain
    \begin{align*}
       \Pstar\del{\fA_n^\complement}
         &\le  \Pstar\del{\lambda\barho^{-1}\int \lossdiff(\theta)\d Q^*(\theta)  >(\sqrt{A_n}-2\xi_5)\ccE_{n}^*}\\
          &\le   \exp\del[2]{-\xi_2(\sqrt{A_n}-2\xi_5)\ccE_n^*}\Pstar\sbr{\exp\del{\xi_2\lambda\int \lossdiff(\theta)\d Q^*(\theta)}}\\
          &\le  \exp\del[2]{-\xi_2(\sqrt{A_n}-\xi_5(1+2\xi_2^{-1}))\ccE_n^*}, 
    \end{align*}
which goes to zero as $A_n\to\infty$.

For the second term in \eqref{eq:eb_decompose}, we note that, on the event $\fM_n\cap \fA_n$, we have
    \begin{align*}
        \int \exp\del[1]{-\lambda\lossdiff(\theta)}\d\Pi_{n,\le\hsm}(\theta)
        &= (\csm)^{-1}\sum_{\k\in[\hsm]}\int \exp\del[1]{-\lambda\lossdiff(\theta)}\d\Pi_{n,\k}(\theta|\cpsi_{n,\k})\\
        &\ge(\csm)^{-1}\int \exp\del[1]{-\lambda\lossdiff(\theta)}\d\Pi_{n,\koracle}(\theta|\cpsi_{n,\koracle})\\
&\ge \frac{1}{M_n}\exp(-\sqrt{A_n} \ccE_n^*).
    \end{align*}
This leads to
    \begin{align*}
    \Pstar   &\sbr{\esaebpst\del{\Xi_n}\ind(\fM_n\cap \fA_n)}\\
        &\le M_n\exp(\sqrt{A_n}  \cE_n^*)
       \Pstar\sbr[4]{\sum_{\k\in[\csm]}\frac{1}{\csm}\int_{\Xi_n}\exp\del[1]{-\lambda\lossdiff(\theta)}\d\Pi_{n,\k}(\theta|\cpsi_{n,\k})\ind(\fM_n)},
    \end{align*}
where the expectation term is further bounded as
    \begin{align*}
       \Pstar&\sbr[4]{\sum_{\k\in[\csm]}\frac{1}{\csm}\int_{\Xi_n}\exp\del[1]{-\lambda\barho^{-1}\lossdiff(\theta)}\d\Pi_{n,\k}(\theta|\cpsi_{n,\k})\ind(\fM_n)}\\
    &=\sum_{\sm=\ckopt+1}^{M_n}\sP_{\star}^{(n)}\sbr[4]{\sum_{\k\in[\sm]}\frac{1}{\sm}\int_{\Xi_n}\exp\del[1]{-\lambda\barho^{-1}\lossdiff(\theta)}\d\Pi_{n,\k}(\theta|\cpsi_{n,\k})\ind(\csm=\sm)}\\
    &\le \sum_{\sm=\ckopt+1}^{M_n}\sP_{\star}^{(n)}\sbr[4]{\sum_{\k\in[\sm]}\frac{1}{\sm}\int_{\Xi_n}\exp\del[1]{-\lambda\barho^{-1}\lossdiff(\theta)}\d\Pi_{n,\k}(\theta|\cpsi_{n,\k})}.
    \end{align*}
We divide the model indices into $\cK_n:=\{\k\in[M_n]:\ccE_n(\k)\ge \xi_1 A_n(\ccE_n^*+\log M_n)\}$ and $\cK_n^\complement$.  For $\k\in\cK_n$, \cref{lemma:eb_lossdiff_bound} implies that
    \begin{align*}
        \sP_{\star}^{(n)}\sbr{\int_{\Xi_n}\exp\del[1]{-\lambda\barho^{-1}\lossdiff(\theta)}\d\Pi_{n,\k}(\theta|\cpsi_{n,\k})}
        &\le \sP_{\star}^{(n)}\sbr{\int\exp\del[1]{-\lambda\barho^{-1}\lossdiff(\theta)}\d\Pi_{n,\k}(\theta|\cpsi_{n,\k})}\\
        &\le \exp(-(2\barho)^{-1}\ccE_n(\k))\\
        &\le\exp(-(2\barho)^{-1}\xi_1 A_n(\ccE_n^*+\log M_n)\}).
    \end{align*}
On the other hand, for  $\k\in\cK_n^\complement$, by an argument similar to that in the proof of \cref{lemma:eb_lossdiff_bound}, which combines H\"older inequality and the notion of our R\'enyi-type hyperparameter cover $\Psib$, we have 
    \begin{align*}
        \sP_{\star}^{(n)} &\sbr{\int_{\Xi_n}\exp\del[1]{-\lambda\barho^{-1}\lossdiff(\theta)}\d\Pi_{n,\k}(\theta|\cpsi_{n,\k})}\\
         &\le \sum_{\psib\in\Psib}\exp(\rho R_{n,\k})  \cbr{\int_{\Xi_n}  \exp\del{-\xi_1\cL_n(\theta)}\d\Pi_{n,\k}(\theta|\psi^\bullet)}^{1/\barho}\\
          &\le \exp(U_{n,\k}+\rho R_{n,\k})  \exp(-\barho^{-1}\xi_1 A_n(\ccE_n^*+\log M_n))\\
         &\le \exp(-(2\barho)^{-1}\xi_1 A_n(\ccE_n^*+\log M_n)\}),
    \end{align*}
where the second and last inequality follow from he definition of $\Xi_n$ and the assumption that $U_{n,\k}+\rho  R_{n,\k}\le (2\barho)^{-1}\ccE_n(\k)\le (2\barho)^{-1}\xi_1 A_n(\ccE_n^*+\log M_n)$, respectively. In either case, we have the exponentially decaying upper bound, and therefore,
\begin{align*}
       \Pstar\sbr{\esaebpst\del{\Xi_n}\ind(\fM_n\cap \fA_n)}
        &\le (M_n)^2\exp\del[1]{\sqrt{A_n}  \ccE_n^*-(2\barho)^{-1}A_n(\ccE_n^*+\log M_n)},
    \end{align*}
 which tends to 0 since $A_n$ is eventually larger than $\sqrt{A_n}$.

Lastly, we focus on the third term in \eqref{eq:eb_decompose}. For each $\sm\in[M_n]$,  we define
    \begin{align*}
       \bar{\ell}_{n}(\sm):=-\frac{1}{\lambda\barho^{-1}}\log\del{ \int \exp(-\lambda\barho^{-1}\ell_n(\theta))\d\Pi_{n,\le\sm}(\theta)}.
    \end{align*}
Then we have the same decomposition as in \eqref{eq:variational_error}. Bounding the first term in the decomposition is almost similar, so we omit it. For the second term, using a similar argument used in the proof of \cref{thm:contract_posterior}, we have
  \begin{align*}
   \Pstar&\sbr{-\lambda\barho^{-1}(\bar{\ell}_{n}(\csm)-\ell_n^\star)\ind(\fM_n)}\\
        &\le   \log\del{ 1+\sum_{\sm=1}^{M_n}\sP_{\star}^{(n)}\sbr{\int \exp(-\lambda\barho^{-1}\lossdiff(\theta))\d\Pi_{n,\le\sm}(\theta)}}\\
        &=  \log\del{ 1+\sum_{\sm=1}^{M_n}\frac{1}{\sm}\sum_{\k=1}^{\sm}\sP_{\star}^{(n)}\sbr{\int \exp(-\lambda\barho^{-1}\lossdiff(\theta))\d\Pi_{n,\k}(\theta|\cpsi_{n,\k})}}\\
        &\le\log\del{ 1+\sum_{\sm=1}^{M_n}\frac{1}{\sm}\sum_{\k=1}^{\sm}\exp(-(2\barho)^{-1}\ccE_n(\k))}\\
        &\le \log (1+ M_n),
    \end{align*}
where we use \cref{lemma:eb_lossdiff_bound} for the second inequality, which completes the proof.
\end{proof}

\section{Proofs for \cref{sec:erm}}

\subsection{Proof of \cref{prop:risk_bound_erm}}

\begin{proof}
We first divide the parameter space $\Theta_{n,\k}$ into shells as
    \begin{align*}
        \Theta_{n,\k,j}:=\cbr{\theta\in\Theta_{n,\k}:
        2^{j-1}H_{n,\k}\le \xi_1 \cL_n(\theta)+H_{n,\k}< 2^jH_{n,\k} }
    \end{align*}
for $j\in\bN$.  Let $\Thetabb$ be a minimal $\zeta_{n,\k,j}:=(\xi_7n^{-1}\xi_1^{-1}2^{j-1}H_{n,\k})^{1/\kappa}$ cover of the set $\Theta_{n,\k,j}$ with $|\Thetabb|=N_{n,\k,j}$. Note that since $\Theta_{n,\k,j}\subset\{\theta\in\Theta_{n,\k}:
        \cL_n(\theta)\le   \xi_1^{-1}2^jH_{n,\k}\} $ and $\xi_1^{-1}2^jH_{n,\k}=2\xi_7^{-1}n\zeta_{n,\k,j}^\kappa $, \cref{assume:penalty} implies that
    \begin{align*}
         N_{n,\k,j}:=\cN(\zeta_{n,\k,j}, \Theta_{n,\k,j},\fh)\le \exp( H_{n,\k}/4).
    \end{align*}
For an absolute constant $\tC'_1$ we will specify later, we have
        \begin{align*}
        \max_{\theta\in \Theta_{n,\k,j}}\{\tC'\cL_n(\theta)- \lambda\lossdiff(\theta)\}
        &\le \tC'(2^j-1)H_{n,\k} 
        - \min_{\thetab\in \Thetabb}\lambda\lossdiff(\thetab)
        +\xi_1^{-1} \xi_6 \xi_72^{j-1}H_{n,\k} \\
        &\le (4\tC'+1) 2^{j-2}H_{n,\k} - \min_{\thetab\in \Thetabb}\lambda\lossdiff(\thetab),
    \end{align*}
where the second inequality is due to our assumption that $\xi_7<\xi_1/(2\xi_6)$.
Therefore, by the union bound and Markov's inequality, we have
 \begin{align*}
      \Pstar&\del{\max_{\theta\in \Theta_{n,\k,j}}\{\tC'\cL_n(\theta)- \lambda\lossdiff(\theta)\}\ge t}\\
       &\le \sum_{\thetab\in\Thetabb}\Pstar\del{- \lambda\lossdiff(\thetab)\ge t- (4\tC'+1) 2^{j-2}H_{n,\k}}\\
       &\le \exp((4\tC'+1) 2^{j-2}H_{n,\k}-t)  \sum_{\thetab\in\Thetabb}\Pstar\sbr{\exp(- \lambda\lossdiff(\thetab)))}\\
       &\le N_{n,\k,j}\exp((4\tC'+1) 2^{j-2}H_{n,\k}-t-(2^{j-1}-1)H_{n,\k}),
    \end{align*}
where we use \cref{assume:loss} and the fact that $\xi_1\cL_n(\theta)\ge 2^{j-1}H_{n,\k}-H_{n,\k}$ for any $\theta\in\Theta_{n,\k,j}$ for the last inequality. Since $\log (N_{n,\k,j})\le H_{n,\k}/4$, if we take $\tC'=1/4$,  the last display is further bounded as
    \begin{align*}
        \Pstar\del{\max_{\theta\in \Theta_{n,\k,j}}\{\tC'_1\cL_n(\theta)- \lambda\lossdiff(\theta)\}\ge t}
        \le\exp(H_{n,\k}/4 -t-2^{j-3}H_{n,\k}+H_{n,\k}).
    \end{align*}
By the union bound, we further have
    \begin{align}
        &\Pstar\del{\max_{\theta\in \Theta_{n,\k}}\{\tC'_1\cL_n(\theta)- \lambda\lossdiff(\theta)\}\ge t}\nonumber\\
         &\le\sum_{j=1}^\infty \exp(-t-2^{j-3}H_{n,\k}+ H_{n,\k}/4 +H_{n,\k})
          \lesssim\exp(H_{n,\k}-t).
          \label{eq:erm_sup_bound}
    \end{align}
Therefore, by applying the last display with taking $t=\log (\tC''n)+H_{n,\k}$ for sufficiently large $\tC''>0$, we have
   \begin{align*}
        \tC'\cL_n(\htheta_{n,\k})
        &\le \max_{\theta\in \Theta_{n,\k}}\{\tC'\cL_n(\theta)- \lambda\lossdiff(\theta)\} +\lambda\lossdiff(\htheta_{n,\k})\\ 
        &\le \lossdiff(\htheta_{n,\k}) +H_{n,\k}+\log(\tC''n)\\
        &\le \lossdiff(\theta_{n,\k}^*) +H_{n,\k}+\log(\tC''n)
    \end{align*}
with probability at least $1-1/(2n)$, where we denote $\theta^*_{n,\k}=\argmin_{\theta\in \Theta_{n,\k}}\cL_n(\theta)$. In the last display, the second inequality follows the optimization optimality of the ERM.
Moreover, by \cref{assume:loss}, we have
    \begin{align*}
    \Pstar\del{\lambda\xi_2\lossdiff(\theta^*_{n,\k})-\xi_3\cL_n(\theta^*_{n,\k})\le \log (2n)}
    \le \frac{1}{2n}\Pstar\sbr{\exp\del[1]{\lambda\xi_2\lossdiff(\theta^*_{n,\k})-\xi_3\cL_n(\theta^*_{n,\k})}}\le\frac{1}{2n}
    \end{align*}
Therefore, we have $\cL_n(\htheta_{n,\k})\lesssim \cL_n(\theta^*_{n,\k})+\log n$ with probability at least $1-n^{-1}$, which completes the proof.
\end{proof}

\subsection{Proof of \cref{thm:selection_erm}}

\begin{proof}
Let $\ukopt:=\ukopt(\utau_*)$. 
For each $\k$, we define
    \begin{align*}
        \Gamma_n(\k)=\mPER(\htheta_{n,\k})-\lambda \ell_n^\star
        = \lambda\lossdiff( \htheta_{n,\k})+H_{n,\k}
    \end{align*}    
so that our stopping model index can be  equivalently written as
    \begin{align*}
        \usm=\inf\cbr{\k\in[M_n]\setminus\{1\}:
        \uGamma_n(\k-1)< \uGamma_n(\k)}\wedge M_n.
    \end{align*}
As we did in the proof of \cref{thm:selection_vb}, it suffices to establish upper bounds of both $I_{n1}:=\Pstar\del{ \uGamma_n(\k)\ge \tC'\ucE_{n}(\k)}$ and $I_{n2}:=\sP_{\star}^{(n)}\del{ \uGamma_n(\k-1)\le \tC'\ucE_{n}(\k)}$ for some constant $\tC'>0$. 

To bound $I_{n2}$, we divide the parameter space $\Theta_{n,\k-1}$ into shells as
    \begin{align*}
        \Theta_{n,\k-1,j}:=\cbr{\theta\in\Theta_{n,\k-1}:
        2^{j-1}\ucE_n(\k-1) \le \xi_1 \cL_n(\theta)+H_{n,\k-1}< 2^j\ucE_n(\k-1) }
    \end{align*}
for $j\in\bN$. This peeling device is slightly different to what we have used for proving \cref{prop:risk_bound_erm}. Let $\Thetab$ be a minimal $\zeta_{n,\k-1,j}:=(\xi_7n^{-1}\xi_1^{-1}2^{j-1}\ucE_n(\k-1))^{1/\kappa}$ cover of the set $\Theta_{n,\k-1,j}$ with $|\Thetab|=N_{n,\k-1,j}$. Note that since $\Theta_{n,\k-1,j}\subset\{\theta\in\Theta_{n,\k-1}:
        \cL_n(\theta)\le   \xi_1^{-1}2^j\ucE_n(\k-1)\} $ and $\xi_1^{-1}2^j\ucE_n(\k-1)=2\xi_7^{-1}n\zeta_{n,\k-1,j}^\kappa $, \cref{assume:penalty} implies that
    \begin{align*}
         N_{n,\k-1,j}:=\cN(\zeta_{n,\k-1,j}, \Theta_{n,\k-1,j},\fh)\le \exp( H_{n,\k}/4)
    \end{align*}
Now we define the ``prior'' distribution
    \begin{align*}
        \Pi_{n,\k-1,j}:=\frac{1}{N_{n,\k-1,j}}\sum_{\theta^\bullet \in\Thetab}\delta_{\theta^\bullet} 
    \end{align*}
and the ``posterior'' distribution
    \begin{align*}
        \hQ_{n,\k}:=\delta_{\hthetab}
        \text{ with }
        \hthetab=\argmin_{\theta^\bullet\in\Thetab}\cd(\theta^\bullet, \htheta_{n,\k-1}),
    \end{align*}
where $\delta_\theta$ denotes the Dirac-delta measure at $\theta$.
Then, by \cref{assume:penalty}, the KL divergence between the posterior and the prior is bounded by
    \begin{align*}
        \kl(\hQ_{n,\k-1,j},\Pi_{n,\k-1,j})
        =\log N_{n,\k-1,j}\le H_{n,\k-1}.
    \end{align*}
On the event $\htheta_{n,\k-1}\in\Theta_{n,\k-1,j}$, we have that, by the Lipschitz continuity of the loss function,
    \begin{align*}
        \uGamma_n(\k-1)
        &\ge \lambda\lossdiff(\hthetab)+\log(N_{n,\k-1,j})
        - \xi_6n\fh^\kappa(\hthetab,\htheta_{n,\k-1})\\
        &=\int\lambda\lossdiff(\theta)\d\hQ_{n,\k-1,j}(\theta) + \kl(\hQ_{n,\k-1,j},\Pi_{n,\k-1,j}) -\xi_1^{-1}\xi_6 \xi_72^{j-1}\ucE_n(\k-1)\\
        &\ge - \log\del{\int  \exp\del[1]{-\lambda \lossdiff(\theta)}\d\Pi_{n,\k-1,j}(\theta)}-2^{j-2}\ucE_n(\k-1)
    \end{align*}
everywhere, where  the last inequality holds due to Donsker and Varadhan’s variational formula and the assumption that $\xi_7\le \xi_1/(2\xi_6)$. Then by Markov's inequality, we  have
  \begin{align*}   
        &\Pstar\del{ \uGamma_n(\k-1)\le \tC'\ucE_{n}(\k),\htheta_{n,\k-1}\in\Theta_{n,\k-1,j}}      \\
        &\le\Pstar\del{ \log\del{\int  \exp\del[1]{-\lambda \lossdiff(\theta)}\d\Pi_{n,\k-1,j}(\theta)}\ge -\tC'\ucE_{n}(\k)-2^{j-2}\ucE_n(\k-1)}  \\
        &\le \exp(\tC'\ucE_n(\k)+2^{j-2}\ucE_n(\k-1)) \Pstar\sbr{\int  \exp\del[1]{-\lambda \lossdiff(\theta)}\d\Pi_{n,\k-1,j}(\theta)}.
    \end{align*}
Moreover, using \cref{assume:loss} and the fact that $\cL_n(\theta)\ge 2^{j-1}\ucE_n(\k-1)-H_{n,\k-1}$ for any $\theta\in\Theta_{n,\k-1,j}$,
    \begin{align*}   
      \Pstar\sbr{\int  \exp\del[1]{-\lambda \lossdiff(\theta)}\d\Pi_{n,\k-1,j}(\theta)}
           &=\int  \exp\del{-\xi_1\cL_n(\theta)}\d\Pi_{n,\k-1,j}(\theta)\\
         &\le \exp\del{-2^{j-1}\ucE_n(\k-1)+H_{n,\k-1}}.
    \end{align*}
Since $H_{n,\k-1}\le H_{n,\k}$ in our model construction, we have
  \begin{align*}   
        \Pstar\del{ \Gamma_n(\k-1)\le \tC'\ucE_{n}(\k),\htheta_{n,\k-1}\in\Theta_{n,\k-1,j}}      
        &\le \exp((\tC'+1)\ucE_n(\k)-2^{j-2}\ucE_n(\k-1)) .
    \end{align*}
We therefore have
    \begin{align*}
       \sP_{\star}^{(n)}\del{ \uGamma_n(\k-1)\le \tC'\ucE_{n}(\k)}
       &\le\sum_{j=1}^\infty\Pstar\del{ \uGamma_n(\k-1)\le \tC'\ucE_{n}(\k),\htheta_{n,\k-1}\in\Theta_{n,\k-1,j}} \\
       &\le \exp((\tC'_1+1)\ucE_n(\k))\sum_{j=1}^\infty\exp(-2^{j-2}\ucE_n(\k-1)) \\
       &\lesssim \exp((\tC'+1)\ucE_n(\k)-\ucE_n(\k-1)/2) \\
      &\lesssim  \exp((\tC'+1)\ucE_n(\k)-(1+\tau_*)\ucE_n(\k)/2) ,
    \end{align*}
where the last inequality follows from  $\ucE_n(\k-1)\ge (1+\tau_*)\ucE_n(\k)$ for $\k\in[\ukopt]$.

We proceed to bound $I_{n1}$. By the optimization optimality of the ERM, we have
    \begin{align*}
        \uGamma_n(\k)
        &\le \lambda \lossdiff(\theta^*_{n,\k}) + H_{n,\k}
       \le  \lambda \lossdiff(\theta^*_{n,\k}) + \ucE_n(\k),
    \end{align*}
where we denote by $\theta^*_{n,\k}=\argmin_{\theta\in \Theta_{n,\k}}\cL_n(\theta)$. Therefore, by applying Markov's inequality and \cref{assume:loss} in order, we have
    \begin{align*}
        \Pstar\del{ \uGamma_n(\k)\ge \tC'\ucE_{n}(\k)}
          &\le  \Pstar\del{ \lambda \lossdiff(\theta^*_{n,\k}) \ge (\tC'-1)\ucE_n(\k)}\\
         &\le   \exp\del[2]{-\xi_2(\tC'-1)\ucE_{n}(\k)}\Pstar\sbr{\exp\del{\lambda\xi_2\lossdiff(\theta^*_{n,\k})}}\\
           &\le   \exp\del[2]{-\xi_2(\tC'-1)\ucE_{n}(\k)+\xi_3\cL_n(\theta^*_{n,\k}) }\\
         &\le  \exp\del[2]{-\xi_2(\tC'-1-\xi_3/(\xi_2\xi_1))\ucE_n(\k)}.  
    \end{align*}
Hence, if we take $\tC'$ and $\tau_*$ sufficiently large, we get the desired result.
\end{proof}

\subsection{Proof of \cref{thm:oracle_erm}}

\begin{proof}

Let $\ukopt:=\ukopt(\utau_*)$ and $\fM_n:=\{ X^{(n)}\in\cX_n:\usm\ge \ukopt+1 \}$. We also define the event
    \begin{align}
        \text{$\fU_n$}
        :=
        \bigcap_{\k=1}^{M_n}\cbr{X^{(n)}\in\cX_n:\max_{\theta\in \Theta_{n,\k}}\{\tC'_1\cL_n(\theta)- \lambda\lossdiff(\theta)\}\le  H_{n,\k}+\log(\tC'M_nn) }.
    \end{align}
for sufficiently large $\tC'>0$. Then, by \eqref{eq:erm_sup_bound} we have established in the proof of \cref{prop:risk_bound_erm} together with the union bound, we have $\Pstar(\fU_n^\complement)\le 1/(2n) $. On the event $\fU_n$, by the convexity of the excess risk function $\cL_n$, we have
    \begin{align*}
        \cL_n(\esaerm)
        &\le \sum_{k=1}^{\usm}\omegafesa\cL_n(\htheta_{n,\k})\\
           &\lesssim\sum_{k=1}^{\usm}\omegafesa\cbr{\lambda\lossdiff(\htheta_{n,\k}) 
        + H_{n,\k} }+\log(M_nn).
    \end{align*}
Note that Jensen's inequality implies that
    \begin{align*}
       \sum_{k=1}^{\usm}\omegafesa\log(\omegaesa)+\log(\usm) =\sum_{k=1}^{\usm}\omegafesa\log\del{\frac{\omegafesa}{\usm^{-1}}}\le 0.
    \end{align*}
Therefore, we have
    \begin{align*}
       \sum_{k=1}^{\usm}\omegafesa\cbr{\lambda\lossdiff(\htheta_{n,\k}) 
        + H_{n,\k} }
        &\le 
        \sum_{k=1}^{\usm}\omegafesa\cbr{\log(\omegafesa)+\lambda\lossdiff(\htheta_{n,\k}) 
        + H_{n,\k} }+\log(\usm) \\
        &=\sum_{k=1}^{\usm}\omegafesa\log\del{\frac{\omegafesa}{\exp(-\lambda\ell_n(\htheta_{n,\k}) - H_{n,\k} )}}-\lambda\ell_n^\star+\log(\usm).
    \end{align*}
Since $\omegafesa\propto \exp(-\lambda\ell_n(\htheta_{n,\k}) - H_{n,\k})$ is the minimizer of a function 
    \begin{align*}
        (\omega_1,\dots, \omega_{\sm})\mapsto \sum_{k=1}^{\sm}\omega_{\k}\log\del{\frac{\omega_{\k}}{\exp(-\lambda\ell_n(\htheta_{n,\k}) - H_{n,\k} )}},
    \end{align*}
we obtain the following series of inequalities holding on the event $ \fU_n$
   \begin{align*}
        \cL_n(\esaerm)
        &\lesssim \min_{\k\in[\usm]}\cbr{\lambda\lossdiff(\htheta_{n,\k}) +H_{n,\k}  }+\log(M_nn)\\
        &\le \min_{\k\in[\usm]}\cbr{\lambda \lossdiff(\theta_{n,\k}^*) +H_{n,\k}}+\log(M_nn)
    \end{align*}
with $\theta^*_{n,\k}=\argmin_{\theta\in \Theta_{n,\k}}\cL_n(\theta)$, where the second inequality follows the optimization optimality of the ERM and the third one holds since $\ukopt\le \usm$ on the event $\fM_n$. Next, we define the event
     \begin{align*}
     \text{$\fV_n$}
     := \bigcap_{\k=1}^{M_n}\cbr{X^{(n)}\in\cX_n: \lambda\xi_2\lossdiff(\theta^*_{n,\k})-\xi_3\cL_n(\theta^*_{n,\k})\le \log (2M_n n)}.
    \end{align*}
By \cref{assume:loss} and the use of the union bound, we have $\Pstar(\fV_n^\complement)\le 1/n$. Then, on the event $\fM_n \cap \fU_n \cap \fV_n$, we finally have
   \begin{align*}
        \cL_n(\esaerm)
        &\lesssim\min_{\k\in[\usm]}\cbr{\cL_n(\theta_{n,\k}^*) +H_{n,\k}}+\log(M_nn)
        \lesssim \ucE_n^*+\log (M_n n).
    \end{align*}
We complete the proof by noting that the probability of the complement of the event $\fM_n\cap \fU_n \cap \fV_n$ is bounded above by
$\Pstar(\fU_n^\complement)+\Pstar(\fV_n^\complement)+\Pstar(\fM_n^\complement)\le n^{-1}+\ueta_n$.
\end{proof}

\section{Proofs for \cref{appendix:freq_etc}}

\subsection{Proof of \cref{thm:selection_ss}}

\begin{proof}
Let $\tkopt:=\tkopt(\ttau_*)$. Note that our  stopping model index can be  equivalently written as
    \begin{align}
        \tsm=\inf\cbr{\k\in[M_n]\setminus\{1\}:
        \lossdiff(\ttheta_{n,\k-1})< \lossdiff(\ttheta_{n,\k})}\wedge M_n.
    \end{align}
We start with the following series of inequalities
 \begin{align*}
       \Pstar\del[1]{\tsm< \tkopt+1}
       &\le \sum_{\k=1}^{\tkopt}\Pstar\del{\tsm= \k}\\
       &\le \sum_{\k=1}^{\tkopt}\Pstar\del{\lossdiff(\ttheta_{n,\k-1})< \lossdiff(\ttheta_{n,\k})}\\
       &\le \sum_{\k=1}^{\tkopt}\sbr{ \Pstar\del{\lossdiff(\ttheta_{n,\k-1})<\tC'\cL_n(\ttheta_{n,\k})}
         + \Pstar\del{\lossdiff(\ttheta_{n,\k})\ge\tC'\cL_n(\ttheta_{n,\k})}}
    \end{align*}
with some constant $\tC'>0$ we will specify later. The first term in the summand can be bounded as, by \cref{assume:loss},
 \begin{align*}
             \Pstar\del{\lossdiff(\ttheta_{n,\k-1})<\tC'\cL_n(\ttheta_{n,\k})}
            & \le \exp\del{\lambda\tC'\cL_n(\ttheta_{n,\k})}
             \Pstar\sbr{\exp(-\lambda\lossdiff(\ttheta_{n,\k-1})}\\
             &\le\exp\del{\lambda\tC'\cL_n(\ttheta_{n,\k})-\xi_1\cL_n(\ttheta_{n,\k-1})}\\
              &\le\exp\del{\lambda\tC'\cL_n(\ttheta_{n,\k})-\xi_1(1+\ttau_*)\cL_n(\ttheta_{n,\k})}
        \end{align*}
where we use the fact that $\cL_n(\ttheta_{n,\k-1})>(1+\ttau_*)\cL_n(\ttheta_{n,\k})$ for $\k\in[\tkopt]$, and the second term can be bounded as,  by \cref{assume:loss} again,
    \begin{align*}
         \Pstar\del{\lossdiff(\ttheta_{n,\k})\ge\tC'\cL_n(\ttheta_{n,\k})}
        & \le \exp(-\tC'\cL_n(\ttheta_{n,\k}))\Pstar\sbr{\exp(\xi_2\lambda\lossdiff(\ttheta_{n,\k})}\\
         & \le \exp(-(\tC'-\xi_3)\cL_n(\ttheta_{n,\k})).
    \end{align*}
Hence if take sufficiently large $\tC'$ and $\ttau_*$, we get the desired result.
\end{proof}

\subsection{Proof of \cref{thm:oracle_ss}}

\begin{proof}

Let $\tkopt:=\tkopt(\ttau_*)$ and define the event $\fM_n:=\{ X^{(n)}\in\cX_n:\tsm\ge \tkopt+1 \}$. Note that $X^{(n)}$ denotes the aggregation sample. We define two events as
     \begin{align*}
      \text{$\fU_n$}
        &:=
        \bigcap_{\k=1}^{M_n}\cbr{X^{(n)}\in\cX_n:\xi_1\cL_n(\ttheta_{n,\k}) -\lambda\lossdiff(\ttheta_{n,\k})\le  \log(2M_nn) },\\
     \text{$\fV_n$}
     &:= \bigcap_{\k=1}^{M_n}\cbr{ X^{(n)}\in\cX_n:\lambda\xi_2\lossdiff(\ttheta_{n,\k})-\xi_3\cL_n(\ttheta_{n,\k})\le \log (2M_n n)}.
    \end{align*}
Then by \cref{assume:loss} together with the union bound, we have
    \begin{align*}
        \Pstar(\fU_n^\complement)
        \le M_n\e^{-\log(2M_nn)}\Pstar\sbr{\exp\del{\xi_1\cL_n(\ttheta_{n,\k}) -\lambda\lossdiff(\ttheta_{n,\k})}}\le \frac{1}{2n}
    \end{align*}
and
    \begin{align*}
        \Pstar(\fV_n^\complement)
        \le M_n\e^{-\log(2M_nn)}\Pstar\sbr{\exp\del{\lambda\xi_2\lossdiff(\ttheta_{n,\k})-\xi_3\cL_n(\ttheta_{n,\k})}}\le \frac{1}{2n}.
    \end{align*}
On the event $\fU_n$, we have, by the convexity of the excess risk function $\cL_n$, 
    \begin{align*}
        \cL_n(\esass)
        &\le \sum_{k=1}^{\tsm}\omegass\cL_n(\ttheta_{n,\k})\\
        &\lesssim\sum_{k=1}^{\tsm}\omegass\cbr{\alpha\lossdiff(\htheta_{n,\k}) 
         }+\log(M_nn).
    \end{align*}
By Jensen's inequality,
    \begin{align*}
       \sum_{k=1}^{\tsm}\omegass\log(\omegass)+\log(\tsm) =\sum_{k=1}^{\tsm}\omegass\log\del{\frac{\omegass}{\tsm^{-1}}}\le 0.
    \end{align*}
Therefore, we have, on the event $\fU_n$,
    \begin{align*}
      \sum_{k=1}^{\tsm}\omegass\cbr{\alpha\lossdiff(\htheta_{n,\k}) 
         }
        &\le 
        \sum_{k=1}^{\tsm}\omegass\cbr{\log(\omegass)+\alpha\lossdiff(\ttheta_{n,\k}) }+\log(\tsm) \\
        &=\sum_{k=1}^{\tsm}\omegass\log\del{\frac{\omegass}{\exp(-\alpha\ell_n(\ttheta_{n,\k}))}}-\alpha\ell_n^\star+\log(M_n)\\
        &\le \alpha\min_{\k\in[\tsm]}\lossdiff(\ttheta_{n,\k})+\log(M_n).
    \end{align*}
Moreover, on the event $\fM_n \cap \fU_n\cap \fV_n$, we have
  \begin{align*}
        \cL_n(\esass)
        &\lesssim \min_{\k\in[\tsm]}\lossdiff(\ttheta_{n,\k})+\log(M_nn)\\
        &\lesssim \min_{\k\in[\tsm]}\cL_n(\ttheta_{n,\k})+\log(M_nn)\\
         &\le \min_{\k\in[\tkopt+1]}\cL_n(\ttheta_{n,\k})+\log(M_nn),
    \end{align*}
which completes the proof.
\end{proof}

\subsection{Proof of \cref{thm:oracle_vbmean}}
\begin{proof}
 By the convexity of the excess risk function $\cL_n$, we have
    \begin{align*}
        \cL_n(\esaerm)
        \le   \sum_{k=1}^{\hsm}\omegaesa\int \cL_n(\theta) \d\hQ_{n,\k}(\theta)
        =\int \cL_n(\theta)\d\esavp(\theta).
    \end{align*}
Let $\kopt:=\kopt(\tau_*)$ and $\fM_n:=\{ X^{(n)}\in\cX_n:\hsm\ge \kopt+1 \}$. Moreover, let  $g(\theta):=\xi_1\cL_n(\theta)-\lambda\lossdiff(\theta)$. Then for any $t>0$, we have
    \begin{align*}
          \Pstar&\del{  \sup_{Q\in\cP(\Theta_n)} \cbr{\int g(\theta)\d Q(\theta)-\kl(Q,\Pi_{n,\le\hsm})}\ge t, \hsm\ge \kopt+1}\\
          &= \Pstar\del{ \log\del{\int \exp(g(\theta))\d\Pi_{n,\le\hsm}}\ge t,\hsm\ge \kopt+1}\\
          &\le  \sum_{\sm=\kopt+1}^{M_n}\Pstar\del{ \log\del{\int \exp(g(\theta))\d\Pi_{n,\le \sm}}\ge t}\\
          &\le  \exp(-t)\sum_{\sm=\kopt+1}^{M_n}\Pstar\sbr{\exp(\xi_1\cL_n(\theta)-\lambda\lossdiff(\theta))}\\
          &\le M_n\exp(-t),
    \end{align*}
where we use \cref{lemma:variational_formula} for the first line, the union bound for the second, Markov's inequality for the third, and \cref{assume:loss} for the last. As $\esavp\in \cP(\Theta_n)$, the last display implies that, with taking $t=\log(2M_n n)$, we have
    \begin{align*}
        \int \xi_1\cL_n(\theta)\d\esavp(\theta)
        \le \int \lambda\lossdiff(\theta)\d\esavp(\theta)+
        \kl(Q,\Pi_{n,\le\hsm})+\log(M_n n)
    \end{align*}
with probability at least $1-1/(2n)-\P(\fM_n^\complement)$. Moreover, by \cref{lemma:composite_elbo_max}, the last display is further bounded as
    \begin{align*}
     \int \xi_1\cL_n(\theta)\d\esavp(\theta) &\le \inf_{\k\in[\kopt+1]}\cbr{\int  \lambda\lossdiff(\theta)\d \hQ_{n,\k}(\theta)+
        \kl(\hQ_{n,\k},\Pi_{n,\k}) }+2\log (M_nn)\\
        &\le\int  \lambda\lossdiff(\theta)\d \hQ_{n,\koracle}(\theta)+
        \kl(\hQ_{n,\koracle},\Pi_{n,\koracle})+ 2\log (M_nn)\\
        &\le\int  \lambda\lossdiff(\theta)\d Q^*_{n,\koracle}(\theta)+
        \kl(Q^*_{n,\koracle},\Pi_{n,\koracle})+ 2\log (M_nn).
    \end{align*}
Moreover, we have  that
    \begin{align*}
        \Pstar&\del{\int \{\xi_2\lambda\lossdiff(\theta)-\xi_4\cL_n(\theta)\}\d Q^*_{n,\koracle}(\theta)+\kl(Q^*_{n,\koracle},\Pi_{n,\koracle})\ge \log(2n)}\\
        &\le\Pstar\del{\int \exp(\xi_2\lambda\lossdiff(\theta)-\xi_4\cL_n(\theta))\d \Pi_{n,\koracle}(\theta)\ge 2n}\\
         &\le \frac{1}{2n}\int \Pstar[\exp(\xi_2\lambda\lossdiff(\theta)-\xi_4\cL_n(\theta))]\d \Pi_{n,\koracle}(\theta)\le  \frac{1}{2n},
    \end{align*}
where we use \cref{lemma:variational_formula,lemma:exp_moment_bound} for the first and last inequalities, respectively. Therefore, the following holds
 \begin{align*}
    \int \cL_n(\theta)\d\esavp(\theta) 
        &\lesssim \int\cL_n(\theta) \d Q^*_{n,\koracle}(\theta)+
        \kl(Q^*_{n,\koracle},\Pi_{n,\koracle})+ \log (M_nn)\\
        &\lesssim  \cE_n(\koracle)+\log (M_nn)
    \end{align*}
with probability at least $1-1/n-\P(\fM_n^\complement)$, which completes the proof.
\end{proof}

\section{Additional experiment: LoRA fine-tuning for large language models}

In this experiment, we investigate ESA in the context of fine-tuning a large language model (LLM). We consider two standard language modeling benchmarks: WikiText-2 (WT2) and WikiText-103 (WT103). Both datasets consist of high-quality Wikipedia articles that preserve natural language structure, including punctuation, capitalization, and long-range dependencies. Each dataset is split into training, validation, and test sets following standard conventions and is widely used for evaluating autoregressive language models and parameter-efficient fine-tuning methods.

\paragraph{LoRA fine-tuning.}
As the base model, we use a pretrained GPT-2 language model, whose parameters remain frozen throughout all experiments. Fine-tuning is performed using Low-Rank Adaptation (LoRA; \citepS{hu2022lora}), in which trainable low-rank matrices are injected into selected linear transformations of the pretrained model while keeping the base parameters fixed. Specifically, for a weight matrix $W_0 \in \mathbb{R}^{d_{\text{out}} \times d_{\text{in}}}$ in the pretrained LLM, LoRA introduces a rank-$r$ update
\[
W = W_0 + BA,
\qquad
A \in \mathbb{R}^{r \times d_{\text{in}}},
\quad
B \in \mathbb{R}^{d_{\text{out}} \times r},
\]
where only the LoRA matrices $A$ and $B$ are optimized during fine-tuning.

\paragraph{ESA-LoRA.}
To apply ESA, we consider a ladder of LoRA candidates ordered by increasing expressive capacity, determined by the LoRA rank and the set of target modules.  \cref{tab:lora_configs} summarizes the candidate configurations.
For each LoRA candidate indexed by $\k$, let $\Theta_{n,\k}$ denote the parameter space consisting of all LoRA matrices.
The LoRA parameters $\htheta_{n,\k}$ are trained the standard causal language modeling negative log-likelihood
\begin{align*}
\htheta_{n,\k}
=
\argmin_{\theta\in\Theta_{n,\k}}
\Bigg\{
-
\sum_{x\in \cD_{\text{train}}}
\sum_{t=1}^{T_i}
\log p_{\theta}(x_{t}\mid x_{<t})
\Bigg\},
\end{align*}
where $p_\theta$ denotes the conditional next-token probability induced by the LLM with LoRA parameter $\theta$, $\cD_{\text{train}}$ denotes the training data set, $x=(x_{1},\dots,x_{T_i})$ denotes a tokenized sequence, and $x_{<t}=(x_{1},\dots,x_{t-1})$ denotes the preceding context.

To compare candidates of different capacity, we evaluate each trained candidate on the validation split using a penalized empirical risk
\[
\mPER(\sk)
=
\ell_{n}(\htheta_{n,\k})
+
H_{n,\sk},
\qquad
H_{n,\sk}
=
\frac{\|\htheta_{n,\k}\|_2^2}{P_{\mathrm{ref}}}
\left(\frac{P_{\sk}}{P_{\mathrm{ref}}}\right),
\]
where $\ell_{n}(\htheta_{n,\k})$ is the token-average validation negative log-likelihood of the adapted model,
$\|\htheta_{n,\k}\|_2^2$ is the squared $L_2$ norm of the learned LoRA parameters,
$P_{\sk}$ is the number of trainable LoRA parameters in candidate $\sk$,
$P_{\mathrm{ref}}=10^4$ is a fixed reference scale. 
The ESA procedure sequentially explores the candidate ladder and stops once increasing model capacity no longer improves the criterion $\mPER(\sk)$ as we described in \cref{sec:erm}.

\paragraph{Baselines and results}
For comparison, we consider full aggregation and best model selection as we did in the other experiments. We also consider a fixed LoRA adapter with a larger training budget (LoRA-Large), which uses the same configuration as the $\sk=4$ candidate but is trained for a substantially larger number of optimization steps. This provides a strong LoRA fine-tuning baseline. Lastly, we report the results for no-fine-tuned models. 

For each method, we report its token-level negative log-likelihood (NLL) and next-token accuracy (Acc) on the test set, as well as the training time. Across both WT2 and WT103, the fixed LoRA with larger training budget  achieves the strongest predictive accuracy. However, ESA substantially reduces the total wall-clock training time by stopping before reaching the highest-capacity adapters.

\begin{table}[t]
\centering
\caption{LoRA candidate configurations used in the experiments.
All adapters use LoRA scaling $\alpha=r$ and zero LoRA dropout.
Here, \texttt{q\_proj} and \texttt{v\_proj} denote the linear transformations that generate the query and value representations used in the self-attention computation, while \texttt{out\_proj} denotes the output projection that maps the resulting attention representation back to the model's hidden space.}
\label{tab:lora_configs}
\begin{tabular}{ccc}
\hline
Candidate & Rank $r$ & Target modules \\
\hline
$\sk=1$ & 4  & $\{\texttt{q\_proj}, \texttt{v\_proj}\}$ \\
$\sk=2$ & 8  & $\{\texttt{q\_proj}, \texttt{v\_proj}\}$ \\
$\sk=3$ & 8  & $\{\texttt{q\_proj}, \texttt{v\_proj}, \texttt{out\_proj}\}$ \\
$\sk=4$ & 16 & $\{\texttt{q\_proj}, \texttt{v\_proj}, \texttt{out\_proj}\}$ \\
$\sk=5$ & 32 & $\{\texttt{q\_proj}, \texttt{v\_proj}, \texttt{out\_proj}\}$ \\
\hline
\end{tabular}
\end{table}

\begin{table}[t]
\centering
\caption{Results of LoRA fine-tuning on WikiText-2 (WT2) and WikiText-103 (WT103) using GPT-2. Test negative log-likelihood (NLL), next-token accuracy (Acc), and total training time (minutes) are reported as mean $\pm$ standard deviation over three runs with random seeds 2025, 2026, and 2027. Lower NLL and time indicate better performance, while higher accuracy is preferred.}
\label{tab:lora_results}
\begin{tabular}{llccc}
\hline
Dataset & Method & Test NLL $\downarrow$ & Test Acc $\uparrow$ & Time (min) $\downarrow$ \\
\hline
\multirow{5}{*}{WT2}
& ESA        & 3.2728 $\pm$ 0.0002 & 0.4024 $\pm$ 0.0001 & \textbf{160.8 $\pm$ 68.4} \\
& FA         & 3.2708 $\pm$ 0.0002 & 0.4028 $\pm$ 0.0002 & 202.8 $\pm$ 85.4 \\
& MS         & 3.2701 $\pm$ 0.0001 & 0.4028 $\pm$ 0.0001 & 202.8 $\pm$ 85.4 \\
& LoRA-Large & \textbf{3.2263 $\pm$ 0.0003} & \textbf{0.4087 $\pm$ 0.0002} & 235.3 $\pm$ 74.8 \\
& Base       & 3.7750 $\pm$ 0.0000 & 0.3375 $\pm$ 0.0000 & 0.0 $\pm$ 0.0 \\
\hline
\multirow{5}{*}{WT103}
& ESA        & 3.0474 $\pm$ 0.0052 & 0.4280 $\pm$ 0.0007 & \textbf{374.5 $\pm$ 42.1} \\
& FA         & 3.0454 $\pm$ 0.0058 & 0.4283 $\pm$ 0.0008 & 464.6 $\pm$ 47.6 \\
& MS         & 3.0483 $\pm$ 0.0005 & 0.4279 $\pm$ 0.0001 & 464.6 $\pm$ 47.6 \\
& LoRA-Large & \textbf{2.9767 $\pm$ 0.0002} & \textbf{0.4374 $\pm$ 0.0001} & 467.4 $\pm$ 36.7 \\
& Base       & 3.3800 $\pm$ 0.0000 & 0.3869 $\pm$ 0.0000 & 0.0 $\pm$ 0.0 \\
\hline
\end{tabular}
\end{table}

\section{Further details of numerical experiments}

\subsection{Implementation details of the image classification experiment}

\subsubsection{Candidate neural network models}
\label{appendix:nn_models}

\paragraph{Bayesian CNN ladder for MNIST and Fashion-MNIST.}
For grayscale image datasets, we use a compact family of Bayesian convolutional neural networks whose complexity increases monotonically with the width of the hidden feature maps. Each model consists of one or two convolutional blocks followed by a Bayesian linear classifier. Specifically, the architecture takes the form
\[
x \;\mapsto\; \mathrm{ConvBlock}(c_1) \;\mapsto\; \mathrm{ConvBlock}(c_2)\;(\text{optional}) \;\mapsto\; \mathrm{Flatten} \;\mapsto\; \mathrm{BayesianLinear}(K),
\]
where \(K\) denotes the number of classes. Each convolutional block is defined as
\[
\mathrm{BayesianConv2d}(3\times 3)\;\to\;\mathrm{BatchNorm}\;\to\;\mathrm{ReLU}.
\]
All convolutional and linear layers are Bayesian and therefore contribute layer-wise KL terms to the variational objective. We consider a model ladder indexed by the channel widths of the convolutional blocks. The smallest models contain a single convolutional block, while the larger models contain two blocks:
\begin{itemize}
    \item \texttt{Small1:} one convolutional block with width \(c_1=1\).
    \item \texttt{Small2:} one convolutional block with width \(c_1=2\).
    \item \texttt{Medium1:} two convolutional blocks with widths \((c_1,c_2)=(2,4)\).
    \item \texttt{Medium2:} two convolutional blocks with widths \((c_1,c_2)=(4,8)\).
    \item \texttt{Medium3:} two convolutional blocks with widths \((c_1,c_2)=(8,16)\).
\end{itemize}

\paragraph{Bayesian ResNet ladder for CIFAR and Tiny-ImageNet.}
For color image datasets, we use Bayesian ResNet architectures built from standard residual basic blocks. Each block contains two Bayesian \(3\times 3\) convolutions, with an optional batch-normalization layer after each convolution and a ReLU activation after the first convolution and after the residual addition. Writing \(F(x)\) for the residual branch and \(S(x)\) for the shortcut connection, the block output is
\[
\mathrm{Block}(x)=\mathrm{ReLU}\bigl(F(x)+S(x)\bigr),
\]
where \(S(x)\) is the identity when the input and output dimensions agree, and otherwise is implemented by a Bayesian \(1\times 1\) convolution to match spatial resolution and channel dimension. The full network begins with an initial Bayesian \(3\times 3\) convolution, followed by three residual stages, global average pooling, and a Bayesian linear classifier. The widths of three residual stages are fixed at \([16,32,64]\), while model complexity is increased by increasing the number of residual blocks per stage. Specifically, we consider the following depth ladder, following the CIFAR experimental setting in Section 4.2 of \citetS{he2016deep}:
\begin{itemize}
    \item \texttt{ResNet20:} block configuration \([3,3,3]\).
    \item \texttt{ResNet32:} block configuration \([5,5,5]\).
    \item \texttt{ResNet44:} block configuration \([7,7,7]\).
    \item \texttt{ResNet56:} block configuration \([9,9,9]\).
    \item \texttt{ResNet110:} block configuration \([18,18,18]\).
\end{itemize}
All color-dataset models use batch normalization, and the second and third residual stages downsample the feature maps by using stride \(2\) in the first block of each stage.
\subsubsection{Tuning of inverse temperature parameter}
\label{appendix:inv_temp}

We calibrate the inverse temperature parameter $\lambda$ using a short probe run on the first two models in the ladder. For a candidate value of $\lambda$, let $\Vfe^{\sk}_t(\lambda)$ denote the VFE at the $t$-th training epoch of the $\sk$-th neural network in the model ladder. To summarize the recent training dynamics, we compute the average slope of the VFE over the last $m$ epochs:
    \begin{align*}
        \text{slope}_{\sk}(\lambda):=
\frac{1}{m}\sum_{t=1}^{m}\cbr{\Vfe^{\sk}_{T-t+1}(\lambda)-\Vfe^{\sk}_{T-t}(\lambda)},
    \end{align*}
where $T$ denotes the total number of the training epochs.
We select $\lambda_*$ so that the VFE slopes of the first two models in the ladder are matched:
    \begin{align*}
        \lambda_*=\argmin_{\lambda}\abs{ \text{slope}_1(\lambda) -  \text{slope}_2(\lambda) }.
    \end{align*}
This choice aligns the local training dynamics of the two smallest models, providing a data-dependent calibration of the inverse temperature parameter. For color image datasets using the ResNet ladder, we further apply a simple architecture-based scaling,
    \begin{align*}
        \lambda_{**}
        :=\lambda_*
        \left(\frac{P_0}{P}\right)^{1/2}\left(\frac{C_0}{C}\right)^{1/2},
        \qquad
        P_0 = 10^4,\; C_0 = 10,
    \end{align*}
where $P$ is the number of parameters of the baseline model and $C$ is the number of classes. This adjustment compensates for changes in model size and task complexity when transferring the calibrated temperature across architectures.

\subsection{Implementation details of the empirical Bayes experiments}

For the experiment, we follow the implementation details in Section~5.2 of \citetS{ray2021variational}. For \texttt{sparsevb}, we use a Laplace slab prior with hyperparameters $(\alpha_0,b_0,\lambda)=(1,p,1)$, with \texttt{maxiter}$=10^4$ and \texttt{tol}$=10^{-5}$. For \texttt{varbvs}, we set \texttt{tol}$=10^{-4}$ and \texttt{maxiter}$=10^4$. For \texttt{SSLASSO}, we set $\lambda_1=0.01$ and take $\lambda_0$ to be an arithmetic sequence of 200 values from $\lambda_1$ to $p$; we use \texttt{variance="unknown"}, $a=1$, $b=p$, and \texttt{penalty="adaptive"}, and report the stabilized $\lambda_0$ solution recommended by the authors. LASSO and SCAD are tuned by cross-validation using \texttt{cv.glmnet} and \texttt{cv.ncvreg}, respectively, and we report the corresponding \texttt{lambda.min} fits.

\bibliographystyleS{plainnat}
\bibliographyS{_references}

\end{appendices}

\end{document}

%% file: _preamble.tex
\usepackage{amsfonts,amsmath,amsthm,mathtools,bbm,bm,eucal}
\usepackage{times,amssymb,mathabx} 
\usepackage[notrig]{physics}
\usepackage{graphicx,caption,subcaption,hyperref,enumitem,multirow} 
\usepackage[ruled,vlined,linesnumbered]{algorithm2e}
\usepackage[nameinlink,capitalize,noabbrev]{cleveref}
\usepackage{crossreftools,multibib} 
\usepackage[authoryear]{natbib}
\usepackage[title,toc]{appendix}
\usepackage[dvipsnames,x11names]{xcolor}
\usepackage{booktabs,makecell} 
\usepackage{tikz}

\newcites{S}{References}
\hypersetup{colorlinks=True, urlcolor=Blue, citecolor=MidnightBlue, linkcolor=RedViolet}

\numberwithin{equation}{section}

\newtheorem{theorem}{Theorem}[section]
\newtheorem{lemma}{Lemma}[section]
\newtheorem{proposition}{Proposition}
\newtheorem{corollary}[theorem]{Corollary}
\theoremstyle{definition}
\newtheorem{definition}{Definition}
\newtheorem{assumption}{Assumption}

\newtheorem{examplex}{Example}
\newenvironment{example}
  {\pushQED{\qed}\examplex}
  {\popQED\endexamplex}
\newtheorem{remarkx}{Remark}
\newenvironment{remark}
  {\pushQED{\qed}\remarkx}
  {\popQED\endremarkx}

\crefname{assumption}{Assumption}{Assumptions}
\crefname{examplex}{Example}{Examples}
\crefname{theorem}{Theorem}{Theorems}
\crefname{lemma}{Lemma}{Lemmas}
\crefname{section}{Section}{Sections}
\crefname{algorithm}{Algorithm}{Algorithms}
\crefname{enumi}{Assumption}{Assumptions}
\crefname{equation}{Inequality}{Inequalities}
\Crefname{equation}{Equation}{Equations}

\makeatletter
\newcommand*{\rom}[1]{\expandafter\@slowromancap\romannumeral #1@}
\makeatother

\newcommand{\genericdel}[4]{%
  \ifcase#3\relax
  \ifx#1.\else#1\fi#4\ifx#2.\else#2\fi\or
  \bigl#1#4\bigr#2\or
  \Bigl#1#4\Bigr#2\or
  \biggl#1#4\biggr#2\or
  \Biggl#1#4\Biggr#2\else
  \left#1#4\right#2\fi
}
\newcommand{\del}[2][-1]{\genericdel(){#1}{#2}}
\newcommand{\cbr}[2][-1]{\genericdel\{\}{#1}{#2}}
\newcommand{\sbr}[2][-1]{\genericdel[]{#1}{#2}}

 \let\abs\Abs

\def\d{\textup{d}}
\def\e{\textup{e}}
\def\k{\textsf{\textup{k}}}

\def\P{\textsf{\textup{P}}}
\def\R{\mathbb{R}}

\DeclareMathOperator*{\argmin}{\arg\!\min}

\def\iidsim{\stackrel{\textup{iid}}{\sim}}
\def\indsim{\stackrel{\textup{ind}}{\sim}}

\def\ind{\mathbbm{1}}
\def\kl{\textup{KL}}

\def\diag{\textup{diag}}

\def\Pstar{\textsf{\textup{P}}_\star^{(n)}}
\def\Ptheta{\textsf{\textup{P}}_\theta^{(n)}}
\def\lossdiff{\ell_n^\triangle}
\def\ceta{\check{\eta}}
\def\ueta{\underline{\eta}}
\def\teta{\tilde{\eta}}
\def\ccE{\check{\mathcal{E}}}
\def\chQ{\check{Q}}
\def\esa{\texttt{\textup{ESA}}}

\def\hsm{\widehat{\textsf{\textup{m}}}_n}
\def\csm{\check{\textsf{\textup{m}}}_n}
\def\usm{\underline{\textsf{\textup{m}}}_n}
\def\tsm{\tilde{\textsf{\textup{m}}}_n}

\def\kopt{\sk^\dag_n}
\def\kover{\sk^\circ_n}
\def\ckopt{\check{\sk}^\dag_n}
\def\ukopt{\underline{\sk}^\dag_n}
\def\tkopt{\tilde{\sk}^\dag_n}
\def\koracle{\sk^*_n}

\def\vfe{{\textup{VFE}_{n,\sk}}}
\def\Vfe{{\textup{VFE}}}
\def\mvfe{{\textup{mVFE}_n}}

\def\ebmvfe{{\textup{mVFE-EB}_n}}
\def\esavp{\widehat{Q}^{\esa}_n}
\def\esapst{\widehat{\Pi}^{\esa}_n}
\def\esaebpst{\check{\Pi}^{\esa}_n}
\def\esaebvp{\check{Q}^{\esa}_n}
\def\omegaesa{\widehat{\omega}_{n,\sk}^{\esa}}
\def\omegaebesa{\check{\omega}_{n,\sk}^{\esa}}
\def\omegafesa{\underline{\omega}_{n,\sk}^{\esa}}
\def\omegass{\tilde{\omega}_{n,\sk}^{\esa}}

\def\mPER{{\textup{mPER}_n}}
\def\esaerm{\htheta^{\esa}_n}
\def\esavbmean{\bar{\theta}^{\esa}_n}
\def\esass{\ttheta^{\esa}_{n}}
\def\psib{\psi^\bullet}
\def\Psib{\Psi^\bullet_{n,\sk}}
\def\renyihp{\cD_{1+\rho}^{\Pi_{n,\sk}}}
\def\thetab{\theta^\bullet}
\def\hthetab{\htheta^\bullet_{n,\sk-1,j}}
\def\Thetab{\Theta^\bullet_{n,\sk-1,j}}
\def\Thetabb{\Theta^\bullet_{n,\sk,j}}

\def\Ber{\texttt{\textup{Bernoulli}}}

\def\Unif{\texttt{\textup{Unif}}}

\def\N{\texttt{\textup{N}}}

\makeatletter
\def\greekvectors#1{%
 \@for\next:=#1\do{%
    \def\X##1;{\expandafter\def\csname b##1\endcsname{\bm{\csname##1\endcsname}}}
    \expandafter\X\next;}
 \@for\next:=#1\do{%
    \def\X##1;{\expandafter\def\csname h##1\endcsname{\widehat{\csname##1\endcsname}}}
    \expandafter\X\next;}
 \@for\next:=#1\do{%
    \def\X##1;{\expandafter\def\csname t##1\endcsname{\widetilde{\csname##1\endcsname}}}
    \expandafter\X\next;}
 \@for\next:=#1\do{%
    \def\X##1;{\expandafter\def\csname ba##1\endcsname{\bar{\csname##1\endcsname}}}
    \expandafter\X\next;}
 \@for\next:=#1\do{%
    \def\X##1;{\expandafter\def\csname c##1\endcsname{\check{\csname##1\endcsname}}}
    \expandafter\X\next;}
 \@for\next:=#1\do{%
    \def\X##1;{\expandafter\def\csname u##1\endcsname{\underline{\csname##1\endcsname}}}
    \expandafter\X\next;}
 \@for\next:=#1\do{%
    \def\X##1;{\expandafter\def\csname hb##1\endcsname{\widehat{\bm{\csname##1\endcsname}}}}
    \expandafter\X\next;}
}
\greekvectors{alpha,beta,gamma,delta,epsilon,zeta,theta,kappa,lambda,mu,nu,xi,pi,rho,tau,sigma,phi,chi,psi,omega,varepsilon,varpi,varphi,vartheta,  Delta,Gamma,Theta,Lambda,Xi,Pi,Sigma,Phi,Psi,Omega}
    
\@tfor\next:=abfghijlmnopqrstuwxyzABCDGHIJKLMOQSTUVWXYZ\do{%
    \def\command@factory#1{\expandafter\def\csname #1\endcsname{\mathbf{#1}} }
    \expandafter\command@factory\next
}
\@tfor\next:=abcdeghijklnopqrstuvwxyzABCDEFGHIJKLMNOQRSTUVWXYZ\do{%
    \def\command@factory#1{\expandafter\def\csname b#1\endcsname{\mathbbm{#1}} }
    \expandafter\command@factory\next
}
\@tfor\next:=abcdeghijklnpqrstuvwxyzABCDEFGHIJKLMNOPQRSTUVWXYZ\do{%
    \def\command@factory#1{\expandafter\def\csname t#1\endcsname{\texttt{\textup{#1}}} }
    \expandafter\command@factory\next
}
\@tfor\next:=dknoABCDEFGHIJKLMNOPQRSTUVWXYZ\do{%
    \def\command@factory#1{\expandafter\def\csname c#1\endcsname{\mathcal{#1}} }
    \expandafter\command@factory\next
}
\@tfor\next:=abdefghijklmnopqrstuvwxyzABCDEFGHIJKLMNOPQRSTUVWXYZ\do{%
    \def\command@factory#1{\expandafter\def\csname s#1\endcsname{\textsf{\textup{#1}}} }
    \expandafter\command@factory\next
}
\@tfor\next:=abcdefghjklmnopqrstuvwxyzABCDEFGHIJKLMNOPQRSTUVWXYZ\do{
    \def\command@factory#1{\expandafter\def\csname f#1\endcsname{\mathfrak{#1}} }
    \expandafter\command@factory\next
}
\@tfor\next:=abcdeghijklmnopqstuvwxyzABCDEFGHIJKLMNOPQRSTUVWXYZ\do{%
    \def\command@factory#1{\expandafter\def\csname ba#1\endcsname{\bar{#1}} }
    \expandafter\command@factory\next
}
\@tfor\next:=fxyzABCDEFGHIJKLMNOPQRSTUVWXYZ\do{%
    \def\command@factory#1{\expandafter\def\csname h#1\endcsname{\widehat{#1}} }
    \expandafter\command@factory\next
}
\@tfor\next:=abcdefghijklmnopqrstuvwxyzABCDEFGHIJKLMNOPQRSTUVWXYZ\do{%
    \def\command@factory#1{\expandafter\def\csname ti#1\endcsname{\widetilde{#1}} }
    \expandafter\command@factory\next
}
\@tfor\next:=abcdefghjklmnopqrstuvwxyzABCDEFGHIJKLMNOPQRSTUVWXYZ\do{
    \def\command@factory#1{\expandafter\def\csname u#1\endcsname{\underline{#1}} }
    \expandafter\command@factory\next
}
\@tfor\next:=ABCDEFGHIJKLMNOPQRSTUVWXYZ\do{%
    \def\command@factory#1{\expandafter\def\csname hc#1\endcsname{\widehat{\mathcal{#1}}} }
    \expandafter\command@factory\next
}
\@tfor\next:=ABCDEFGHIJKLMNOPQRSTUVWXYZ\do{%
    \def\command@factory#1{\expandafter\def\csname uc#1\endcsname{\underline{\mathcal{#1}}} }
    \expandafter\command@factory\next
}